\documentclass[11pt]{amsart}
\usepackage{amsmath}
\usepackage{amssymb}
\usepackage{mathrsfs}
\usepackage{pst-node}
\usepackage{tikz-cd}
\usepackage{comment}
\usetikzlibrary{shapes.geometric}
\usepackage{enumerate}
\usepackage{amsthm}
\usepackage{enumitem}
\usepackage{fullpage}
\usepackage{euscript}
\usepackage{mathtools}
\usepackage{appendix}
\usepackage{MnSymbol}
\usetikzlibrary{calc}
\usepackage{marginnote}
\usepackage[breaklinks=true]{hyperref}
\usepackage{xurl}

\newtheoremstyle{introthms}
	{}{}{\itshape}{}{\bfseries }{}{ }
	{\thmname{#1} \thmnumber{#2}. \thmnote{\bfseries{(#3)}}}

\newtheoremstyle{introcors}
	{}{}{\itshape}{}{\bfseries }{}{ }
	{\thmname{#1}. \thmnote{\bfseries{(#3)}}}
	
\theoremstyle{introthms}
\newtheorem{introthm}{Theorem}

\theoremstyle{introcors}
\newtheorem{introcor}{Corollary}

\numberwithin{equation}{section}

\DeclareMathAlphabet{\mathbbm}{U}{bbm}{m}{n}

\usepackage[T1]{fontenc}

\usepackage{mathrsfs}
\DeclareMathAlphabet{\mathpzc}{OT1}{pzc}{m}{it}

\usepackage{hyperref}

\hypersetup{
    colorlinks=true,
    citecolor=red,
    linktoc=all,
    linkcolor=black,
    backref=page,
}

\let\amsamp=&

\usepackage[margin=1in,footskip=.85in]{geometry}

\usepackage[
backend=biber,
style=alphabetic,
sorting=nyt,maxbibnames=9,backref
]{biblatex}
\usetikzlibrary{decorations.markings}
\tikzset{
  closed/.style = {decoration = {markings, mark = at position 0.5 with { \node[transform shape, xscale = .8, yscale=.4] {/}; } }, postaction = {decorate} },
  open/.style = {decoration = {markings, mark = at position 0.5 with { \node[transform shape, scale = .7] {$\circ$}; } }, postaction = {decorate} }
}

\theoremstyle{plain}
\newtheorem{theorem}{Theorem}[subsection]
\newtheorem{proposition}[theorem]{Proposition}
\newtheorem{lemma}[theorem]{Lemma}

\newtheorem{corollary}[theorem]{Corollary}
\theoremstyle{definition}
\newtheorem{definition}[theorem]{Definition}
\newtheorem{example}[theorem]{Example}
\newtheorem{remark}[theorem]{Remark}
\newtheorem{warning}[theorem]{Warning}
\newtheorem{recollection}[theorem]{Recollection}
\newtheorem{construction}[theorem]{Construction}
\theoremstyle{remark}

\newtheorem{notation}[theorem]{Notation}

\DeclareMathOperator\Hom{Hom}
\DeclareMathOperator\Map{\mathrm{Map}}
\DeclareMathOperator\spec{\mathrm{Spec}}
\DeclareMathOperator\map{\mathrm{map}}
\DeclareMathOperator\twar{\mathrm{TwAr}}
\DeclareMathOperator\cob{\mathrm{Cob}}
\DeclareMathOperator\op{\mathrm{op}}
\DeclareMathOperator{\qdb}{Q^{\Delta^\bullet}}
\DeclareMathOperator{\tqdb}{Q_\mathrm{ch}^{\Delta^\bullet}}
\DeclareMathOperator\hzmod{Mod_{H\kern-1.5pt\zz}}
\DeclareMathOperator\whzmod{Mod_{H\kern-1.5pt\zz, \ge 0}}
\DeclareMathOperator\hzmodw{Mod_{H\kern-1.5pt\zz, \ge 0}}
\DeclareMathOperator\wquad{w\kern-1pt\mathcal{Q}\mathrm{uad}}
\DeclareMathOperator\hz{H\kern-1.5pt\zz}
\DeclareMathOperator{\ev}{ev}

\DeclareMathOperator\wformcat{w\kern-.5pt\cal{F}\mathrm{orm}\kern-1.5pt\cat}
\DeclareMathOperator\wcompformcat{w\kern-.5pt\cal{C}\mathrm{omp}\cal{F}\mathrm{orm}\kern-1.5pt\cat}
\DeclareMathOperator\formcat{\kern-.5pt\cal{F}\mathrm{orm}\kern-1.5pt\cat}
\DeclareMathOperator\relcat{\cal{R}\mathrm{el}\kern-1.75pt\cat}
\DeclareMathOperator\excat{\cal{E}\mathrm{x}\kern-1.75pt\cat}
\DeclareMathOperator\wexcat{w\kern-1pt\cal{E}\mathrm{x}\kern-1.75pt\cat}
\DeclareMathOperator\wcompcat{w\cal{C}\mathrm{omp}\cal{E}\mathrm{x}\kern-1.75pt\cat}
\DeclareMathOperator\frobcat{\cal{F}\mathrm{rob}\cal{E}\mathrm{x}\kern-1.75pt\cat}
\DeclareMathOperator\compcat{\cal{C}\mathrm{omp}\cal{E}\mathrm{x}\kern-1.75pt\cat}
\DeclareMathOperator\chb{\mathrm{Ch}_b}
\DeclareMathOperator\db{\mathrm{D}_b}
\DeclareMathOperator\kb{\mathrm{K}_b}
\DeclareMathOperator\dperf{\mathrm{D}^p}
\DeclareMathOperator\projz{\mathrm{Proj^{fg}}_{\zz}}

\DeclareMathOperator\wchb{\mathrm{Ch}_{b, \le 0}}
\DeclareMathOperator\acb{\mathrm{Ac}_b}
\DeclareMathOperator{\zz}{\mathbb{Z}}
\DeclareMathOperator{\D}{\mathbb{D}}
\DeclareMathOperator\hct{hC_2}
\DeclareMathOperator\fct{C_2}
\DeclareMathOperator\tct{tC_2}
\DeclareMathOperator{\hocolim}{hocolim}

\DeclareMathOperator{\coker}{coker}
\DeclareMathOperator{\arr}{\mathrm{Ar}}
\DeclareMathOperator{\colim}{colim}
\DeclareMathOperator\elex{\cal{E}^\mathrm{lex}}
\DeclareMathOperator{\epi}{\twoheadrightarrow}
\DeclareMathOperator{\mono}{\rightarrowtail}
\DeclareMathOperator\fun{\mathrm{Fun}}
\DeclareMathOperator{\sset}{s\cal{S}\kern -.5pt et}
\DeclareMathOperator{\cat}{\cal{C}\kern -.5pt at}
\DeclareMathOperator{\Ab}{\cal{A}\kern -.5pt \text{b}}
\DeclareMathOperator\sab{s\cal{A}\kern -.5pt b}
\DeclareMathOperator{\id}{id}
\DeclareMathOperator{\Sp}{\cal{S}\kern-.5pt p}
\DeclareMathOperator\stab{Stab}
\DeclareMathOperator{\Kan}{\cal{K}\kern -.5pt an}
\DeclareMathOperator{\set}{\cal{S}\kern -.5pt et}
\DeclareMathOperator\fib{\mathrm{fib}}
\DeclareMathOperator\cofib{\mathrm{cofib}}
\DeclareMathOperator{\sh}{\cal{S}\kern -.5pt h}
\newcommand{\cal}[1]{\EuScript{#1}}
\newcommand{\xto}[1]{\xrightarrow{#1}}

\DeclareMathOperator\hofib{\mathrm{hofib}}

\usepackage[bbgreekl]{mathbbol}
\DeclareMathSymbol\bbDelta  \mathord{bbold}{"01}

\DeclareFontFamily{U}{cbgreek}{}
\DeclareFontShape{U}{cbgreek}{m}{n}{
        <-6>    grmn0500
        <6-7>   grmn0600
        <7-8>   grmn0700
        <8-9>   grmn0800
        <9-10>  grmn0900
        <10-12> grmn1000
        <12-17> grmn1200
        <17->   grmn1728
      }{}
\DeclareFontShape{U}{cbgreek}{bx}{n}{
        <-6>    grxn0500
        <6-7>   grxn0600
        <7-8>   grxn0700
        <8-9>   grxn0800
        <9-10>  grxn0900
        <10-12> grxn1000
        <12-17> grxn1200
        <17->   grxn1728
      }{}

\DeclareRobustCommand{\Qoppa}{%
  \text{\usefont{U}{cbgreek}{\normalorbold}{n}\symbol{21}}%
}

\makeatletter
\newcommand{\normalorbold}{%
  \ifnum\pdf@strcmp{\math@version}{bold}=\z@ bx\else m\fi
}
\makeatother

\title{Higher $K$-theory of forms III: from chain complexes to derived categories}
\author{Daniel Marlowe and Marco Schlichting}

\begin{document}
\begin{abstract}
    We exhibit a canonical equivalence between the hermitian $K$-theory (alias Grothendieck-Witt) spectrum of an exact form category and that of its derived Poincar\'e $\infty$-category, with no assumptions on the invertibility of $2$. Along the way, we obtain a model for the nonabelian derived functor of a nondegenerate quadratic functor on an exact category.
\end{abstract}
\maketitle
\tableofcontents
\section{Introduction}
Hermitian $K$-theory is at heart the study of the stable algebra of projective modules equipped with a generalised notion of quadratic form. It is the algebraic analogue of Real topological $K$-theory, and a quadratic refinement of algebraic $K$-theory. The construction and properties of higher hermitian $K$-groups for rings were initially laid out by Karoubi \cite{Kar73}, culminating in particular in the proof of the celebrated \textit{fundamental theorem} \cite{Kar80} relating hermitian and ordinary algebraic $K$-theory. Through work of Giffen, Hornbostel \cite{Hor05}, and the second author \cite{HS04}, \cite{Sch10a}, \cite{Sch10b}, among others, these results were generalised to the setting of exact categories with duality, and later to dg-categories with duality and uniquely 2-divisible mapping complexes \cite{Sch17}. As the latter restriction suggests, the prime 2 is of particular importance in governing the behaviour of hermitian $K$-theory: given a ring $R$ in which 2 is a unit, the map associating to a quadratic form on a finitely generated projective $R$-module $P$ its bilinear polarisation furnishes an isomorphism between the sets of quadratic and symmetric bilinear forms on $P$, but this fails in general, as does the invariance of hermitian $K$-theory under a na\"ive version of derived equivalence (see \cite[Prop.\ 2.1]{Sch17}). The foundations of hermitian $K$-theory were revisited accordingly for more general notions of form and with no assumption on the invertibility of 2 in \cite{Sch21}, \cite{Sch24}, and more recently, in pioneering work \cite{CD23a}, \cite{CD23b}, \cite{CD21} of Calm\`es-Dotto-Harpaz-Hebestreit-Land-Moi-Nardin-Nikolaus-Steimle building on Lurie's formalism of Poincar\'e $\infty$-categories \cite{Lur11}.\\
The Poincar\'e formalism has proved hugely fruitful in developing a framework for quadratic invariants such as hermitian $K$-theory, $L$-theory, and Real algebraic $K$-theory; importantly, in this setting the hermitian $K$-functor $\mathrm{GW}$ enjoys a universal property as the initial grouplike additive spectrum-valued functor on the $\infty$-category $\cat^p_\infty$ of Poincar\'e categories, whose underlying infinite loopspace receives a natural transformation from $\mathrm{Pn}:\cat^p_\infty \to \cal{S}$, where here $\mathrm{Pn}(\cal{C}, \Qoppa)$ is in an appropriate sense the moduli space of nondegenerate forms on a Poincar\'e category $(\cal{C}, \Qoppa)$, playing a role analogous to that of the core groupoid for algebraic $K$-theory.\\
As an invariant of algebraic varieties, hermitian $K$-theory has come to occupy a celebrated place in motivic homotopy theory \cite{MV99}: fundamental work of Morel \cite{Mor12} (see also \cite{RSO19}) identifies $\mathrm{GW}_0(k)$ with the endomorphisms ring $\mathrm{End}_{\mathrm{SH}(k)}(\mathbb{S}_k)$ of the motivic sphere spectrum over a perfect field $k$ of characteristic not equal to $2$. As such, $\mathrm{GW}_0(k)$ plays a role in motivic homotopy theory over $\spec(k)$ analogous to that of the integers in usual homotopy theory, and is the universal receptacle for $\mathbb{A}^1$-Euler characteristics \cite{ML19} and Brouwer degrees \cite{Bra23}. Computations of hermitian $K$-groups have been fundamental in the Asok-Fasel vector bundle classification program \cite{AF14}, \cite{AF23}, allowing the importation of classical obstruction-theoretic techniques into the motivic setting, and more recently in work of Sch\"appi \cite{Scha24} providing a counterexample to the Hermite ring conjecture. As early as the '70s, hermitian $K$-theory computations were used by Karoubi to give a lower bound on the order of $K_3(\zz) \cong \zz\kern-1.5pt/48\kern-1.5pt\zz$, refuting (and refining) a conjecture of Lichtenbaum. By work of Hornbostel when 2 is invertible, and Calmès-Harpaz-Nardin \cite{CHN24} in the general setting, (homotopy) symmetric hermitian $K$-theory (over some regular noetherian case scheme $S$) satisfies Nisnevich descent, $\mathbb{A}^1$-invariance, and a projective bundle formula, and is accordingly represented by a motivic spectrum $\mathrm{KQ}_S$ . \\
The development of the hermitian formalism largely follows the standard course of $K$-theoretic invariants: one studies initially the zeroth Grothendieck-Witt group $\mathrm{GW}_0(k)$ attached to a field (and later commutative ring with involution) $k$, obtained as the group-completion of a certain abelian monoid of isometry classes of forms under orthogonal sum, and via a plus-construction style definition extends this to higher groups amenable to study by $K$-theoretic techniques. As a categorified invariant, one then notes that in an appropriate sense hermitian $K$-theory depends only upon the derived $\infty$-category, and as such enjoys a universal property most naturally phrased in the higher categorical language. Successive generalisations of the $K$-theoretic apparatus have been accompanied by comparison results, namely Quillen's `$+ =\mathcal{Q}$' theorem \cite{Gra76}, the Gillet-Waldhausen theorem \cite[Th.\ 1.11.7]{TT90}, and in the higher categorical setting \cite[Cor.\ 10.10, 10.16]{Bar13} and \cite[Cor.\ 7.12]{BGT13}. This note addresses the lack of a comparison in full generality between the Quillen-Waldhausen-style hermitian $K$-theory of exact categories \cite{Sch21} and its higher categorical analogue \cite{CD23a}. The main result is the following (Theorem \ref{main-comp-sp} in the main text).
\begin{introthm}\label{thma}
	For $(\cal{E}, Q, \mathbb{D}, \eta, w)$ a complicial exact form category with weak equivalences, the $\infty$-categorical localisation $\cal{E} \to L_w(\cal{E})$ induces a functorial equivalence of hermitian $K$-theory spectra
	\[
		\mathrm{GW}(\cal{E}, Q, w) \xto{\simeq} \mathrm{GW}(L_w(\cal{E}), \mathbf{R}Q),
	\]
	where the left- and right-hand sides denote respectively the form categorical and Poincar\'e-categorical hermitian $K$-theory spectra.
\end{introthm}
As a consequence, we obtain the following, generalising \cite[Cor.\ B.2.5]{CD23b} and the main result of \cite{HS23}.
\begin{introthm}\label{thmb}
	For $(\cal{E}, \D, \eta, Q)$ an exact form category, the canonical embedding $\cal{E} \hookrightarrow \db(\cal{E})$ into the bounded derived $\infty$-category induces an equivalence of spectra
	\[
		\mathrm{GW}(\cal{E}, Q) \xto{\simeq} \mathrm{GW}(\db(\cal{E}), \Qoppa),
	\]
	for $(\db(\cal{E}),\Qoppa)$ the derived Poincar\'e category associated to $(\cal{E}, Q)$.
\end{introthm}
As a corollary of this, we exhibit a canonical equivalence between the 1-categorical symmetric Grothendieck-Witt theory of an exact category with duality $\cal{E}$ and the genuine symmetric Grothendieck-Witt theory of the derived $\infty$-category $\db(\cal{E})$. This may we well-known to experts, but to our knowledge has been recorded only in the split-exact case.
\begin{introcor}\label{thmc}
	Let $(\cal{E}, \mathbb{D}, \eta)$ be an exact category with strong duality. Then the derived Poincar\'e category of the symmetric form category $(\cal{E}, Q^s)$ canonically identifies with the genuine symmetric structure $(\db(\cal{E}), \Qoppa^\mathrm{gs})$, and the embedding $\cal{E} \hookrightarrow \db(\cal{E})$ induces an equivalence of hermitian $K$-theory spectra
	\[
		\mathrm{GW}(\cal{E}, Q^s) \to \mathrm{GW}(\db(\cal{E}), \Qoppa^\mathrm{gs}).
	\]
\end{introcor}
\subsubsection{Relation to other work}
Partial comparison results exist in the literature in the setting of additive form categories \cite{HS23}, and in the symmetric forms case for $2$ invertible \cite[App.\ B]{CD23b}.
The former comparison follows from a treatment of algebraic surgery on cobordism categories, incorporating in an essential manner the notion of a weight structure on a stable $\infty$-category. The corresponding theorem of the weighted heart \cite[Cor.\ A.2.8]{HS23} states that for a Poincar\'e category $(\cal{C}, \Qoppa)$ equipped with an exhaustive weight structure of dimension 0, the inclusion of the heart (an additive form category in the sense of \cite{Sch21}) induces an equivalence of Grothendieck-Witt spaces:
\begin{equation}\label{harpcomp}
	\cal{GW}^\oplus(\cal{C}^\mathrm{ht}, \Qoppa\mid_{\cal{C}^\mathrm{ht}}) \xto{\simeq} \cal{GW}(\cal{C}, \Qoppa),
\end{equation}
where here $\cal{GW}^\oplus(\cal{A}, Q)$ denotes the direct-sum Grothendieck-Witt space of an exact form category $(\cal{A}, Q)$, which by the main result of \cite{Sch21} is homotopy equivalent to $\cal{GW}(\cal{A}, Q)$ for $\cal{A}$  split exact.\\
That these techniques are limited to split-exact categories follows from the observation that an exact sequence in the heart of a weight structure necessarily splits. It is worth mentioning that the additive comparison (\ref{harpcomp}) suffices to compare, for instance, the Grothendieck-Witt spaces associated to affine schemes, and via Zariski descent to divisorial schemes \cite[Prop.\ 4.6.1]{CHN24}.\\
Our argument proceeds instead along the lines of \cite[App.\ B]{CD23b}, starting with a complicial exact form category with weak equivalences $(\cal{E}, Q, w, \D, \eta)$ in the sense of \cite{Sch21} and deriving the relevant structures to obtain a Poincar\'e category $(L_w(\cal{E}), \mathbf{R}Q)$ in the sense of \cite{CD23a}. In doing so, we obtain an explicit description of the underlying infinite loopspace of the nonabelian derived functor associated with the quadratic functor $Q$, upon which the comparison theorem hinges. The stipulation that $\cal{E}$ be complicial is essentially a homotopical soundness condition, ensuring that $\cal{E}$ is an $\infty$-category of fibrant objects in the sense of \cite{Cis19}; any exact category admits a fully faithful exact functor into a complicial exact category with weak equivalences under which the hermitian $K$-theory space is invariant, cf.\ \S\ref{excat-to-chain-complexes}.
\subsubsection{A note on organisation}
We give a brief summary of the contents of the current paper: \S\ref{from-additivity-to-homotopy-coherence} is concerned with the study of nonabelian derived functors on complicial exact categories, serving as a technical foundation for the comparison of Grothendieck-Witt spaces in \S\ref{comp}. In \S\ref{gwsp}, we upgrade the comparison to the level of Grothendieck-Witt spectra in the sense of \cite[\S11]{Sch24} and \cite[\S4.2]{CD23b}. In \S\ref{genuine-structures}, we give a brief treatment of genuine Poincar\'e structures, showing that the definition in \cite{CHN24} recovers precisely the derived Poincar\'e structures of exact form categories. For the reader unfamiliar with complicial exact categories, we include a detailed survey in Appendix \ref{appa}.
\subsection{Notation and conventions}
This paper presents a recipe for the extraction of homotopy-coherent data from algebraic structures. To give a uniform treatment of the various homotopical ideas, we make use of the quasi-categorical model of $(\infty,1)$-categories as developed by Joyal \cite{Joy08} and Lurie \cite{HTT}, \cite{HA}. We shall refer to quasi-categories as $\infty$-categories, drawing a distinction between these and ordinary categories when necessary (implicitly identifying an ordinary category with its nerve throughout). We refer to higher-categorical (co)limits as (co)limits, using the term homotopy (co)limit when we make explicit reference to model-categorical models for these. \textit{Essentially unique} should be read as \textit{unique up to a contractible space of choices}. For $\infty$-categories $\cal{C}, \cal{D}, \cal{E}$ and functors $i:\cal{C} \to \cal{D}$, $f:\cal{C} \to \cal{E}$, we denote by $i_!f$ the left Kan extension of $f$ along $i$, when this exists. For a functor $f: \cal{C} \to \cal{D}$ of $\infty$-categories, we write the $\infty$-categorical slices $\cal{D}_{/f}$ and $\cal{D}_{f/}$ as $(\cal{D} \downarrow f)$ and $(f \downarrow \cal{D})$, noting that when $\cal{C}, \cal{D}$ are ordinary categories, these coincide with the nerve of the 1-categorical slice categories by \cite[Rem.\ 1.2.9.6]{HTT}. More prosaically, we grade our chain complexes homologically.
\subsection{Acknowledgements}
The first author would like to thank Emanuele Dotto for sharing his expertise in hermitian $K$-theory, and Martin Gallauer and Victor Saunier for several fruitful exchanges.
\section{From additivity to homotopy coherence}\label{from-additivity-to-homotopy-coherence}
The 1-categorical setting of Grothendieck-Witt theory \cite{Sch21} is that of an exact category $\cal{E}$ with duality, equipped with a quadratic abelian group-valued functor generalising the classical notion of quadratic, symmetric, hermitian forms on objects of $\cal{E}$. The recent treatment \cite{CD23a}, \cite{CD23b} \cite{CD21} is a homotopy-coherent incarnation of these ideas, comprising the data of a stable $\infty$-category equipped with a nondegenerate quadratic functor taking values in spectra.\\
We begin by making some recollections on the formalism of \cite{Sch21}, \cite{Sch24}, and \cite{CD23a}-\cite{CD21}. In \S\ref{dquad} we compute the (right) derived functor of a quadratic left-exact functor associated to an exact form category with weak equivalences $(\cal{E}, Q, w, \D, \eta)$ in the sense of Schlichting \cite{Sch21}, arriving in \S\kern-1pt\S\ref{dadd}-\ref{poistr} at a canonical Poincar\'e $\infty$-category $(L_w(\cal{E}), \mathbf{R}Q)$ in the sense of \cite{CD23a}.
\subsection{Recollections on form categories}\label{rec}
We begin by making the following recollections from \cite[\S\kern-1pt\S2.1-2.3]{Sch10a}.
\begin{definition}\label{exdual}
	An exact category with weak equivalences and duality $(\cal{E}, \D, \eta, w)$ is the data of a Quillen exact category \cite{Qui73}, equipped with the following data:
	\begin{enumerate}[label=(\roman*)]
		\item A wide subcategory $w\cal{E} \subset \cal{E}$ of \textit{weak equivalences}, satisfying 2-of-3, and closed under isomorphism and retracts in $\mathrm{Ar}(\cal{E})$, and under pushouts along inflations and pullbacks along deflations.
		\item An exact functor $\D:\cal{E}^{\op} \to \cal{E}$, the \textit{duality}, which preserves weak equivalences, and a natural transformation $\eta:1_{\cal{E}} \to \D\D^{\op}$, the \textit{double dual identification}, such that for each $x \in \cal{E}$, $\D(\eta_x)\circ\eta_{\D(x)} = 1_{\D(x)}$. 
	\end{enumerate}
    The duality is said to be \textit{strong} if $\eta$ is a natural weak equivalence, and \textit{strict} if it is the identity transformation. If $w$ is the class of isomorphisms, we simply say $(\cal{E}, \D, \eta)$ is an \textit{exact category with duality}. Moreover, in the case $\cal{E}$ is an additive category equipped with the split-exact structure, we say that $(\cal{E}, \D, \eta)$ is an additive category with duality.\\
    A \textit{form functor} $(F, \varphi):(\cal{E}, \D, \eta, w) \to (\cal{E}', \D', \eta', w')$ between exact categories with weak equivalences and duality is the data of an exact functor $F:\cal{E} \to \cal{E}'$ (preserving conflations and weak equivalences), and a \textit{duality compatibility transformation} $\varphi: F\circ\D \Rightarrow \D'\circ F^{\op}$ such that the diagram
	\[\begin{tikzcd}
		F(x) \ar[r, "\eta'_{F(x)}"] \ar[d, "F(\eta_x)"] & (\D')^2F(x) \ar[d, "\D'(\varphi_x)"] \\
		F(\D^2(x)) \ar[r, "\varphi_{\D(x)}"] & \D'F(\D(x))
	\end{tikzcd}\]
	commutes for each $x \in \cal{E}$. The pair $(F, \varphi)$ is said to be \textit{nonsingular} if $\varphi$ is a natural weak equivalence.
\end{definition}
\begin{remark}
    Call an object $x \in \cal{E}$ \textit{acyclic} if the unique map $0 \to x$ is in $w\cal{E}$. Then for $(\cal{E}, w)$ an exact category with weak equivalence, by \cite[Lem.\ 7.1]{Sch24}, an inflation $x \mono y$ in $\cal{E}$ is a weak equivalence if and only if its cokernel is acyclic, and dually a deflation $x \epi y$ is a weak equivalence if and only if its kernel is acyclic. Accordingly, we see that $\cal{E}^w \subset \cal{E}$ is an exact subcategory closed under retracts in $\cal{E}$.
\end{remark}
For $(\cal{E}, \D, \eta)$ an exact category with duality and objects $x, y \in \cal{E}$, there is a natural isomorphism
\[
	\sigma_{x,y}:\Hom_{\cal{E}}(x, \D(y)) \to \Hom_{\cal{E}}(y, \D(x)), \quad f \mapsto \D(f)\circ\eta_y
\]
satisfying $\sigma_{y, x}\sigma_{x, y}(f) = f$. In particular, the mapping groups $\Hom_{\cal{E}}(x, \D(x))$ carry a canonical $C_2$-action, a fixed point for which is said to be a \textit{symmetric form} on $x$. The assignment $x \mapsto \Hom_{\cal{E}}(x, \D(x))^{\fct}$ assembles into a quadratic left-exact functor, the \textit{symmetric forms functor}, which equips any exact category with duality with the canonical structure of an \textit{exact form category}. The following is \cite[Def.\ 7.2]{Sch24}.
\begin{definition}\label{addq}
	An \textit{exact form category with weak equivalences} is the data of an exact category with weak equivalences\footnote{If $w$ is the class of isomorphisms, we simply refer to the tuple $(\cal{E}, Q, \D, \eta)$ as an \textit{exact form category}, and if moreover $\cal{E}$ is an additive category equipped with the split-exact structure, we refer to $(\cal{E}, Q, \D, \eta)$ as an \textit{additive form category}.} and duality $(\cal{E}, \D, \eta, w)$ equipped with a functor $Q:\cal{E}^{\op} \to \Ab$ and natural transformations $\tau, \rho$ of functors $\cal{E}^{\op} \to \Ab$,
	\[
		\Hom_{\cal{E}}(-, \D(-)) \xto{\tau} Q(-) \xto{\rho} \Hom_{\cal{E}}(-, \D(-)),
	\]
	satisfying the following conditions:
	\begin{enumerate}[label=(\roman*)]
		\item \label{c2mack} For each $x \in \cal{E}$, the diagrams
		\[
			\Hom_{\cal{E}}(x, \D(x)) \xto{\tau_x} Q(x) \xto{\rho_x} \Hom_{\cal{E}}(x, \D(x))
		\]
		are $C_2$-equivariant, and satisfy $\rho_x\tau_x = 1 + \sigma_{x, x}$.
		\item For each $f, g \in \Hom_{\cal{E}}(x, y)$ and $\xi \in Q(y)$,
		\[
			(f+g)^\bullet(\xi) = f^\bullet(\xi) + g^\bullet(\xi) + \tau_x(\D(g)\rho_y(\xi)f),
		\]
		where we write $f^\bullet := Q(f)$.
		\item $Q$ is additionally \textit{quadratic left exact}: for each conflation $x \stackrel{i}\mono y \stackrel{p}\epi z$ in $\cal{E}$, the sequence
		\[
			0 \to Q(z) \xto{p^\bullet} Q(y) \xto{(i^\bullet, \D(i)_*\rho_y)} Q(x) \oplus \Hom_{\cal{E}}(y, \D(x))
		\]
		of abelian groups is exact.
	\end{enumerate}
\end{definition}
\begin{remark}
    For $\cal{E}$ an exact category and $Q:\cal{E}^{\op} \to \Ab$ a reduced functor, the \textit{polarisation} $B_Q$ is defined as the kernel
    \[
        B_Q(x, y) := \ker\left(Q(x \oplus y) \to Q(x) \oplus Q(y)\right).
    \]
    As per \cite[App.\ A]{Sch21} (see also \cite[Def.\ 2.4]{BGMN22}), $Q$ is said to be \textit{quadratic} if its polarisation preserves direct sums in either variable; such a pair $(\cal{E}, Q)$ then defines an exact form category if there is a natural isomorphism $\theta_{x, y}:B_Q(x, y) \cong \Hom_{\cal{E}}(x, \D(y))$ of functors $\cal{E}^{\op} \times \cal{E}^{\op} \to \Ab$, for $\D:\cal{E}^{\op} \to \cal{E}$ a duality determined up to isomorphism by $B_Q$, and additionally $Q$ is quadratic left exact. The maps $\tau, \rho$ are determined in this setup by $\theta$, as the composites
    \[
        \Hom_{\cal{E}}(x, \D(x)) \xto{\theta_{x,x}^{-1}} B_Q(x, x) \hookrightarrow Q(x \oplus x) \xto{\Delta^\bullet_x} Q(x)
    \]
    and
    \[
        Q(x) \xto{\nabla_x^\bullet} Q(x \oplus x) \epi B_Q(x, x) \xto{\theta_{x, x}} \Hom_{\cal{E}}(x, \D(x)).
    \]
\end{remark}
\begin{remark}
	If $(\cal{E}, Q, \D, \eta)$ is an additive form category and $x \stackrel{i}\mono y \stackrel{p}\epi z$ is a split exact sequence in $\cal{E}$, then under the first two hypotheses of Definition \ref{addq}, the sequence
\[
	0 \to Q(z) \xto{Q(p)} Q(y) \xto{(Q(i), -\D(i)_*\rho_y)} Q(x)\oplus \Hom_{\cal{E}}(y, \D(x)) \xto{(\rho_x, i^*)} \Hom_{\cal{E}}(x, \D(x)) \to 0
\]
is exact, i.e.\ $Q(z) \to Q(y)$ exhibits $Q(z)$ as the total kernel in abelian groups of the diagram
\[\begin{tikzcd}
	\
	Q(y) \ar[r] \ar[d] & Q(x) \ar[d] \\
	B_Q(y, x) \ar[r] & B_Q(x, x);
\end{tikzcd}\]
see \cite[Lem.\ A.12]{Sch21}.
\end{remark}
\begin{example}\label{bq}
	As remarked above, for $(\cal{E}, \D, \eta)$ an exact category with duality, the functor
	\[
		Q^s:\cal{E}^{\op} \to \Ab, \quad x \mapsto \Hom_{\cal{E}}(x, \D(x))^{\fct}
	\]
	equips $\cal{E}$ with the structure of an exact form category (see \cite[Ex.\ 2.24]{Sch21}): the polarisation $B_{Q^s}$ is given by
	\begin{align*}
		& \ker\left[\Hom_{\cal{E}}(x \oplus y, \D(x \oplus y))^{\fct} \to \Hom_{\cal{E}}(x, \D(x))^{\fct} \oplus \Hom_{\cal{E}}(y, \D(y))^{\fct}\right] \\ \cong & \ker\left[\Hom_{\cal{E}}(x \oplus y, \D(x \oplus y)) \to \Hom_{\cal{E}}(x, \D(x)) \oplus \Hom_{\cal{E}}(y, \D(y))\right]^{\fct} \\
		\cong & \left[\Hom_{\cal{E}}(x, \D(y)) \oplus \Hom_{\cal{E}}(y, \D(x))\right]^{\fct} \cong \Hom_{\cal{E}}(x, \D(y)).
	\end{align*}
        There is an analogous construction of the so-called \textit{quadratic forms functor} $Q^q$ on $\cal{E}$, given by taking $C_2$-orbits (and using that $B_Q(x, y)$ is equivalently the cokernel of the split inclusion $Q(x) \oplus Q(y) \to Q(x \oplus y)$). The quadratic forms functor is quadratic, but in general fails to be quadratic left exact if the exact structure on $\cal{E}$ is not split.
\end{example}
A \textit{form functor} between exact form categories with weak equivalences and duality
\[
	(F, \varphi_q, \varphi):(\cal{A}, Q_{\cal{A}}, \D_{\cal{A}}, \eta_{\cal{A}}, w_{\cal{A}}) \to (\cal{B}, Q_{\cal{B}}, \D_{\cal{B}}, \eta_{\cal{B}}, w_{\cal{B}})
\]
is the data of a form functor $(F, \varphi):(\cal{A}, \D_{\cal{A}}, \eta_{\cal{A}}, w_{\cal{A}}) \to (\cal{B}, \D_{\cal{B}}, \eta_{\cal{B}}, w_{\cal{B}})$ of exact categories with weak equivalences and duality in the sense of Definition \ref{exdual}, and $\varphi_q:Q_{\cal{A}} \to Q_{\cal{B}} \circ F^{\op}$ a natural transformation such that for each $x \in \cal{A}$, the diagram
\begin{equation}\begin{tikzcd}\label{phi_x}
	\Hom_{\cal{A}}(x, \D_{\cal{A}}(x)) \ar[r, "\tau_x"] \ar[d, "f \mapsto \varphi_xF(f)"'] & Q_{\cal{A}}(x) \ar[d, "\varphi_{q, x}"] \ar[r, "\rho_x"] & \Hom_{\cal{A}}(x, \D_{\cal{A}}(x)) \ar[d, "f \mapsto \varphi_xF(f)"] \\
	\Hom_{\cal{B}}(F(x), \D_{\cal{B}}(F(x))) \ar[r, "\tau_{F(x)}"] & Q_{\cal{B}}(F(x)) \ar[r, "\rho_{F(x)}"] & \Hom_{\cal{B}}(F(x), \D_{\cal{B}}(F(x)))
\end{tikzcd}\end{equation}
commutes. If both $Q_{\cal{A}}$ and $Q_{\cal{B}}$ encode symmetric forms, a form functor reduces to the notion of a form functor between exact categories with weak equivalences and duality in the sense of Definition \ref{exdual}.\\
Write $\wformcat$ for the category spanned by the exact form categories with weak equivalences and strong duality, and nonsingular form functors between them, and $\formcat$ for the full subcategory of exact form categories with strong duality.
\begin{remark}\label{oversp}
	The data of a nonsingular form functor $(F, \varphi, \varphi_q):(\cal{E}, Q, \D, \eta) \to (\cal{E}', Q', \D', \eta')$ between form categories is overspecified in the sense that the duality compatibility is determined by the pair $(F, \varphi_q)$: there is a map
	\[
		\psi_{x, y}: \Hom_{\cal{E}}(x, \D(y)) \to \Hom_{\cal{E}'}(F(x), \D'F(y))
	\]
	induced by the naturality of $\varphi_q$ and the identification $\ker\left(Q(x \oplus y) \to Q(x) \oplus Q(y)\right) \cong \Hom_{\cal{E}}(x, \D(y))$. Setting $\varphi_x := \psi_{\D(x),x}\left(1_{\D(x)}\right)$ for each $x \in \cal{E}$, we see that for each map $f:x\to y$, the diagram
	\[\begin{tikzcd}
		\Hom_{\cal{E}}(\D(y), \D(y)) \ar[r, "\D(f)_*"] \ar[d, "\psi_{\D(y), y}"] & \Hom_{\cal{E}}(\D(y), \D(x)) \ar[d, "\psi_{\D(y), x}"] & \Hom_{\cal{E}}(\D(x), \D(x)) \ar[d, "\psi_{\D(x),x}"] \ar[l, "\D(f)^*"'] \\
		\Hom_{\cal{E}'}(F\D(y), \D'F(y)) \ar[r, "\D'F(f)_*"] & \Hom_{\cal{E}'}(F\D(y), \D'F(x)) & \Hom_{\cal{E}'}(F\D(x), \D'F(x)) \ar[l, "F\D(f)^*"']
	\end{tikzcd}\]
	witnesses $\D'F(f)\circ \varphi_y = \psi_{\D(y),x}(\D(f)) = \varphi_x \circ F\D(f)$, i.e.\ $\varphi$ furnishes a natural transformation $F\D \Rightarrow \D'F^{\op}$. Commutativity of the diagram (\ref{phi_x}) follows from that of
	\[\begin{tikzcd}
		\Hom_{\cal{E}}(x, \D(x)) \ar[r, "F"] \ar[rd, "\psi_{x, x}"'] & \Hom_{\cal{E}'}(F(x), F\D(x)) \ar[d, "\varphi_{x, *}"] \\
		& \Hom_{\cal{E}'}(F(x), \D'F(x)),
	\end{tikzcd}\]
	which in turn follows from the Yoneda lemma (the natural transformations $\varphi_{x, *} \circ F$ and $\psi_{-, x}$, $\Hom_{\cal{E}}(-, \D(x)) \Rightarrow \Hom_{\cal{E}'}(F(-),\D'F(x)))$ each send $1_{\D(x)}$ to $\varphi_x$). The condition that $\D'(\varphi_x)\circ \eta_{F(x)} = \varphi_{\D(x)}\circ F(\eta_x)$ of \cite[\S2]{Sch21} is equivalent to commutativity of the diagram
	\[\begin{tikzcd}
		\Hom_{\cal{E}}(\D(x), \D(y)) \ar[r, "\psi_{\D(x),y}"] \ar[d, "\rotatebox{270}{$\cong$}"] & \Hom_{\cal{E}'}(F\D(x), \D'F(y)) \ar[d, "\rotatebox{270}{$\cong$}"] \\
		\Hom_{\cal{E}}(y, \D^2(x)) \ar[r, "\psi_{y, \D(x)}"] & \Hom_{\cal{E}'}(F(y), \D'F\D(x)),
	\end{tikzcd}\]
	where the vertical isomorphisms come from self-adjointness of $\D$.
\end{remark}
\begin{remark}\label{formcat}
	It will be useful in the sequel to have another construction of $\wformcat$. Write $\wexcat$ for the ordinary category of small exact categories with weak equivalences and exact functors between them, and $\mathrm{CAT}_1$ for the (large) category of small categories. Consider the functor
	\[
		\wexcat^{\op} \to \mathrm{CAT}_1, \quad \cal{E} \mapsto \fun^\mathrm{q}(\cal{E}),
	\]
	where $\fun^\mathrm{q}(\cal{E}) \subset \fun(\cal{E}^{\op}, \Ab)$ denotes the full subcategory spanned by quadratic functors in the sense of \cite[App.\ A]{Sch21}. The unstraightening $\cat^{h, w}_1 := \int_{\wexcat}\fun^\mathrm{q}$ has objects pairs $(\cal{E}, Q, w)$ for $(\cal{E}, w)$ an exact category with weak equivalences and $Q$ a quadratic functor on $\cal{E}$, and maps $(F, \varphi_q):(\cal{E}, Q, w) \to (\cal{E}', Q', w')$ the data of an exact functor $F:(\cal{E}, w) \to (\cal{E}', w')$ along with a natural transformation $\varphi_q:Q \to Q'\circ F^{\op}$. Define $\wformcat \subset \cat^{h, w}_1$ to be the (non-full) subcategory consisting of tuples $(\cal{E}, Q, w)$ and maps $(F, \varphi_q)$ which together satisfy the following:
	\begin{enumerate}[label=(\roman*)]
		\item The polarisation $B_Q(-,-)$ is naturally isomorphic to $\Hom_{\cal{E}}(-, \D(-))$ for some duality $(\D, \eta)$ on $\cal{E}$, and the associated double dual identification $1 \to \D\circ\D^{\op}$ is a natural weak equivalence.
		\item $Q$ is quadratic left exact.
		\item The natural transformation $\varphi_q:Q \to Q'\circ F^{\op}$ induces a duality compatibility $\varphi$ as in Remark \ref{oversp} which is a natural weak equivalence, and such that for each $x \in \cal{E}$, the diagram
		\[\begin{tikzcd}
			F(x) \ar[r, "\eta_{f(x)}"] \ar[d, "f(\eta_x)"] & \D_{Q'}^2(F(x)) \ar[d, "\D_{Q'}(\varphi_x)"] \\
			F(\D_Q^2(x)) \ar[r, "\varphi_{\D_Q(x)}"] & \D_{Q'}(F(\D_Q(x)))
		\end{tikzcd}\]
		commutes.
	\end{enumerate}
	Restricting along the full inclusion $\excat \subset \wexcat$, $\cal{E} \mapsto (\cal{E}, \mathbf{iso})$ we likeqwise construct the category $\formcat \subset \cat_1^{h, w}$ of exact form categories and nonsingular form functors; we likewise write $\wcompformcat \subset \cat^{h, w}_1$ for the pullback along the forgetful functor $\wcompcat \subset \wexcat$, for $\wcompcat$ the subcategory of (small) complicial exact categories with weak equivalences and exact functors between them (see below).
\end{remark}
Write $\mathscr{C}_{\zz}$ for the category of bounded chain complexes of finitely generated free abelian groups. Equipped with the degreewise-split exact structure, the class of chain homotopy equivalences, and the duality $X \mapsto [X, \mathbb{1}]$ coming from the closed symmetric monoidal structure furnished by the tensor product of bounded chain complexes, $\mathscr{C}_{\zz}$ carries the canonical structure of an exact category with weak equivalences and duality $(\mathscr{C}_{\zz}, \mathbf{ch.htp.}, [-,\mathbb{1}], \mathrm{can})$, for $\mathrm{can}_X$ the adjunct of the evaluation morphism $X \otimes [X, \mathbb{1}] \to \mathbb{1}$ (see Appendix \ref{excat}). The Dwyer-Kan localisation $L_\mathrm{ch.htp.}(\mathscr{C}_{\zz})$ at the chain homotopy equivalences is canonically equivalent to the perfect derived $\infty$-category $\dperf(\zz)$.\\
Recall that a \textit{complicial exact category} (see Appendix \ref{excat}) is an exact category equipped with a bi-exact action
\[
	\otimes : \mathscr{C}_{\zz} \times \cal{E} \to \cal{E}
\]
which is associative and unital in that the usual diagrams commute (see \cite{Gra76}). An exact category with weak equivalences and duality $(\cal{E}, w, \D, \eta)$ has the structure of a complicial exact category with weak equivalences and duality if $\cal{E}$ is complicial, such that the action of $\mathscr{C}_{\zz}$ preserves weak equivalences in each variable, i.e.\ for $f:A \to B$ a chain homotopy equivalence in $\mathscr{C}_{\zz}$ and $g:X \to Y$ in $w\cal{E}$, the map
\[
	f \otimes g : A \otimes X \to B \otimes Y
\]
is a weak equivalence; and if $\otimes$ refines to a nonsingular form functor between exact categories with duality, i.e.\ that we have a natural \textit{isomorphism} $\varphi$ rendering the diagram 
\[\begin{tikzcd}
	\mathscr{C}_{\zz}^{\op} \times \cal{E}^{\op} \ar[d, "\otimes^{\op}"] \ar[r, "{[-, \mathbb{1}] \times \D}"] & |[alias=cze]| \mathscr{C}_{\zz} \times \cal{E} \ar[d, "\otimes"] \\
	|[alias=e]| \cal{E}^{\op} \ar[r, "\D"] & \cal{E}
	\ar[Rightarrow,from=cze, to=e,shorten >=5mm,shorten <=.5mm, "\varphi"', near start]
\end{tikzcd}\]
commutative.
\begin{remark}
    We will often abuse notation and write $\cal{E}$ for an exact category with a complicial structure, leaving the tensor product and higher coherences implicit.
\end{remark}
\begin{remark}
     Write $C$ for the complex $(0 \to \zz =\joinrel= \zz \to 0)$ concentrated in degrees $[0,1]$, sitting in the conflation $\mathbb{1} \mono C \epi T$,
     and $P := [C, \mathbb{1}]$, sitting in the shifted conflation $\Omega \mono P \epi \mathbb{1}$. Then the tuple $(C \otimes -, P \otimes -, \varphi)$ is the data of a strong symmetric cone on $\cal{E}$ in the sense of \cite[Def.\ 9.1]{Sch24}, since each of $C, P$ is Frobenius contractible in $\mathscr{C}_{\zz}$.
\end{remark}
\begin{warning}
	The action of $\mathscr{C}_{\zz}$ on $\cal{E}$ in the definition of a complicial exact form category is \textit{not} required to promote to a form functor
	\[
		(\mathscr{C}_{\zz}, Q^s, [-, \mathbb{1}], \mathrm{can}) \otimes (\cal{E}, Q, \D, \eta) \to (\cal{E}, Q, \D, \eta),
	\]
	where $\otimes$ denotes the tensor product of form categories of \cite[Def.\ 2.34]{Sch21}; that is, we require no compatibility between the quadratic functor $Q$ and the action of $\mathscr{C}_{\zz}$.
\end{warning}
\begin{example}
	For $(\cal{E}, \D, \eta)$ an exact category with duality, the category of bounded chain complexes $\chb(\cal{E})$ canonically acquires the structure of a complicial exact category with weak equivalences (the quasi-isomorphisms) and strong duality; see Appendix \ref{excat}. If $\cal{E}$ is complicial exact with weak equivalences and $I$ is a small category, the functor category $\fun(I, \cal{E})$ carries a canonical (pointwise) structure of a complicial exact category with weak equivalences. Finally, the opposite of a complicial exact category with weak equivalences canonically has the structure of a complicial exact category with weak equivalences (see (\ref{opcomp})).
\end{example}
A complicial exact category $\cal{E}$ comes equipped with an internal class of \textit{Frobenius equivalences} $w_\mathrm{Frob}$, the complicial analogue of chain homotopy equivalences. Any class $w$ of weak equivalences on $\cal{E}$ contains the Frobenius equivalences (\ref{Frob-in-w}), and we write $\gamma:\cal{E} \to L_w(\cal{E})$ for the Dwyer-Kan localisation at $w$; note that $L_w(\cal{E})$ is by Proposition \ref{comstab} is a stable $\infty$-category. If $w$ is the class of Frobenius equivalences, we write $L_{\mathrm{Frob}}(\cal{E}) := L_{w_{\mathrm{Frob}}}(\cal{E})$.
\subsection{From exact categories to chain complexes}\label{excat-to-chain-complexes}
Given an exact form category with strong duality $(\cal{E}, Q, \D, \eta)$, there is a fully faithful exact form functor \cite{Sch24}
\begin{equation}\label{gillet-waldhausen}
	(\cal{E}, Q, \D, \eta, \mathbf{iso}) \to (\chb(\cal{E}), Q_\mathrm{ch}, \D, \eta, \mathbf{qis}),
\end{equation}
of exact form categories with strong duality and weak equivalences. Here we abusively write $(\D, \eta)$ for the canonical extensions of the duality on $\cal{E}$ to $\chb(\cal{E})$ (see Appendix \ref{excat}), and $Q_\mathrm{ch}$ for the quadratic functor defined on a chain complex $X$ by
\begin{equation}\label{qtil}
	Q_\mathrm{ch}(X) := \{(\xi, \varphi) \mid \varphi \in \Hom_{\chb(\cal{E}}(X, \D(X))^{\fct},  \xi \in Q(X[0]), d_1^\bullet(\xi) = 0, \rho_{X[0]}(\xi) = \varphi_0\};
\end{equation}
see \cite[Def.\ 10.2]{Sch24}. The following Gillet-Waldhausen-type theorem is \cite[Thm.\ 10.5]{Sch24}.
\begin{theorem}
	For $\cal{E}$ an exact form category with strong duality, the form functor (\ref{gillet-waldhausen}) induces a natural weak equivalence of Grothendieck-Witt spaces
	\[
		\cal{GW}(\cal{E}, Q) \to \cal{GW}(\chb(\cal{E}), Q_\mathrm{ch}, \mathbf{qis}).
	\]
\end{theorem}
Many of the formal properties of the category $\chb(\cal{E})$ (simplicial enrichment, stability of the localisation at the chain homotopy equivalences) ultimately derive from its complicial structure. We couch the discussion below in terms of complicial exact categories with weak equivalences, showing below that derived functors of presheaves on these admit explicit models. In turn, complicial exact form categories with weak equivalences and strong duality are tractable 1-categorical models for Poincar\'e categories.
\subsection{Hermitian and Poincar\'e $\infty$-categories}\label{hermpoi}
In this brief interlude we recall some of the technical background for Grothendieck-Witt theory in the higher-categorical setting; for a detailed exposition, see \cite[\S\kern-1pt\S 1-2]{CD23a}. Recall \cite[Def.\ 1.1.1.9]{HA} that an $\infty$-category is stable if it is pointed (i.e., has a zero object), admits all fibres and cofibres, and has the property that a nullcomposite sequence
\[\begin{tikzcd}
    x \ar[r, "i"] \ar[d] & y \ar[d, "p"] \\
    0 \ar[r] & z
\end{tikzcd}\]
is a fibre sequence if and only if it is a cofibre sequence. Fix a small stable $\infty$-category $\cal{C}$. To any reduced functor $\Qoppa:\cal{C}^{\op} \to \Sp$ we may associate its \textit{polarisation}
\[
	B_\Qoppa(x, y) := \fib(\Qoppa(x \oplus y) \to \Qoppa(x) \oplus \Qoppa(y));
\]
this assignment is functorial in $\Qoppa$, and exhibits $B_\Qoppa$ as the universal bi-reduced\footnote{A functor $B:\cal{C}^{\op} \times \cal{C}^{\op} \to \Sp$ is bi-reduced if for each $x \in \cal{C}$, each restriction $B(x, -), B(-,x):\cal{C}^{\op} \to \Sp$ is reduced.} replacement of the bifunctor $\Qoppa(- \oplus -): \cal{C}^{\op} \times \cal{C}^{\op} \to \Sp$ (see \cite[\S1]{CD23a}. Write $\mathrm{BiFun}_*(\cal{C}^{\op}; \Sp)$ for the $\infty$-category of bi-reduced functors $\cal{C}^{\op} \times \cal{C}^{\op} \to \Sp$. Then we have natural transformations
\begin{equation}\label{rhotau}
	\Delta^*B_\Qoppa \Rightarrow \Qoppa \Rightarrow \Delta^*B_\Qoppa,
\end{equation}
which are counit and unit in adjunctions exhibiting $\Qoppa \mapsto B_\Qoppa$ as respectively right and left adjoint to restriction along the diagonal $\Delta^*: \mathrm{BiFun}_*(\cal{C}^{\op}; \Sp) \to \fun_*(\cal{C}^{\op}, \Sp)$. For $\Qoppa$ reduced, $B_\Qoppa$ refines to a $C_2$-homotopy fixed point in $\mathrm{BiFun}_*(\cal{C}^{\op}, \Sp)$, where $C_2$ acts by flipping the input variables; we call such a functor \textit{symmetric}. By \cite[Lem.\ 1.1.10]{CD23a}, the maps (\ref{rhotau}) refine to $C_2$-equivariant transformations, where $C_2$ acts trivially on $\fun_*(\cal{C}^{\op}, \Sp)$, and hence to maps
\[
	(\Delta^*B_\Qoppa)_{\hct} \Rightarrow \Qoppa \Rightarrow (\Delta^*B_\Qoppa)^{\hct}.
\]
Call a bi-reduced functor $B_\Qoppa:\cal{C}^{\op} \times \cal{C}^{\op} \to \Sp$ \textit{bilinear} if it is $(1, 1)$-excisive, i.e.\ 1-excisive in each variable. Then by \cite[Prop.\ 1.1.13]{CD23a}, a reduced functor $\Qoppa:\cal{C}^{\op} \to \Sp$ is 2-excisive if and only if its polarisation $B_\Qoppa$ is bilinear, and if the functor
\[
	\Lambda_\Qoppa : \cal{C}^{\op} \to \Sp, \quad x \mapsto \fib\left(\Qoppa(x) \to B_\Qoppa(x, x)^{\hct}\right)
\]
is 1-excisive. We use also the term \textit{quadratic} to mean reduced and 2-excisive, and call the pair $(\cal{C}, \Qoppa)$ a \textit{hermitian $\infty$-category}. Analogously to the case of form categories, we wish to study quadratic functors on $\cal{C}$ satisfying an additional nondegeneracy condition: call a quadratic functor $\Qoppa$ \textit{nondegenerate} if there is a natural equivalence
\[
	B_\Qoppa(x, y) \simeq \hom_{\cal{C}}(x, \D_\Qoppa(y))
\]
for some (necessarily essentially unique) duality $\D_\Qoppa:\cal{C}^{\op} \to \cal{C}$. The equivalences
\[
	\map_{\cal{C}}(x, \D_\Qoppa(y)) \simeq B_\Qoppa(x, y) \simeq B_\Qoppa(y, x) \simeq \map_{\cal{C}}(y, \D_\Qoppa(x)) \simeq \map_{\cal{C}^{\op}}(\D_\Qoppa^{\op}(x), y)
\]
exhibit $\D_\Qoppa$ as self-adjoint, and we call $B_\Qoppa$ \textit{perfect} if the unit (and hence also counit) $\ev:\id_{\cal{C}} \Rightarrow \D_\Qoppa\D_\Qoppa^{\op}$ of this adjunction is a natural equivalence.
\begin{definition}[{\cite[Def.\ 1.2.8]{CD23a}}]
	A hermitian $\infty$-category $(\cal{C}, \Qoppa)$ is \textit{Poincar\'e} if $B_\Qoppa$ is nondegenerate and perfect; this is the higher-categorical analogue of a form category with strong duality.
\end{definition}
Write $\fun^\mathrm{q}(\cal{C})$ for the subcategory of $\fun(\cal{C}^{\op}, \Sp)$ spanned by reduced 2-excisive functors, and write $\cat_\infty^h$ for the cartesian unstraightening of
\[
	(\cat^{\mathrm{st}}_\infty)^{\op} \to \cal{C}\mathrm{AT}_\infty, \quad \cal{C} \mapsto \fun^\mathrm{q}(\cal{C}),
\]
the (large) $\infty$-category of small hermitian $\infty$-categories. A map $(\cal{C}, \Qoppa) \to (\cal{C}', \Qoppa')$ in $\cat_\infty^h$ is a pair $(f, \eta)$, for $f:\cal{C} \to \cal{C}'$ an exact functor, and $\eta : \Qoppa \Rightarrow \Qoppa'\circ f^{\op}$. Such an $\eta$ induces a map (see \cite[Lem.\ 1.2.4]{CD23a})
\[
	\theta_{f, \eta}:f\D_\Qoppa \Rightarrow \D_{\Qoppa'}\circ f^{\op},
\]
and we call the pair $(f, \eta)$ \textit{duality preserving} if $\theta_{f, \eta}$ is an equivalence. Denote by $\cat^p_\infty \subset \cat^h_\infty$ the subcategory spanned by Poincar\'e $\infty$-categories and duality-preserving functors. $\cat_\infty^p$ and $\cat^h_\infty$ are compactly generated presentable $\infty$-categories, by the results of \cite{CD24}.
\subsection{Deriving presheaves on complicial exact categories}\label{dquad}
Let $(\cal{E}, w)$ be a complicial exact categories with weak equivalences; recall that the identity functor gives rise to an exact inclusion of complicial exact categories with weak equivalences $(\cal{E}_{\mathrm{Frob}}, w_{\mathrm{Frob}}) \subset (\cal{E}, w)$.
\begin{definition}[{\cite[Def.\ 7.4.12, 7.5.7]{Cis19}}]\label{wef}
	An \textit{$\infty$-category with weak equivalences and fibrations} is a tuple $(\cal{C}, w, \mathrm{Fib})$, with $\cal{C}$ an $\infty$-category with final object $*$, $w \subset \cal{C}$\footnote{We may occasionally write $w\cal{C}$ for $w$.} a wide subcategory satisfying the 2-of-3 property, and $\mathrm{Fib} \subset \cal{C}$ a class of fibrations, such that the following properties are satisfied.
	\begin{enumerate}[label=(\roman*)]
		\item \label{i} For any cartesian square
		\[\begin{tikzcd}
			x' \ar[r] \ar[d, "q"'] & x \ar[d, "p"] \\
			y' \ar[r] & y
		\end{tikzcd}\]
		in $\cal{C}$ with $p$ a fibration between fibrant objects and $y'$ fibrant, if $p$ belongs to $w$, so does $q$.
		\item \label{ii} For any map $f:x \to y$ with $y$ fibrant, there exists a map $i:x \to x'$ in $w$ and $p:x' \to y$ in $\mathrm{Fib}$ such that $f$ is a composition of $p$ and $i$.
	\end{enumerate}
	An object $x \in \cal{E}$ is \textit{fibrant} if the essentially unique map $x \to *$ lies in $\mathrm{Fib}$, and $\cal{C}$ is moreover an \textit{$\infty$-category of fibrant objects} if each object $x \in \cal{C}$ is fibrant.
\end{definition}
\begin{lemma}\label{fib}
	(The nerve of) any complicial exact category with weak equivalences $(\cal{E}, w)$ is an $\infty$-category of fibrant objects upon setting $\mathrm{Fib}$ to be the class of deflations.
\end{lemma}
\begin{proof}
    Since the sequences $x =\joinrel= x \to 0$ are conflations, each object is fibrant. Given a cartesian square as in \ref{i}, if $p$ is a deflation then so is $q$; a deflation is then trivial if and only if its kernel is acyclic \cite[Lem.\ 7.1]{Sch24}, and since this is so for $p$, it follows for $q$. For \ref{ii}, given $f:x \to y$, we note that in the diagram
    \[\begin{tikzcd}
        x \ar[r, "{\begin{psmallmatrix} 1 \\ 0 \end{psmallmatrix}}"] \ar[rd, "f"'] & x \oplus Py \ar[d, "{\begin{psmallmatrix} f \amsamp \pi \end{psmallmatrix}}"] \\
        & y,
    \end{tikzcd}\]
    for $\pi:Py \to y$ the path fibration of Appendix \ref{excat}, the inflation $x \to x \oplus Py$ has (Frobenius) contractible cokernel $Py$, and so the former is a (Frobenius) weak equivalence. Now the composite $Py \xto{\begin{psmallmatrix}0 \\ 1\end{psmallmatrix}} x \oplus Py \to y$ is a fibration, and since the kernel of the latter map exists as the pullback
    \[\begin{tikzcd}
        \ker({\begin{psmallmatrix} f \amsamp \pi \end{psmallmatrix}}) \ar[r] \ar[d, twoheadrightarrow] & Py \ar[d, twoheadrightarrow, "\pi"] \\
        x \ar[r, "-f"] & y,
    \end{tikzcd}\]
    ${\begin{psmallmatrix} f \amsamp \pi \end{psmallmatrix}}$ is a deflation, by Quillen's obscure axiom \cite[App.\ A.1]{Kel90}.
\end{proof}
For $(\cal{C}, w, \mathrm{Fib})$ an $\infty$-category with weak equivalences and fibrations, write $\gamma:\cal{C} \to L_w(\cal{C})$ for the Dwyer-Kan localisation at the class of weak equivalences. For $x \in \cal{C}$ fibrant, the slice category $(\cal{C} \downarrow x)$ inherits the structure of an $\infty$-category with weak equivalences and fibrations in which the weak equivalences resp.\ fibrations are those maps which are sent to such under the projection $(\cal{C} \downarrow x) \to \cal{C}$ \cite[\S7.6.12]{Cis19}. Fibrant objects in this inherited structure are precisely the fibrations $y \epi x$. We note the following for the record.
\begin{proposition}[{\cite[Cor.\ 7.6.13]{Cis19}}]\label{loc}
	Let $\cal{C}$ be an $\infty$-category with weak equivalences and fibrations. For any fibrant object $x \in \cal{C}$, the canonical functor $(\cal{C} \downarrow x) \to (L_w(\cal{C})\downarrow \gamma(x))$ induces an equivalences of $\infty$-categories
	\[
		L_w(\cal{C}\downarrow x) \simeq (L_w(\cal{C}) \downarrow \gamma(x)).
	\]
\end{proposition}
\begin{corollary}\label{cof}
	Let $(\cal{E}, w)$ be a complicial exact category with weak equivalences. For $x \in \cal{E}$ and $\gamma:\cal{E} \to L_w(\cal{E})$ the localisation, the functors
	\[
		\gamma_{x/}:(x \downarrow \cal{E}) \to (\gamma(x) \downarrow L_w(\cal{E})), \quad \gamma_{/x}:(\cal{E} \downarrow x) \to (L_w(\cal{E}) \downarrow \gamma(x))
	\]
	are both final and cofinal.
\end{corollary}
\begin{proof}
	By Proposition \ref{loc}, $\gamma_{/x}$ is a localisation. The result for $\gamma_{/x}$ then follows since any Dwyer-Kan localisation is both final and cofinal \cite[Prop.\ 7.1.10]{Cis19}. That for $\gamma_{x/}$ follows upon noting that $(\cal{E}\downarrow x)^{\op} \cong (x \downarrow \cal{E}^{\op})$ and that the opposite of a complicial exact category with weak equivalences is a complicial exact category with weak equivalences.
\end{proof}
Write $\sset$ for the ordinary category of simplicial sets. There is a functor
\[
	N\zz[-]:\sset^{\mathrm{fin}} \to \mathscr{C}_{\zz},
\]
sending a finite simplicial set to the normalised chain complex on the associated free simplicial abelian group. $N_n\zz[K]$ identifies with the free abelian subgroup $\zz[K_n^{\mathrm{nd}}]$ on the nondegenerate $n$-simplices of $K$, with differential
\[
	\partial_n : N_n\zz[K] \to N_{n-1}\zz[K], \quad \sigma \mapsto \sum_{0 \le i \le n}\begin{cases}(-1)^id^n_i(\sigma), & d^n_i(\sigma) \text{ nondegenerate,} \\ 0, & \text{ else.}\end{cases}
\]
The canonical isomorphisms $\zz[K \times L] \cong \zz[K] \otimes \zz[L]$ of simplicial abelian groups, and the Alexander-Whitney and Eilenberg-Zilber maps $\Delta_{A, B} : N(A \otimes B) \to N(A) \otimes N(B)$ and $\nabla_{A, B}:N(A) \otimes N(B) \to N(A \otimes B)$ respectively, for simplicial abelian groups $A$ and $B$, render $N\zz[-]$ a lax and oplax monoidal functor of symmetric monoidal categories, where the source is given the cartesian symmetric monoidal structure, and the target the symmetric monoidal structure induced by the tensor product of chain complexes (see for instance \cite[\S29]{May92}). $\Delta_{A, B}$ and $\nabla_{A, B}$ are moreover chain homotopy equivalences: we have a diagram
\[\begin{tikzcd}
	N(A) \otimes N(B) \ar[rd, equal]\ar[r, "\nabla_{A, B}"] & N(A \otimes B) \ar[rd, equal, ""{name=eq}] \ar[d, "\Delta_{A, B}"'] \\
	& |[alias=nanb]| N(A) \otimes N(B) \ar[r, "\nabla_{A, B}"'] & N(A \otimes B),
	\ar[Rightarrow,from=nanb, to=eq,shorten >=5mm,shorten <=.025mm]
\end{tikzcd}\]
where the left triangle commutes on the nose, and the right up to a specified chain homotopy \cite[Cor.\ 29.10]{May92}. Given a finite simplicial set $K$, the composite
\[
	N(\zz[K]) \xto{N(\delta_{K, K})} N(\zz[K] \otimes \zz[K]) \xto{\nabla_{K, K}} N(\zz[K]) \otimes N(\zz[K])
\]
equips $N(\zz[K])$ with a cocommutative counital coalgebra structure, where we abuse notation in writing $\delta_{K, K} := \delta_{\zz[K], \zz[K]}$ for the diagonal map in simplicial abelian groups, etc. The counit
\[
	N(\zz[K]) \to N(\zz[\Delta^0]) = \zz[0]
\]
is induced by the unique map $K \to \Delta^0$ of simplicial sets. For $\Delta^\bullet$ the standard cosimplicial simplicial set, $N(\zz[\Delta^\bullet])$ is by naturality a cosimplicial counital cocommutative coalgebra.\\
Fix for the rest of this section a complicial exact category with weak equivalences and strong duality $(\cal{E}, w, D, \eta)$, and write $w_{\mathrm{Frob}} \subset w$ for the subcategory of Frobenius weak equivalences. For a finite simplicial set $K$ and $x \in \cal{E}$, we write $Kx := N\zz[K] \otimes x$.
\begin{construction}\label{qdb}
	Given a complicial exact category $(\cal{E}, \otimes)$, a presheaf $F$ of abelian groups on $\cal{E}$, write $F^{\Delta^\bullet}:\cal{E}^{\op} \to \sab$ for the presheaf of simplicial abelian groups
	\[
		F^{\Delta^\bullet}(x) := F(\Delta^\bullet x).
	\]
    $F^{\Delta^\bullet}$ receives a natural transformation from $F$ induced by the unique map of cosimplicial simplicial sets $\Delta^\bullet \to \Delta^0$.
\end{construction}
\begin{proposition}\label{ftil}
	The functor
		\[
			F^{\Delta^\bullet}:\cal{E}^{\op} \to \sab
		\]
		sends Frobenius homotopies to simplicial homotopies.
\end{proposition}
\begin{proof}
	Write $G:= F^{\Delta^\bullet}$. We claim for each $x$ that the composites
	\[\begin{tikzcd}
		G(\Delta^1x) \ar[r, shift left=.6ex, "G(d^1)"] \ar[r, shift right=.6ex, "G(d^0)"'] & G(x) \ar[r, "p"] & G(\Delta^\bullet x),
	\end{tikzcd}\]
	are simplicially homotopic, where $p$ is induced upon taking diagonals by the map $\Delta^\bullet \to \Delta^0$. We have maps
	\[
		\Hom_{\sab}(\zz[\Delta^\bullet], \zz[\Delta^1]) \otimes G(\Delta^1x) \to G(\Delta^\bullet x), \quad f \otimes \xi \mapsto G(f \otimes 1)(\xi),
	\]
	and
	\[
		\Delta^1 \to \Hom_{\sab}(\zz[\Delta^\bullet], \zz[\Delta^1])
	\]
	with the latter the image under Yoneda of $1_{\zz[\Delta^1]} \in \Hom_{\sab}(\zz[\Delta^1], \zz[\Delta^1])$. The composite
	\[
		\zz[\Delta^1] \otimes G(\Delta^1x) \to G(\Delta^\bullet x)
	\]
	then gives the desired homotopy. The comultiplication $\nabla : \Delta^\bullet \to \Delta^\bullet \otimes \Delta^\bullet$ induces a map of simplicial abelian groups
	\[
		G(\Delta^\bullet x) = F((\Delta^\bullet \otimes \Delta^\bullet) x) \xto{F(\nabla \otimes 1)} G(x),
	\]
	and the composite
	\[
		\Delta^1 \otimes G(\Delta^1x) \to G(\Delta^\bullet x) \to G(x)
	\]
	is a homotopy between $G(d^1)$ and $G(d^0)$. Given then a Frobenius homotopy witnessed by the diagram
	\[\begin{tikzcd}
		x \ar[d, "d^1"] \ar[rd, "f"] \\
		\Delta^1x \ar[r, "H"] & y. \\
		x \ar[u, "d^0"'] \ar[ru, "g"']
	\end{tikzcd}\]
	we see that $G(f) = G(Hd^1) \sim G(Hd^0) = G(g)$.
\end{proof}
Write $\mathrm{Ch}(\zz)$ for the abelian category of (unbounded) complexes of abelian groups. By for instance \cite[Prop.\ 1.3.5.3]{HA}, $\mathrm{Ch}(\zz)$ admits a left proper cofibrantly generated combinatorial model structure with cofibrations the levelwise monomorphisms, and weak equivalences the quasi-isomorphisms, and we write $\mathrm{D}(\zz)$ for the underlying derived $\infty$-category, obtained for instance as the dg-nerve of the full subcategory $\mathrm{Ch}(\zz)^\circ \subset \mathrm{Ch}(\zz)$ of fibrant-cofibrant objects. $\mathrm{D}(\zz)$ is stable, presentable, and carries a natural $t$-structure with connectives (coconnectives) those complexes with vanishing homology in negative (positive) degrees; this $t$-structure is moreover accessible, right-complete, and such that $\mathrm{D}_{\le 0}(\zz)$ is stable under small filtered colimits by \cite[Prop.\ 1.3.5.21]{HA}. We may identify $\mathrm{D}_{\ge 0}(\zz)$ with the animation of $\Ab$, i.e.\ with the $\infty$-category $\fun^\Pi(\mathrm{Free}^\mathrm{fg}(\zz)^{\op}, \cal{S})$ of additive presheaves of spaces on the category of compact projective (i.e.\ finitely generated free) abelian groups, which coincides with the localisation of the category $\sab$ of simplicial abelian groups at the weak equivalences \cite[Ex.\ 5.1.6(b)]{CS24}, and also with $\infty$-category of connective $\hz$-modules $\whzmod$ \cite[Th.\ 1.1]{Shi07}.\\
The presheaf $F^{\Delta^\bullet}$ thus descends to a presheaf on the localisation $L_{\mathrm{Frob}}(\cal{E})$, taking values in $\mathrm{D}_{\ge 0}(\zz)$, which we also denote $F^{\Delta^\bullet}$. For $x \in \cal{E}$, write $J_{x, \mathrm{Frob}} \subset (\cal{E}\downarrow x)$ for the full subcategory spanned by the trivial Frobenius deflations over $x$. There is a functor $\iota_x:\bbDelta \to J_{x, \mathrm{Frob}}$, sending $[n]$ to the codegeneracy map
\[
	\Delta^nx \stackrel{\sim}\twoheadrightarrow x,
\]
and an arrow $[n] \xto{\theta} [m]$ in $\bbDelta$ to
\[
	(\Delta^nx \stackrel{\sim}\twoheadrightarrow x) \xto{\theta_* \otimes 1} (\Delta^mx \stackrel{\sim}\twoheadrightarrow x).
\]
\begin{proposition}\label{hoco}
	$\iota^{\op}_x$ is cofinal; in particular for each $x \in \cal{E}$, there is a natural equivalence
	\[
		F^{\Delta^\bullet}(\gamma(x)) \simeq \underset{[n] \in \bbDelta^{\op}}\colim\ F(\Delta^nx)\to \underset{J_{x, \mathrm{Frob}}^{\op}}\colim\ F
	\]
 in $\mathrm{D}_{\ge 0}(\zz)$.
\end{proposition}
\begin{proof}
	It suffices by Quillen's Theorem A (\cite{Qui73}, \cite[Th.\ 4.1.3.1]{HTT}) to show that the slice categories
	\[
		((y, p) \downarrow \iota_x^{\op}) \simeq (\iota_x \downarrow (y, p))^{\op}
	\]
	have weakly contractible nerves for each trivial fibration $p:y \stackrel{\sim}\twoheadrightarrow x$ in $\cal{E}$. Objects of $(\iota_x \downarrow (y, p))^{\op}$ are pairs $([n] \in \bbDelta, \gamma: \Delta^nx \to y)$, $\gamma$ a map over $x$, with morphisms
	\[
		([n], \gamma) \to ([m], \gamma')
	\]
	arrows $\theta:[m] \to [n]$ in $\bbDelta$ rendering the diagram
	\begin{equation}\label{secn}\begin{tikzcd}
		& y \ar[dd, "p", twoheadrightarrow, "\rotatebox{270}{$\sim$}"', near start] \\
		\Delta^mx \ar[rd, twoheadrightarrow, "\rotatebox{325}{$\sim$}"'] \ar[ru, "\gamma'"]  && \Delta^nx \ar[lu, "\gamma"'] \ar[ld, twoheadrightarrow, "\rotatebox{35}{$\sim$}"] \ar[from=ll, "\theta_* \otimes 1", near start, crossing over] \\
		& x
	\end{tikzcd}\end{equation}
	commutative in $\cal{E}$. Since $p$ is a trivial Frobenius deflation, its kernel is contractible, hence injective-projective for the Frobenius exact structure on $\cal{E}$. The conflation
	\[
		\ker(p) \mono y \epi x
	\]
	then splits, and $p$ admits a section; the space $\mathrm{Sec}(p)$ of such sections is a contractible Kan complex by Lemma \ref{ssec}.\\
	Recall that the category of simplices of $\mathrm{Sec}(p)$ is the Grothendieck construction of the associated functor
	\[
		\mathrm{Sec}(p):\bbDelta^{\op} \to \set \subset \cat, \quad [n] \mapsto \mathrm{Sec}_n(p) \subset \Hom_{\cal{E}}(\Delta^nx, y)
	\]
	with objects pairs $([n], \xi \in \mathrm{Sec}_n(p))$, where an object of $\mathrm{Sec}_n(p)$ is a map $\Delta^nx \to y$ whose composition with $p$ is the collapse map $\Delta^nx \stackrel{\sim}\twoheadrightarrow x$. Maps $([n], \gamma) \to ([m], \gamma')$ are maps $\theta:[m] \to [n]$ in $\bbDelta$ such that the diagram (\ref{secn}) commutes, from which we see that there is a canonical identification of $\int_{\bbDelta^{\op}}\mathrm{Sec}(p)$ with $(\bbDelta \downarrow (y, p))^{\op}$. By Thomason's theorem (\cite[Th.\ 1.2]{Tho79}, \cite[Cor.\ 3.3.4.6]{HTT}), it follows that
	\[
		N((\bbDelta \downarrow (y, p))^{\op}) = N\left(\int_{\bbDelta^{\op}}\mathrm{Sec}(p)\right) \simeq \underset{[n] \in \bbDelta^{\op}}\colim\ \mathrm{Sec}_n(p) \simeq \mathrm{Sec}(p) \simeq \Delta^0.
	\]
\end{proof}
\begin{corollary}\label{LKE}
	$F^{\Delta^\bullet}$ is the left Kan extension of $F$ along (the opposite of) the localisation
	\[
		\gamma:\cal{E} \to L_{\mathrm{Frob}}(\cal{E}).
	\]
\end{corollary}
\begin{proof}
	Recall (Lemma \ref{fib}) that $(\cal{E}, \mathrm{Fib}, w_{\mathrm{Frob}})$ is an $\infty$-category of fibrant objects. Set $v_{\mathrm{Frob}} \subset w_{\mathrm{Frob}}$ the subcategory of trivial Frobenius deflations, and note that $v_{\mathrm{Frob}}$ is closed under composition and pullbacks in the sense of \cite[Def.\ 7.2.14]{Cis19}. Note moreover that the saturation $\overline{v_{\mathrm{Frob}}}$ of $v_{\mathrm{Frob}}$, defined by the cartesian square
	\[\begin{tikzcd}
		\overline{v_{\mathrm{Frob}}} \ar[r, hookrightarrow] \ar[d] & \cal{E} \ar[d, "\gamma"] \\
		L_{v_{\mathrm{Frob}}}(\cal{E})^{\simeq} \ar[r, hookrightarrow] & L_{v_{\mathrm{Frob}}}\cal{E},
	\end{tikzcd}\]
	contains $w_{\mathrm{Frob}}$: by Brown's lemma \cite[Prop.\ 7.4.13]{Cis19}, any map $f:x \to y$ in $\cal{E}$ admits a factorisation
	\[\begin{tikzcd}
		x \ar[rd, "f"] \ar[r, "j"] & z \ar[d, "g", twoheadrightarrow] \\
		& y,
	\end{tikzcd}\]
	where $j$ admits a retraction which is a trivial fibration, and $g$ is a fibration. If $f$ is a Frobenius equivalence, so are $j$ and $g$, by 2-of-3. Accordingly, if $F:\cal{E}^{\op} \to \mathrm{D}_{\ge 0}(\zz)$ is a functor inverting maps in $v_{\mathrm{Frob}}$, it also inverts maps in $w_{\mathrm{Frob}}$. We then see from \cite[Th.\ 7.2.16, Cor.\ 7.2.18]{Cis19} that $J_{x, \mathrm{Frob}}$ is a right calculus of fractions at $x$ in the sense of Cisinski, and accordingly by \cite[Cor.\ 7.2.9]{Cis19} the left Kan extension of $F:\cal{E}^{\op} \to \sab \to \mathrm{D}_{\ge 0}(\zz)$ along the localisation $\gamma$ satisfies
\[
	\gamma_!F(\gamma(x)) \simeq \underset{J_{x, \mathrm{Frob}}^{\op}}\colim\ F \simeq F^{\Delta^\bullet}(\gamma(x)).
\]
\end{proof}
\begin{remark}
	Recall from Appendix \ref{excat} that any complicial exact category $\cal{E}$ is canonically enriched over simplicial abelian groups via
	\[
		\Map_\Delta(x, y) := \Hom_{\cal{E}}(\Delta^\bullet x, y),
	\]
	with homotopy coherent nerve a model for the localisation $L_{\mathrm{Frob}}(\cal{E})$ at the Frobenius equivalences. Write $\cal{E}_\Delta$ for $\sab$-enriched category thus obtained. Now $\sab$ is canonically closed symmetric monoidal with internal mapping complexes  $\underline{\Hom}(A, B) := \Hom_{\sab}(\zz[\Delta^\bullet] \otimes A, B)$. The functor $F^{\Delta^\bullet}$ canonically enhances to an enriched functor
	\[
		F^{\Delta^\bullet}:\cal{E}_\Delta^{\op} \to \sab
	\]
	via the map
	\[
		\Map_\Delta(x, y) \to \underline{\Hom}(F^{\Delta^\bullet}(y), F^{\Delta^\bullet}(x))
	\]
	given on $n$-simplices by
	\begin{align*}
		\Hom_{\sab}(\zz[\Delta^n], \Map_\Delta(x, y)) & \to \Hom_{\sab}(\zz[\Delta^n] \otimes F^{\Delta^\bullet}(y), F^{\Delta^\bullet}(x)), \\
		\theta & \mapsto (\sigma \otimes \xi \mapsto Q((1 \otimes \theta(\sigma))\circ \nabla_x^n)(\xi).
	\end{align*}
	Compatibility with composition and unit follow ultimately from coassociativity of the coalgebra structure on $\Delta^n$. Taking homotopy-coherent nerves, we obtain a model for the induced presheaf on localisation $L_{\mathrm{Frob}}(\cal{E})$.
\end{remark}
\subsection{The underlying hermitian $\infty$-category of a complicial exact form category}\label{dadd}
Given a complicial exact form category with weak equivalences $(\cal{E}, Q, \D, \eta, w)$, we may apply the construction of the previous section to obtain a model $\qdb:\cal{E}^{\op} \to \sab$ for the right derived functor of $Q$ along the localisation $\cal{E} \to L_{\mathrm{Frob}}(\cal{E})$. Note that for each $n\ge 0$ the $C_2$-Mackey functors
\[
	\Hom_{\cal{E}}(\Delta^nx, \D(\Delta^nx)) \xto{\tau_{\Delta^nx}} Q(\Delta^nx) \xto{\rho_{\Delta^nx}} \Hom_{\cal{E}}(\Delta^nx, \D(\Delta^nx))
\]
of Definition \ref{addq}\ref{c2mack} assemble into $C_2$-equivariant maps of simplicial abelian groups
\[
	\Hom_{\cal{E}}(\Delta^\bullet x, \D(\Delta^\bullet x)) \xto{\tau_{\Delta^\bullet x}} \qdb(x) \xto{\rho_{\Delta^\bullet x}} \Hom_{\cal{E}}(\Delta^\bullet x, \D(\Delta^\bullet x)),
\]
where $C_2$ acts levelwise, so in particular
\[
	\Hom_{\cal{E}}(\Delta^\bullet x, \D(\Delta^\bullet x))^{\fct} = \left([n] \mapsto \Hom_{\cal{E}}(\Delta^nx, \D(\Delta^nx))^{\fct}\right),
\]
and similarly for $C_2$-orbits.
\begin{remark}\label{poi}
	Consider the subfunctor $Q_\mathrm{nd}$ of the restriction
	\[
		w_\mathrm{Frob}\cal{E}^{\op} \to \Ab \to \set
	\]
	with $Q_\mathrm{nd}(x) := \{\xi \in Q(x) \mid \rho_x(\xi) \in w_\mathrm{Frob}\cal{E}\}$. That this assignment is functorial in weak equivalences follows from the commutativity of
	\[\begin{tikzcd}
		Q(x) \ar[r, "\rho_x"] \ar[d, "f^\bullet"] & \Hom_{\cal{E}}(x, \D(x)) \ar[d, "\D(f)_*f^*"] \\
		Q(y) \ar[r, "\rho_y"] & \Hom_{\cal{E}}(y, \D(y))
	\end{tikzcd}\]
	for each map $f:y \to x$, and that $\D$ preserves Frobenius equivalences. The proof of Proposition \ref{hoco} then implies that the map of spaces
	\[
		\underset{J_{x, \mathrm{Frob}}^{\op}}\colim\ Q_\mathrm{nd} \to Q_\mathrm{nd}^{\Delta^\bullet}(x) = Q_\mathrm{nd}(\Delta^\bullet x)
	\]
	is a weak equivalence.
\end{remark}
In the remainder of this section we prove that $\qdb$ gives a model for (the connective cover of) the 2-excisive (right) derived functor of $Q$ on the stable $\infty$-category $L_{\mathrm{Frob}}(\cal{E})$. Recall that for $\cal{C}$ and $\cal{D}$ $\infty$-categories with finite colimits and limits respectively, a functor $\Qoppa:\cal{C} \to \cal{D}$ is said to be 2-excisive if for each strongly cocartesian cube $C:N\mathcal{P}([2]) \to \cal{C}$, the cube $\Qoppa(C)$ in $\cal{D}$ is a limit diagram; see Appendix \ref{n-exc} for details. The equivalent formulation of 2-excisivity for reduced spectrum-valued functors in the stable setting in \cite[Prop.\ 1.1.13]{CD23a} is for our purposes more wieldy: for $\cal{C}$ a stable $\infty$-category, a reduced functor $\Qoppa:\cal{C}^{\op} \to \Sp$ is 2-excisive if and only if the functors $B_\Qoppa$ and $\Lambda_\Qoppa$ informally given by
\[
	B_\Qoppa : \cal{C}^{\op} \times \cal{C}^{\op} \to \Sp, \quad (x, y) \mapsto \fib\left(\Qoppa(x \oplus y) \to \Qoppa(x) \oplus \Qoppa(y)\right)
\]
and
\[
	\Lambda_\Qoppa : \cal{C}^{\op} \to \Sp, \quad x \mapsto \fib\left(\Qoppa(x) \to B_\Qoppa(x, x)^{\hct}\right)
\]
are respectively bilinear (also called $(1,1)$-excisive) and $1$-excisive, i.e.\ preserve fibre-cofibre sequences (in each variable). The same is true for functors taking values in connective spectra (or connective $\hz$-modules): suppose given $\Qoppa:\cal{C}^{\op} \to \Sp_{\ge 0}$ sitting in a fibre sequence
\[
	\Lambda_\Qoppa \to \Qoppa \to \left[\Delta^*B_\Qoppa\right]^{\hct}
\]
with $\Lambda_\Qoppa$ 1-excisive and $B_\Qoppa$ bilinear. By \cite[Cor.\ 6.1.1.14, 6.1.3.5]{HA}, 1-excisive functors and diagonals of bilinear functors are 2-excisive, and the subcategory of 2-excisive functors is closed under limits and in particular homotopy fixed points. Given then a strongly cocartesian cube $C:N(\mathcal{P}([2])) \to \cal{C}$, there is a fibre sequence of total fibres
\[
	\mathrm{fibt}\left(\Lambda_{\Qoppa}(C)\right) \to \mathrm{fibt}\left(\Qoppa(C)\right) \to \mathrm{fibt}\left(\Delta^*(B_{\Qoppa})^{\hct}(C)\right),
\]
for which the first and last terms, and hence the second, are trivial by 2-excisivity, i.e.\ $\Qoppa(C)$ is a limit diagram in $\Sp_{\ge 0}$, and $\Qoppa$ is 2-excisive. The converse is the same as in \cite[Prop.\ 1.1.13]{CD23a}, and since we do not need this, we omit the details.\\
We first consider the 1-categorical analogue of the linear part $\Lambda_\Qoppa$. For $(\cal{E}, Q, \D, \eta, w)$ a complicial exact form category with weak equivalences, set
\[
	L_Q : \cal{E}^{\op} \to \Ab, \quad x \mapsto \ker\left(Q(x) \xto{\rho_x} \Hom_{\cal{E}}(x, \D(x))^{\fct}\right).
\]
\begin{lemma}\label{llex}
	$L_Q$ is left exact, i.e.\ sends each conflation $x \stackrel{i}\mono y \stackrel{p}\epi z$ to a left-exact sequence of abelian groups
	\[
		0 \to L_Q(z) \xto{L_Q(p)} L_Q(y) \xto{L_Q(i)} L_Q(x).
	\]
\end{lemma}
\begin{proof}
	Recall that $Q$ and $\left[\Delta^*\Hom_{\cal{E}}(-, \D(-))\right]^{\fct}$ are quadratic left exact with the map $\rho:Q \Rightarrow \left[\Delta^*\Hom_{\cal{E}}(-, \D(-))\right]^{\fct}$ inducing an isomorphism on polarisations (Example \ref{bq}). Given a conflation as in the statement, the diagram
	\[\begin{tikzcd}
		0 \ar[r] & L_Q(z) \ar[r, "L_Q(p)"] \ar[d, "\ker"] & L_Q(y) \ar[r, "L_Q(i)"] \ar[d, "\ker"] & L_Q(x) \ar[d, "\ker"] \\
		0 \ar[r] & Q(z) \ar[r, "p^\bullet"] \ar[d, "\rho_z"] & Q(y) \ar[r, "{\begin{psmallmatrix} i^\bullet \\ \D(i)_*\rho_y \end{psmallmatrix}}"] \ar[d, "\rho_y"] & Q(x) \oplus \Hom_{\cal{E}}(y, \D(x)) \ar[d, "{\begin{psmallmatrix}\rho_x \amsamp 0 \\ 0 \amsamp 1 \end{psmallmatrix}}"] \\
		0 \ar[r] & \Hom_{\cal{E}}(z, \D(z))^{\fct} \ar[r, "p^*\D(p)_*"] & \Hom_{\cal{E}}(y, \D(y))^{\fct} \ar[r, "{\begin{psmallmatrix} i^*\D(i)_* \\ \D(i)_* \end{psmallmatrix}}"] & \Hom_{\cal{E}}(x, \D(x))^{\fct} \oplus \Hom_{\cal{E}}(y, \D(x))
	\end{tikzcd}\]
	exhibits $L_Q$ as left exact.
\end{proof}
Define functors
\begin{align*}
	& B_{\qdb} : \cal{E}^{\op} \times \cal{E}^{\op} \to \sab, \quad (x, y) \mapsto \hofib\left(\qdb(x \oplus y) \to \qdb(x) \oplus \qdb(y)\right), \\
	& L_{\qdb} : \cal{E}^{\op} \to \sab, \quad x \mapsto \hofib\left(\qdb(x) \to B_{\qdb}(x, x)^{\hct}\right),
\end{align*}
where the homotopy fibres are taken in simplicial abelian groups with respect to the classical model structure, and the map $\qdb(x) \to B_{\qdb}(x,x)^{\hct}$ is the composite of $\rho_{\Delta^\bullet x}$ with the natural map $B_{\qdb}(x, x)^{\fct} \to B_{\qdb}(x, x)^{\hct}$. Since the inclusions $x \to x \oplus y \leftarrow y$ are split monomorphisms, the map $\qdb(x \oplus y) \to \qdb(x) \oplus \qdb(y)$ is a levelwise split surjection of simplicial abelian groups and hence a fibration. Accordingly, we have a natural zig-zag
\[
	B_{\qdb}(x, y) \simeq \Hom_{\cal{E}}(\Delta^\bullet x, \D(\Delta^\bullet y)) \simeq \Map_\Delta(x, \D(y)),
\]
where the second equivalence is induced by the map of bisimplicial abelian groups given in bidegree $(p, q)$ by 
\[
	\Hom_{\cal{E}}(\Delta^px, \D(y)) \to \Hom_{\cal{E}}(\Delta^px, \D(\Delta^qy))
\]
induced by the codegeneracies $\Delta^q \to \Delta^0$. Since $\D$ preserves Frobenius equivalences (Definition \ref{comdual}) and $\Hom_{\cal{E}}(\Delta^\bullet x, -)$ sends Frobenius equivalences to homotopy equivalences, this is a levelwise weak equivalence and hence an equivalence upon taking diagonals.\\
As before, $\qdb$ descends to a functor
\[
	\qdb : L_{\mathrm{Frob}}(\cal{E})^{\op} \to \mathrm{D}_{\ge 0}(\zz),
\]
and similarly $B_{\qdb}$ and $L_{\qdb}$ descend respectively to functors
\begin{align*}
	B_{\qdb} : L_{\mathrm{Frob}}(\cal{E})^{\op} \times L_{\mathrm{Frob}}(\cal{E})^{\op} \to \mathrm{D}_{\ge 0}(\zz), \\
	\Lambda_{\qdb} : L_{\mathrm{Frob}}(\cal{E})^{\op} \to \mathrm{D}_{\ge 0}(\zz).
\end{align*}/
Note that the functors $\cal{E}^{\op} \to L_w(\cal{E})^{\op}$ and $\cal{E}^{\op} \times \cal{E}^{\op} \to L_w(\cal{E})^{\op} \times L_w(\cal{E})^{\op}$ are Dwyer-Kan localisations by \cite[Prop.\ 7.1.7, 7.1.13]{Cis19}, for any set of weak equivalences $w \subset \cal{E}$.
\begin{lemma}\label{a12}
	Suppose given an exact category $\cal{E}$, and a cosimplicial sequence
	\[
		x^\bullet \to y^\bullet \to z^\bullet
	\]
	in $\cal{E}$ which is levelwise split exact. Then for any quadratic functor $Q:\cal{E}^{\op} \to \Ab$ with polarisation $B$, the map $Q(z^\bullet) \to Q(y^\bullet)$ exhibits $Q(z^\bullet)$ as the total homotopy fibre in simplicial abelian groups of the square
	\begin{equation}\begin{tikzcd}\label{tot}
		Q(y^\bullet) \ar[r] \ar[d] & Q(x^\bullet) \ar[d] \\
		B(y^\bullet, x^\bullet) \ar[r] & B(x^\bullet, x^\bullet).
	\end{tikzcd}\end{equation}
\end{lemma}
\begin{proof}
	By \cite[Lem.\ A.12]{Sch21}, for each $n$ the sequence of abelian groups
	\[
		0 \to Q(z^n) \to Q(y^n) \to Q(x^n) \oplus B(y^n, x^n) \to B(x^n, x^n) \to 0
	\]
	is exact, and so the map $Q(z^n) \to Q(Y^n)$ exhibits $Q(z^n)$ as the total kernel of the square
	\[\begin{tikzcd}
		Q(y^n) \ar[r] \ar[d] & Q(x^n) \ar[d] \\
		B(y^n, x^n) \ar[r] & B(x^n, x^n).
	\end{tikzcd}\]
	Surjectivity of $Q(x^n) \oplus B(y^n, x^n) \to B(x^n, x^n)$ implies in fact that $HQ(z^n) \to HQ(y^n)$ exhibits the former as the total fibre in $\mathrm{D}(\zz)$ of the corresponding square of Eilenberg-Mac Lane spectra. Since fibre-cofibre sequences of spectra commute with realisation, the map $Q(z^\bullet) \to Q(y^\bullet)$ thus exhibits the former as the total fibre in $\mathrm{D}(\zz)$ of (\ref{tot}), and since truncation preserves fibre sequences, we are done.
\end{proof}
\begin{proposition}
	The functor $B_{\qdb}$ is reduced and bilinear, and $\Lambda_{\qdb}$ is reduced and 1-excisive.
\end{proposition}
\begin{proof}
	The first claim follows from the identification $B_{\qdb}(x, y) \simeq \Map_{L_{\mathrm{Frob}}(\cal{E})}(x, \D(y))$, where we write $\D$ also for the duality induced on $L_{\mathrm{Frob}}(\cal{E})$. For the second, it suffices to show that $\Lambda_{\qdb}$ sends exact sequences to fibre sequences in $\mathrm{D}_{\ge 0}(\zz)$: it then follows from pasting for pullbacks (the dual of \cite[Lem.\ 4.4.2.1]{HTT}) that it exhibits the correct behaviour on exact squares.\\
Now an exact sequence in $L_{\mathrm{Frob}}(\cal{E})$ is up to equivalence the image of a Frobenius conflation in $\cal{E}$ under the localisation functor $\cal{E} \to L_{\mathrm{Frob}}(\cal{E})$. For such a sequence $x \stackrel{i}\mono y \stackrel{p}\epi z$ in $\cal{E}$, consider the diagram of cosimplicial objects
\begin{equation}\label{sub}\begin{tikzcd}
	\tilde{x}^\bullet \ar[r, rightarrowtail] \ar[d, rightarrowtail] & \tilde{y}^\bullet \ar[r, twoheadrightarrow] \ar[d, rightarrowtail] & \tilde{z}^\bullet \ar[d, rightarrowtail] \\
	\Delta^\bullet x \ar[r, rightarrowtail] \ar[d, twoheadrightarrow] & \Delta^\bullet y \ar[r, twoheadrightarrow] \ar[d, twoheadrightarrow] & \Delta^\bullet z \ar[d, twoheadrightarrow] \\
	cx \ar[r, rightarrowtail] & cy \ar[r, twoheadrightarrow] & cz,
\end{tikzcd}\end{equation}
with $\tilde{x}^n := \ker(\Delta^nx \to x)$, and $cx$ the constant diagram at $x$, and so on. Each column and the top row is a levelwise Frobenius conflation, the latter since $\tilde{x}^n$, $\tilde{y}^n, \tilde{z}^n$ are contractible, and hence Frobenius injective-projective. For any cosimplicial object $w^\bullet:\bbDelta \to \cal{E}$, there is a natural map of simplicial abelian groups
\[
	Q(w^\bullet) \to B_{Q}(w^\bullet, w^\bullet)^{\fct} \to B_{Q}(w^\bullet, w^\bullet)^{\hct}.
\]
The image of (\ref{sub}) in $\sab$ under this natural transformation is
\begin{tiny}\begin{equation}\label{decomp}\begin{tikzcd}
	& cB_{Q}(z, z)^{\fct} \ar[rr] \ar[dd] && cB_{Q}(y, y)^{\fct} \ar[rr] \ar[dd] && cB_{Q}(x, x)^{\fct} \ar[dd] \\
	cQ(z) \ar[ru] \ar[dd] && cQ(y) \ar[from=ll, crossing over] \ar[ru] && cQ(x) \ar[ru] \ar[from=ll, crossing over] \\
	& B_{Q}(\Delta^\bullet z, \Delta^\bullet z)^{\hct} \ar[rr] \ar[dd] && B_{Q}(\Delta^\bullet y, \Delta^\bullet y)^{\hct} \ar[rr] \ar[dd] && B_{Q}(\Delta^\bullet x, \Delta^\bullet x)^{\hct} \ar[dd] \\
	Q(\Delta^\bullet z) \ar[ru] \ar[dd] && Q(\Delta^\bullet y) \ar[from=ll, crossing over] \ar[from=uu, crossing over] \ar[ru] && Q(\Delta^\bullet x) \ar[from=ll, crossing over] \ar[ru] \ar[from=uu, crossing over] \\
	& B_{Q}(\tilde{z}^\bullet, \tilde{z}^\bullet)^{\hct} \ar[rr] && B_{Q}(\tilde{y}^\bullet, \tilde{y}^\bullet)^{\hct} \ar[rr] && B_{Q}(\tilde{x}^\bullet, \tilde{x}^\bullet)^{\hct}\\
	Q(\tilde{z}^\bullet) \ar[ru] \ar[rr] && Q(\tilde{y}^\bullet) \ar[from=uu, crossing over] \ar[ru] \ar[rr] && Q(\tilde{x}^\bullet) \ar[ru] \ar[from=uu, crossing over]\\
\end{tikzcd}\end{equation}\end{tiny}
where we use that the natural map $A^{\fct} \to A^{\hct}$ for any discrete simplicial abelian group $A$ with $C_2$-action is a weak equivalence. Taking homotopy fibres in $\sab$ of the inward-pointing maps in (\ref{decomp}), we obtain a diagram
\begin{equation}\label{LQ}\begin{tikzcd}
	L_{Q}(z) \ar[r] \ar[d] & L_{Q}(y) \ar[r] \ar[d] & L_{Q}(x) \ar[d] \\
	L_{\qdb}(z) \ar[r] \ar[d] & L_{\qdb}(y) \ar[r] \ar[d] & L_{\qdb}(x) \ar[d] \\
	\hofib(\rho_{\tilde{z}^\bullet}) \ar[r] & \hofib(\rho_{\tilde{y}^\bullet}) \ar[r] & \hofib(\rho_{\tilde{x}^\bullet}),
\end{tikzcd}\end{equation}
for $\rho_{\tilde{z}^\bullet}: Q(\tilde{z}^\bullet) \to B_{Q}(\tilde{z}^\bullet,\tilde{z}^\bullet)^{\hct}$ the map induced levelwise by
\[
	Q(\tilde{z}^n) \xto{\rho_{\tilde{z}^n}} B_{Q}(\tilde{z}^n, \tilde{z}^n),
\]
which factors through $C_2$-fixed points. The top row is a homotopy fibre sequence since $\ker(L_{Q}(y) \to L_{Q}(x)) \cong L_{Q}(z)$; the homotopy fibre of $L_{\qdb}(z) \to \hofib(\rho_{\tilde{z}^\bullet})$ is by definition the total homotopy fibre
\begin{equation}\label{totfib}
	\mathrm{hofibt}\left(\begin{tikzcd}
		Q(\Delta^\bullet z) \ar[r] \ar[d] & B_{Q}(\Delta^\bullet z, \Delta^\bullet z)^{\hct} \ar[d] \\
		Q(\tilde{z}^\bullet) \ar[r] & B_{Q}(\tilde{z}^\bullet, \tilde{z}^\bullet)^{\hct}
	\end{tikzcd}\right).
\end{equation}
But $Q(\Delta^\bullet z) \to Q(\tilde{z}^\bullet)$ is a levelwise split surjection and hence a fibration in the classical model structure on simplicial abelian groups; similarly for the right-hand column prior to taking $C_2$-homotopy fixed points. The total homotopy fibre of the cube
\begin{small}\begin{equation}\begin{tikzcd}\label{cube}
	& B_{Q}(\Delta^\bullet z, \Delta^\bullet z)^{\hct} \ar[rr] \ar[dd] && B_{Q}(\tilde{z}^\bullet, \tilde{z}^\bullet)^{\hct} \ar[dd] \\
	Q(\Delta^\bullet z) \ar[ru] \ar[dd] && \ar[from=ll, crossing over] Q(\tilde{z}^\bullet) \ar[ru] \\
	& B_{Q}(\Delta^\bullet z, \tilde{z}^\bullet) \ar[rr] && B_{Q}(\tilde{z}^\bullet, \tilde{z}^\bullet) \\
	B_{Q}(\Delta^\bullet z, \tilde{z}^\bullet) \ar[ru, equal] \ar[rr] && B_{Q}(\tilde{z}^\bullet, \tilde{z}^\bullet) \ar[from=uu, crossing over] \ar[ru, equal]
\end{tikzcd}\end{equation}\end{small}
identifies with (\ref{totfib}), since the bottom face is trivially cartesian, and by commuting homotopy limits, we may compute this as the homotopy fibre of the map between the respective total homotopy fibres of the front and back faces; by Lemma \ref{a12} applied to the quadratic functors $Q$ and $\Delta^*B_Q$, this is
\[
	\hofib\left(Q(z) \to B_{Q}(z, z)^{\fct}\right) \simeq \ker\left(Q(z) \to B_{Q}(z, z)^{\fct}\right) = L_{Q}(z).
\]
Note that the polarisation $B$ of $\Delta^*B_Q$
\[
	B(u, v) = B_{Q}(u, v) \oplus B_{Q}(v, u),
\]
where the $C_2$-action swaps the factors, and the square corresponding to the setup of Lemma \ref{a12} is
\[\begin{tikzcd}
	B_{Q}(\Delta^\bullet z, \Delta^\bullet z) \ar[r] \ar[d] & B_{Q}(\tilde{z}^\bullet, \tilde{z}^\bullet) \ar[d] \\
	B_{Q}(\Delta^\bullet z, \tilde{z}^\bullet) \oplus B_{Q}(\tilde{z}^\bullet, \Delta^\bullet z) \ar[r] & B_{Q}(\tilde{z}^\bullet, \tilde{z}^\bullet) \oplus B_{Q}(\tilde{z}^\bullet, \tilde{z}^\bullet),
\end{tikzcd}\]
which upon taking $C_2$-homotopy fixed points gives
\[\begin{tikzcd}
	B_{Q}(\Delta^\bullet z, \Delta^\bullet z)^{\hct} \ar[r] \ar[d] & B_{Q}(\tilde{z}^\bullet, \tilde{z}^\bullet)^{\hct} \ar[d] \\
	B_{Q}(\Delta^\bullet z, \tilde{z}^\bullet) \ar[r] & B_{Q}(\tilde{z}^\bullet, \tilde{z}^\bullet).
\end{tikzcd}\]
The same argument applied to the other two columns and the bottom row of (\ref{LQ}) gives that these are homotopy fibre sequences, since $\tilde{x}^\bullet \to \tilde{y}^\bullet \to \tilde{z}^\bullet$ is again levelwise a split Frobenius conflation. It then follows that the middle row is a homotopy fibre sequence of simplicial abelian groups, i.e.\ $L_{\qdb}$ sends conflations to fibre sequences of simplicial abelian groups, so by \cite[Th.\ 4.2.4.1]{HTT} $\Lambda_{\qdb}$ sends exact sequences in $L_\mathrm{Frob}(\cal{E})$ to fibre sequences in $\mathrm{D}_{\ge 0}(\zz)$.
\end{proof}
\begin{remark}
	In fact, each column and the bottom row of (\ref{LQ}) is a fibre sequence in $\mathrm{D}(\zz)$; the failure of $\Lambda_{\qdb}$ to be 1-excisive when considered as a functor valued in $\mathrm{D}(\zz)$ is thus controlled by the failure of the top row of (\ref{LQ}) to be right exact in $\Ab$.
\end{remark}
As a consequence of the discussion prior to Lemma \ref{llex}, we see that:
\begin{corollary}\label{qdb2exc}
	The functor $\qdb: L_{\mathrm{Frob}}(\cal{E})^{\op} \to \mathrm{D}_{\ge 0}(\zz)$ is reduced and 2-excisive.
\end{corollary}
Postcomposing with the inclusion $\mathrm{D}_{\ge 0}(\zz) \hookrightarrow \mathrm{D}(\zz)$, we obtain by abuse of notation a functor
\[
	\qdb : L_{\mathrm{Frob}}(\cal{E})^{\op} \to \mathrm{D}(\zz),
\]
which is reduced but in general fails spectacularly to be 2-excisive. Recall from \cite[Cons.\ 1.1.2.6]{CD23a} that for a stable $\infty$-category $\cal{C}$, the inclusion
\[
	i:\fun_*(\cal{C}^{\op}, \mathrm{D}(\zz)) \hookrightarrow \fun_*^{2\mathrm{-exc}}(\cal{C}^{\op}, \mathrm{D}(\zz))
\]
admits a left adjoint $P_2$, the 2-excisive approximation, where $\fun_*$ denotes the full subcategory of reduced functors. As demonstrated in the following proposition, when applied to a 2-excisive $\mathrm{D}_{\ge 0}(\zz)$-valued functor, this is in some sense a universal 2-excisive delooping. Write $\fun^\mathrm{q} := \fun_*^{2\mathrm{-exc}}$, and $\widetilde{\mathbf{R}}Q := P_2\qdb$. We have adjunctions
    \[\begin{tikzcd}
		\fun_*(\cal{C}^{\op}, \mathrm{D}_{\ge 0}(\zz)) \ar[r, shift left=.4ex, bend left=.75ex, "\rotatebox{90}{$\vdash$}"', "j_*"] & \fun_*(\cal{C}^{\op}, \mathrm{D}(\zz)) \ar[r, shift left=.4ex, bend left=.75ex, "\rotatebox{90}{$\vdash$}"', "P_2"] \ar[l, shift left=.4ex, bend left=.75ex, "(\tau_{\ge 0})_*"] \ar[r, shift right=.4ex, bend right=.75ex, hookleftarrow, "i"'] & \fun^\mathrm{q}(\cal{C}^{\op}, \mathrm{D}(\zz))
	\end{tikzcd}\]
	where we write $j:\mathrm{D}_{\ge 0}(\zz) \hookrightarrow \mathrm{D}(\zz)$ for the inclusion, left adjoint to $\tau_{\ge 0}$. Pointwise truncation sends 2-excisive $\mathrm{D}(\zz)$-valued functors to 2-excisive $\mathrm{D}_{\ge 0}(\zz)$-valued functors, and so this restricts to an adjunction
	\[\begin{tikzcd}
		\fun^\mathrm{q}(\cal{C}^{\op}, \mathrm{D}_{\ge 0}(\zz)) \ar[r, shift left=.4ex, bend left=.75ex, "\rotatebox{90}{$\vdash$}"', "P_2j_*"] \ar[r, shift right=.4ex, bend right=.75ex, hookleftarrow, "\tau_{\ge 0,*}i"'] & \fun^\mathrm{q}(\cal{C}^{\op}, \mathrm{D}(\zz)).
	\end{tikzcd}\]
\begin{proposition}\label{ext}
	For any stable $\infty$-category $\cal{C}$, the above is an adjoint equivalence. In particular, for any complicial exact form category with weak equivalences $(\cal{E}, Q, w, \D, \eta)$, the natural map $\qdb \to \widetilde{\mathbf{R}}Q$ truncates to an equivalence $\qdb \xto{\simeq} \tau_{\ge 0}\widetilde{\mathbf{R}}Q$ of functors $L_\mathrm{Frob}(\cal{E})^{\op} \to \mathrm{D}_{\ge 0}(\zz)$.
\end{proposition}
\begin{proof}
        Suppose given some $\Qoppa:\cal{C}^{\op} \to \mathrm{D}_{\ge 0}(\zz)$ $2$-excisive, the unit
	\[
		\Qoppa \to \tau_{\ge 0,*}iP_2j_*\Qoppa
	\]
	is given pointwise by
	\[
		\Qoppa(x) \to \tau_{\ge 0}P_2j\Qoppa(x) \simeq \tau_{\ge 0}\varinjlim_n T_2^nj\Qoppa(x),
	\]
	and since $\tau_{\ge 0}$ commutes with filtered colimits, to show this is an equivalence it suffices to show that $\tau_{\ge 0}T_2^{\mathrm{D}(\zz)}j \simeq T_2^{\mathrm{D}_{\ge 0}(\zz)}\tau_{\ge 0}j \simeq T_2^{\mathrm{D}_{\ge 0}(\zz)}$, where the superscript indicates where the fibre is taken. But
	\[
		T_2\Qoppa(x) = \Omega\fib(\Qoppa(\Omega x) \to B_\Qoppa(\Omega x, \Omega x)),
	\]
	and $\tau_{\ge 0}$, as a right adjoint, preserves fibre sequences. The left adjoint $P_2j_*$ is then fully faithful, and so it suffices to show that $\tau_{\ge 0_*}$ is conservative; this follows from \cite[Lem.\ 1.1.25]{CD23a}.
\end{proof}
\begin{remark}
	We have a canonical natural transformation $Q \Rightarrow \qdb \Rightarrow \widetilde{\mathbf{R}}Q$, with the first map given by the inclusion of $0$-simplices, and the second the unit of the adjunction $P_2 \dashv i$. Accordingly, there are 2-cells in $\cat_\infty$
	\[\begin{tikzcd}
		\cal{E}^{\op} \ar[rd, "\gamma"] \ar[rr, "cQ"{name=cQ}] && \mathrm{D}_{\ge 0}(\zz). \\
		& |[alias=LwE]|L_w(\cal{E})^{\op} \ar[ru, bend right = 8ex, "\widetilde{\mathbf{R}}Q"'{name=Qoppa}] \ar[ru, "\qdb"{name=qdb}]
		\ar[from=cQ, to=LwE, Rightarrow,shorten >=2mm,shorten <=2mm]
		\ar[from=qdb, to=Qoppa, Rightarrow,shorten >=2mm,shorten <=2mm]
	\end{tikzcd}\]
	Writing $Q_\mathrm{nd} \subset Q$ and $\widetilde{\mathbf{R}}Q_\mathrm{nd} \subset \Omega^\infty\widetilde{\mathbf{R}}Q$ for the subfunctors of nondegenerate forms, this restricts to a natural transformation
	\[
		Q_\mathrm{nd} \Rightarrow \gamma^*\widetilde{\mathbf{R}}Q_\mathrm{nd}.
	\]
\end{remark}
Write $B_{\widetilde{\mathbf{R}}Q}$ for the polarisation of $\widetilde{\mathbf{R}}Q$, given objectwise by
\[
	B_{\widetilde{\mathbf{R}}Q}(x, y) := \fib\left(\widetilde{\mathbf{R}}Q(x \oplus y) \to \widetilde{\mathbf{R}}Q(x) \oplus \widetilde{\mathbf{R}}Q(y)\right),
\]
and $\D:L_{\mathrm{Frob}}(\cal{E})^{\op} \to L_{\mathrm{Frob}}(\cal{E})$ for the functor induced by $\D:\cal{E}^{\op} \to \cal{E}$ under localisation (recall that $\D$ preserves Frobenius equivalences). Recall that mapping spaces in any stable $\infty$-category $\cal{C}$ canonically enhance to mapping spectra, via \cite[Cor.\ 1.4.2.23]{HA}: the functor $\Omega^\infty:\Sp \to \cal{S}$ induces an equivalence of $\infty$-categories
\[
	\Omega^\infty_*:\fun^{\mathrm{ex}}(\cal{C}, \Sp) \xto{\simeq} \fun^{\mathrm{lex}}(\cal{C}, \cal{S}),
\]
where we write $\fun^{\mathrm{ex}}$ for the subcategory of exact functors, and $\fun^{\mathrm{lex}}$ for the subcategory of left-exact (finite limit preserving) functors. For $x \in \cal{C}$, write $\map_{\cal{C}}(x, -)$ for the functor corresponding under this equivalence to $\Map_{\cal{C}}(x, -)$. This assignment is natural in $x$ and accordingly assembles into a mapping spectrum bifunctor, $\map_{\cal{C}} : \cal{C}^{\op} \times \cal{C} \to \Sp$. In our case, $L_w(\cal{E})$ is $\zz$-linear, and we have a mapping bifunctor
\[
	L_\mathrm{Frob}(\cal{E})^{\op} \times L_\mathrm{Frob}(\cal{E}) \to \mathrm{D}(\zz).
\]
Recall that we write $\widetilde{\mathbf{R}}Q:=P_2\qdb:L_\mathrm{Frob}(\cal{E})^{\op} \to \mathrm{D}(\zz)$.
\begin{proposition}\label{ndg}
	There are equivalences
	\[
		B_{\widetilde{\mathbf{R}}Q}(x, y) \simeq \map_{L_\mathrm{Frob}(\cal{E})}(x, \D(y)),
	\]
	in $\mathrm{D}_{\zz}$, natural in $x$ and $y$.
\end{proposition}
\begin{proof}
    Recall that the 2-excisive approximation $P_2\qdb$ is computed pointwise in $\mathrm{D}(\zz)$ as the sequential colimit
    \[
        P_2\qdb(x) \simeq \colim(\qdb(x) \to T_2\qdb(x) \to T_2^2\qdb(x) \to \dots),
    \]
    where $T_2\qdb(x) = \Omega\fib\left(\qdb(\Omega x) \to B_{\qdb}(\Omega x ,\Omega x)\right)$. We compute
    \[
	B_{T_2\qdb}(x, y) \simeq \Omega\kern+1pt\mathrm{fibt}
		\left(\begin{tikzcd}
			\qdb(\Omega x \oplus \Omega y) \ar[r] \ar[d] & B_{\qdb}(\Omega x \oplus \Omega y, \Omega x \oplus \Omega y) \ar[d] \\
			\qdb(\Omega x) \oplus \qdb(\Omega y) \ar[r] & B_{\qdb}(\Omega x, \Omega x) \oplus B_{\qdb}(\Omega y, \Omega y) 
		\end{tikzcd}\right)
    \]
    \begin{align*}
	& \simeq \Omega\kern+1pt\fib\left(B_{\qdb}(\Omega x, \Omega y) \xto{\Delta} B_{\qdb}(\Omega x, \Omega y) \oplus B_{\qdb}(\Omega y, \Omega x)\right) \\
	& \simeq \Omega^2B_{\qdb}(\Omega x, \Omega y) \simeq \Omega^2\Map_\Delta(\Omega x, \Sigma\D(y)),
    \end{align*}
    using that $\D(\Omega y)) \simeq \Sigma\D(y))$. Inductively, we have
    \[
	B_{T_2^n\qdb}(x, y) \simeq \Omega^{2n}B_{\qdb}(\Omega^nx, \Sigma^n\D(y)) \simeq \Omega^{2n}B_{\qdb}(\Omega^{2n}x, \D(y)),
    \]
    and in the colimit,
    \[
	B_{\widetilde{\mathbf{R}}Q}(x, y) \simeq P_1(\Map(x, -))\D(y)),
    \]
    for $P_1$ the $1$-excisive approximation of Appendix \ref{approx}. But this is precisely the mapping spectrum $\map(x, \D(y))$ in $L_{\mathrm{Frob}}(\cal{E})$: for $x \in \cal{E}$ and an exact functor $F : L_{\mathrm{Frob}}(\cal{E}) \to \mathrm{D}(\zz)$ we have a string of equivalences
    \begin{align*}
	& \mathrm{Nat}_{\fun^{\mathrm{ex}}(L_{\mathrm{Frob}}(\cal{E}), \mathrm{D}(\zz))}\left(\map(x, -), F\right) \\
	\simeq & \mathrm{Nat}_{\fun^{\mathrm{lex}}(L_{\mathrm{Frob}}(\cal{E}), \mathrm{D}_{\ge 0}(\zz))}\left(\Map(x, -), \tau_{\ge 0}F\right) \\
	\simeq & \mathrm{Nat}_{\fun(L_{\mathrm{Frob}}(\cal{E}), \mathrm{D}(\zz))}(\Map(x,-), F) \\
	\simeq & \mathrm{Nat}_{\fun^{\mathrm{ex}}(L_{\mathrm{Frob}}(\cal{E}), \mathrm{D}(\zz))}(P_1\Map(x, -), F).
    \end{align*}
\end{proof}
Recall \cite[\S1.1]{CD23b} that the Verdier quotient $\pi:L_{\mathrm{Frob}}(\cal{E}) \to L_w(\cal{E})$ is an exact functor of stable $\infty$-categories, sitting in a fibre-cofibre sequence
\[
	L_{\mathrm{Frob}}(\cal{E}^w) \to L_{\mathrm{Frob}}(\cal{E}) \xto{\pi} L_w(\cal{E}),
\]
 $\cal{E}^w \subset \cal{E}$ the full subcategory of $w$-acyclic objects, which is closed under retracts. By \cite[Lem.\ 1.4.1(iii)]{CD23a}, left Kan extension along an exact functor preserves preserves quadraticity (in the stable setting), and so writing $\mathbf{R}Q := \pi_!\widetilde{\mathbf{R}}Q$, the pair $(L_w(\cal{E}), \mathbf{R}Q)$ is again the data of a hermitian $\infty$-category.
\begin{theorem}\label{comppoi}
	For $(\cal{E}, Q, \D, \eta, w)$ a complicial exact category with weak equivalences and strong duality, the pair $(L_w(\cal{E}), \mathbf{R}Q)$ is the data of a Poincar\'e $\infty$-category. Moreover, for each $x \in \cal{E}$, $\mathbf{R}Q$ satisfies
	\[
		\tau_{\ge 0}\mathbf{R}Q(\gamma(x)) \simeq \underset{J_x^{\op}}\colim\ Q,
	\]
	for $J_x \subset (\cal{E}\downarrow x)$ the full subcategory spanned by trivial deflations over $x$.
\end{theorem}
\begin{proof}
	Suppose firstly that $w$ is the class of Frobenius equivalences. In this case, $\mathbf{R}Q = \widetilde{\mathbf{R}}Q$ is reduced and 2-excisive by Corollary \ref{qdb2exc}, with nondegenerate polarisation by Proposition \ref{ndg}; it remains to show that the induced duality $\D:L_{\mathrm{Frob}}(\cal{E})^{\op} \to L_{\mathrm{Frob}}(\cal{E})$ is an equivalence. Since the unit $\id_{\cal{E}} \Rightarrow \D\D^{\op}$ is a Frobenius equivalence by assumption, the induced functor $L_{\mathrm{Frob}}(\cal{E})^{\op} \to L_{\mathrm{Frob}}(\cal{E})$ is an equivalence.\\
	For the general case, we note that the polarisation $B_{\mathbf{R}Q} \simeq (\pi_! \times \pi_!)B_{\widetilde{\mathbf{R}}Q}$ is nondegenerate by the formula for mapping spaces in a Verdier quotient \cite[Th.\ I.3.3]{NS18}, and that the induced duality $\D:L_w(\cal{E})^{\op} \to L_w(\cal{E})$ is an equivalence since again the unit and counit are natural weak equivalences. The last claim then follows from Proposition \ref{rf} (or Proposition \ref{hoco} if $w=w_{\mathrm{Frob}}$), since the truncation $\tau_{\ge 0}$ commutes with the filtered colimit appearing in the formula for left Kan extension along $L_\mathrm{Frob}(\cal{E}) \to L_w(\cal{E})$.
\end{proof}
\begin{remark}\label{ddual}
	If the double dual identification $\eta$ on $\cal{E}$ is a natural Frobenius equivalence\footnote{This is the case for instance for the exact form category $(\chb(\cal{E}), Q_\mathrm{ch}, \mathbf{qis}, \D, \eta)$ extended from $(\cal{E}, Q, \D, \eta)$}, the claim for general $w$ follows immediately from that for $w_{\mathrm{Frob}}$: by \cite[Ex.\ 1.1.7]{CD23b}, for $(\cal{C}, \Qoppa)$ a Poincar\'e category and $p:\cal{C} \to \cal{D}$ a Verdier projection, the pair $(\cal{D}, p_!\Qoppa)$ is Poincar\'e if and only if $\ker(p)$ is closed under the duality $\D_{\Qoppa}$. Note for this that $\cal{E}^w \subset \cal{E}$ is closed under the (exact) duality functor.
\end{remark}
\subsection{Functoriality}\label{functoriality}
Write $\excat$ for the category of small exact categories and exact functors, and $\wcompcat$ for the category of small complicial exact categories with weak equivalences and exact functors between them. Recall from Remark \ref{formcat} the definition of $\wcompformcat$ as the subcategory of the Grothendieck construction on the functor
\[
	\wcompcat^{\op} \to \mathrm{CAT}_1, \quad (\cal{E}, w) \mapsto \fun^\mathrm{q}(\cal{E})
\]
spanned by those tuples $(\cal{E}, w, Q)$ for which the polarisation $B_Q(-,-) \cong \Hom(-, \D(-))$ induces a strong duality $(\D, \eta)$ on $\cal{E}$, with maps the nonsingular exact form functors. There are functors
\begin{align*}
	\excat \to \wcompcat \xto{L_{(-)}} \cat_\infty, \\
	\cal{E} \mapsto (\chb(\cal{E}), \mathbf{qis}), (\cal{E}, w) \mapsto L_w(\cal{E})
\end{align*}
where the latter is the composite $\wcompcat \subset \relcat \to \cat_\infty$, which by Proposition \ref{comstab} factors through the inclusion $\cat^{\mathrm{st}}_\infty \subset \cat_\infty$. The functor $\relcat \to \cat_\infty$ is the composite
\[
	\relcat \xto{N} \sset^+ \to \cat_\infty
\]
for $\sset^+$ the category of marked simplicial sets with marked model structure \cite[\S3]{HTT}, with underlying $\infty$-category $\cat_\infty$, and $\relcat$ the ordinary category of pairs $(\cal{C}, W)$, for $\cal{C}$ a category and $W \subset \cal{C}$ a wide subcategory, with maps $(\cal{C}, W) \to (\cal{D}, V)$ functors $f:\cal{C} \to \cal{D}$ with $f(W) \subset V$. The functor $N$ sends $(\cal{C}, W)$ to $(N(\cal{C}), N(W))$, and the second is the canonical localisation, which can be modelled for instance as fibrant replacement in $\sset^+$ \cite[\S1.1]{Hin16}.
\begin{proposition}
	There is a canonical functor $\wcompformcat \to \cat^p_\infty$, sending a complicial exact form category with weak equivalences and strong duality $(\cal{E}, Q, w, \D, \eta)$ to the Poincar\'e $\infty$-category $(L_w(\cal{E}), \mathbf{R}Q)$ defined above.
\end{proposition}
\begin{proof}
	We have functors
	\begin{align*}
		(\wcompcat)^{\op} \to \mathrm{CAT}_1 \to \mathrm{CAT}_\infty, & \quad (\cal{E}, w) \mapsto \fun^\mathrm{qlex}(\cal{E}^{\op}, \Ab), \\
	\left(\cat^{\mathrm{st}}_\infty\right)^{\op} \to \mathrm{CAT}_\infty, & \quad \cal{C} \mapsto \fun^\mathrm{q}(\cal{C}^{\op}, \mathrm{D}(\zz)),
	\end{align*}
	where $\fun^\mathrm{qlex}$ and $\fun^\mathrm{q}$ denote the full subcategories of the respective functor categories $\fun(\cal{E}^{\op}, \Ab)$ and $\fun(\cal{C}^{\op}, \mathrm{D}(\zz))$ spanned by quadratic left-exact resp.\ quadratic functors. Recall that by definition \cite[Def.\ 1.2.1]{CD23a}
	\[
		\int_{\cat_\infty^{\mathrm{st}}}\fun^\mathrm{q} = \cat_\infty^h,
	\]
	the $\infty$-category of hermitian $\infty$-categories, and write $\cat^{h, w}_{1, \mathrm{comp}}$ for the unstraightening $\int_{\wcompcat}\fun^\mathrm{qlex}$. The localisation functor $L_{(-)}:\wcompcat \to \cat_\infty^{\mathrm{st}}$ induces a map of right fibrations
	\[
		\int_{\wcompcat}L_{(-)}^*\fun^\mathrm{q} \to \cat^h_\infty,
	\]
	and composing with the functor $\cat^{h, w}_{1, \mathrm{comp}} \to \int_{\wcompcat}L_{(-)}^*\fun^\mathrm{q}$ induced on unstraightenings by the composite
	\[
		\fun^\mathrm{qlex}(\cal{E}^{\op}, \Ab) \xto{\gamma_!(-)^{\Delta^\bullet}} \fun^\mathrm{q}(L_w(\cal{E})^{\op}, \mathrm{D}_{\ge 0}(\zz)) \xto{P_2} \fun^\mathrm{q}(L_w(\cal{E})^{\op}, \mathrm{D}(\zz))
	\]
	for $\gamma:\cal{E} \to L_w(\cal{E})$ the Dwyer-Kan localisation, we obtain a functor $\cat^{h, w}_{1, \mathrm{comp}} \to \cat^h_\infty$, which on objects sends $(\cal{E}, Q) \mapsto (L_w(\cal{E}), P_2\gamma_!\qdb =:\mathbf{R}Q)$. We have a non-full embedding $\wcompformcat \subset \cat^{h, w}_{1, \mathrm{comp}}$. Given a tuple $(\cal{E}, Q, w, \D, \eta)$ in $\wcompformcat$, the pair $(L_w(\cal{E}), \mathbf{R}Q)$ is a Poincar\'e $\infty$-category by \ref{comppoi}, and given a nonsingular form functor
	\[
		(f, \varphi_q, \varphi):(\cal{E}, Q, w, \D, \eta) \to (\cal{E}', Q', w', \D', \eta'),
	\]
	the induced functor $F:L_w(\cal{E}) \to L_{w'}(\cal{E}')$ on localisations is then duality preserving since $\varphi$ is a natural weak equivalence.
\end{proof}
\subsection{Poincar\'e structures on derived $\infty$-categories}\label{poistr}
We now specialise the above to the case of form exact categories of chain complexes. Analogously to the story for $K$-theory, the passage from an exact category to a stable $\infty$-category goes via the (bounded) derived $\infty$-category $\db(\cal{E})$, as recalled in Appendix \ref{db}. Fix an exact form category with strong duality $(\cal{E}, Q, \D, \eta)$. Recall from \ref{rec} that the exact embedding $\cal{E} \hookrightarrow \chb(\cal{E})$ enhances to an exact form functor of exact form categories with weak equivalences
\[
	(\cal{E}, Q, \D, \eta, \mathbf{iso}) \to (\chb(\cal{E}), Q_\mathrm{ch}, \D, \eta, \mathbf{qis}),
\]
where the value of $Q_\mathrm{ch}$ at a bounded chain complex $X$ is the abelian group of pairs $(\xi, \varphi)$, for $\xi \in Q(X_0)$ with $d_1^\bullet(\xi) = 0 \in Q(X_1)$, and $\varphi:X \to \D(X)$ a symmetric form satisfying $\varphi_0 = \rho(\xi)$. We denote by $\tau$ and $\rho$ also the corresponding transfer and restriction maps
\[
	\Hom_{\chb(\cal{E})}(X, X)_{\fct} \xto{\tau_X} Q_\mathrm{ch}(X) \xto{\rho_X} \Hom_{\chb(\cal{E})}(X, X)^{\fct}
\]
for a bounded chain complex $X$. The exact category $\chb(\cal{E})$ has a canonical complicial structure given by an extension of the $\projz$-action on $\cal{E}$ arising from additivity, with corresponding Frobenius exact structure consisting of degreewise split monomorphisms and epimorphisms, and Frobenius equivalences the chain homotopy equivalences. The construction of the previous section then gives rise to a functor $Q_\mathrm{ch}^{\Delta^\bullet}:\chb(\cal{E})^{\op} \to \sab$ sending chain homotopies to simplicial homotopies, and we have a diagram of functors and natural transformations
\[\begin{tikzcd}
	\cal{E}^{\op} \ar[rd, hookrightarrow, "\iota_{\cal{E}}"'] \ar[rr, "Q"{name=Q}] && |[alias=ab]|\Ab \ar[r, hookrightarrow, "c"] & \bbDelta^{\op}\Ab. \\
	& |[alias=chb]| \chb(\cal{E})^{\op} \ar[ru, "Q_\mathrm{ch}"] \ar[rru, "Q_\mathrm{ch}^{\Delta^\bullet}"{swap,name=dQ}, bend right=2ex]
	\ar[Rightarrow,from=ab, to=dQ,shorten >=2mm,shorten <=.5mm, "\eta_1"]
	\ar[Rightarrow,from=Q, to=chb,shorten >=2mm,shorten <=3mm, "\eta_0"]
\end{tikzcd}\]
\begin{lemma}
	$\eta_0$ is an isomorphism, as is the whiskered transformation
	\[
		c\circ Q \stackrel{c_{\eta_0}}\Rightarrow c\circ Q_{\mathrm{ch}}\circ \iota_{\cal{E}} \stackrel{(\eta_1)_{\iota_{\cal{E}}}}\Rightarrow Q_{\mathrm{ch}}^{\Delta^\bullet}\circ \iota_{\cal{E}}.
	\]
\end{lemma}
\begin{proof}
	The first statement then follows from the definition of $Q_\mathrm{ch}$; for the second, for an object $x$ of $\cal{E}$, consider the map
	\[
		cQ(x) = cQ_\mathrm{ch}(x[0]) \to Q_\mathrm{ch}^{\Delta^\bullet}(x[0])
	\]
	induced by the unique map $\Delta^\bullet \to c\Delta^0$. The codegeneracies $s_n:\Delta^nx[0] \epi x[0]$ induce a splitting
	\[
		\Delta^nx[0] \cong \tilde{x}^n \oplus x[0]
	\]
	for each $n$, where $\tilde{x}^n := \ker(s_n)$ is contractible. For $x \in \cal{E}$ and $i \in \zz$, write $D_i(x) := C \otimes x[i-1]$ for the complex
	\[
		\dots \to 0 \to x =\joinrel=\joinrel= x \to 0 \to \dots,
	\]
	concentrated in degrees $i, i-1$.
		\begin{lemma}
		For each $n \ge 1$ we have $\tilde{x}^n \cong \bigoplus_{1 \le i \le n}D_i\left(x^{\oplus\binom{n}{i}}\right)$.
	\end{lemma}
	\begin{proof}
		This follows from the analogous statement in $\mathscr{C}_{\zz}$, i.e.\ that
		\[
			\tilde{\Delta}^n := \ker\left(N\zz[\Delta^n] \epi N\zz[\Delta^0]\right) \cong \bigoplus_{1 \le i \le n}D_i\left(\zz^{\oplus\binom{n}{i}}\right).
		\]
		The category of finitely generated projective $\zz$-modules is idempotent complete and hence so is $\mathscr{C}_{\zz}$, so the contractible chain complex $\tilde{\Delta}^n$ splits as a finite direct sum of discs $D_i(\zz)$, $\tilde{\Delta}^n \cong \bigoplus_{i \in S}D_i(\zz^{n_i})$ for some finite set $S \subset \mathbb{N}$ and $n_i \in \mathbb{N}$. In degrees $0 \le i \le n$, $\Delta^n$ is isomorphic to $\zz[(\Delta^n_i)^{\mathrm{nd}}] \cong \zz^{\binom{n+1}{i+1}}$ and is 0 elsewhere, so by descending induction on the degree $0 \le i \le n$, we have that the $i^{\mathrm{th}}$ component of the direct sum is
		\[
			D_i\left(\zz^{\oplus\sum_{i \le j \le n}(-1)^{j+i}\binom{n+1}{j+1}}\right) \cong D_i\left(\zz^{\oplus\binom{n}{i}}\right).
		\]
	\end{proof}
	Now for each $n \ge 1$ we have
	\begin{align*}
		Q_\mathrm{ch}(\Delta^nx[0]) & \cong Q_\mathrm{ch}\left(x[0] \oplus \bigoplus_{1 \le i \le n}D_i\left(x^{\oplus\binom{n}{i}}\right)\right) \\
		& \cong Q_\mathrm{ch}(x[0]) \oplus Q_\mathrm{ch}\left(\bigoplus_{1 \le i \le n}D_i\left(x^{\oplus\binom{n}{i}}\right)\right) \oplus \bigoplus_{1 \le i \le n}\Hom_{\chb(\cal{E})}\left(D_i\left(x^{\oplus\binom{n}{i}}\right), \D(x)[0]\right),
	\end{align*}
	by quadraticity. The second summand evaluates to
	\[
		Q_\mathrm{ch}\left(\bigoplus_{1 \le i \le n}D_i\left(x^{\oplus\binom{n}{i}}\right)\right) \cong \bigoplus_{1 \le i \le n}Q_\mathrm{ch}\left(D_i\left(x^{\oplus\binom{n}{i}}\right)\right) \oplus \bigoplus_{1 \le i < j \le n}\Hom_{\chb(\cal{E})}\left(D_i\left(x^{\oplus\binom{n}{i}}\right), \D\left(D_j\left(x^{\oplus\binom{n}{j}}\right)\right)\right) = 0,
	\]
	since
	\[
		\Hom_{\chb(\cal{E})}(D_i(y), \D(D_j(z))) = \Hom_{\chb(\cal{E})}(D_i(y), D_{1-j}(\D(z))) =
			\begin{cases}
				\Hom_{\cal{E}}(y, \D(z)), & j = -i, -i+1, \\
				0, & \text{ else,}
			\end{cases}
	\]
	for each $y, z \in \cal{E}$, and
	\begin{align*}
		Q_\mathrm{ch}(D_i(y)) = & \{(\xi, \varphi) \mid \xi \in Q(D_i(y)_0), \varphi \in \Hom_{\chb(\cal{E})}(D_i(y), \D(D_i(y)))^{\fct}, \rho(\xi) = \varphi_0, d_1^\bullet(\xi) = 0\} \\
		\cong & \begin{cases}
			Q(y), & i = 0, \\
			0, & \text{ else.}
		\end{cases}
	\end{align*}
	Similarly, the third term evaluates to
	\[
		\bigoplus_{1 \le i \le n}\Hom_{\chb(\cal{E})}\left(D_i\left(x^{\oplus\binom{n}{i}}\right), \D(x)[0]\right) = 0,
	\]
	since
	\[
		\Hom_{\chb(\cal{E})}\left(D_i\left(x^{\oplus\binom{n}{i}}\right), \D(x)[0]\right) = \begin{cases} \Hom_{\chb(\cal{E})}\left(x, \D(x)\right), & i=0, \\ 0, & \text{ else.}\end{cases}
	\]
	The composite $Q_\mathrm{ch}(x[0]) \to Q_\mathrm{ch}(\Delta^nx[0]) \cong Q_\mathrm{ch}(\tilde{x}^n) \oplus Q_\mathrm{ch}(x[0]) \oplus \Hom_{\chb(\cal{E})}(\tilde{x}^n, \D(x))$ is
	\[
		\begin{psmallmatrix}
			Q_\mathrm{ch}(\iota_0) \\ Q_\mathrm{ch}(\iota_1) \\ \tau(\D(\iota_1)\rho(-)\iota_0)
		\end{psmallmatrix} \circ Q_\mathrm{ch}(\pi_1) = \begin{psmallmatrix}
			0 \\ 1 \\ 0
		\end{psmallmatrix},
	\]
	and clearly this furnishes an isomorphism $Q(x) \cong Q_\mathrm{ch}(x[0]) \cong Q_\mathrm{ch}(\Delta^nx[0])$. But the above map coincides with the degeneracy $Q_\mathrm{ch}^{\Delta^0}(x[0]) \to Q_\mathrm{ch}^{\Delta^n}(x[0])$, and hence the levelwise isomorphism above assembles into an isomorphism of simplicial abelian groups
	\[
		cQ(x) \cong \tqdb(x[0]).
	\]
\end{proof}
In the sequel we write $Q$ for both the functor $\cal{E}^{\op} \to \Ab$ and its extension to $\chb(\cal{E})$. By Corollary \ref{qdb2exc}, we have that $\qdb$ descends to a $\mathrm{D}_{\ge 0}(\zz)$-valued presheaf on the stable $\infty$-category $L_\mathrm{Frob}(\chb(\cal{E})) = \kb(\cal{E})$ which is 2-excisive, and admits an essentially unique 2-excisive $\mathrm{D}(\zz)$-valued delooping which we denote $\Qoppa : \kb(\cal{E})^{\op} \to \mathrm{D}(\zz)$. In general, this delooping will be non-connective; we claim however that for coconnective complexes $X$ the natural map $\qdb(X) \to \Qoppa(X)$, the unit of the adjunction $P_2 \dashv i$ of \S\ref{dadd}, is an equivalence. Observe firstly that for such $X$, the mapping spectrum $\map(X, \D(X))$ is connective, since $\D(X)$ is concentrated in nonnegative degrees.\footnote{This is a consequence of the natural weight structure on $\kb(\cal{E})$; see \cite[Ex.\ 3.1.3(i)]{HS23}.} The fibre $F$ of the map
\[
	\Omega\qdb(\Omega X) \to \Omega \tau_{\ge 0}\Map_\Delta(\Omega X, \D(\Omega X))
\]
in $\mathrm{D}(\zz)$ sits in a long exact sequence
\begin{align*}
	\dots & \to \pi_0F \to \pi_1\qdb(\Omega X) \to \pi_1\Map(\Omega X, \D(\Omega X)) \to \pi_{-1}F \to \pi_0\qdb(\Omega X) \to \dots \\
	\dots & \to \pi_0\Map(\Omega X, \D(\Omega X)) \to \pi_{-2}F \to \pi_{-1}\qdb(X) \to 0.
\end{align*}
Now $\Map(\Omega X, \D(\Omega X)) \simeq \Omega^2\Map(X, \D(X))$ is 2-connective, and $\pi_0\qdb(\Omega X) = 0$ since $Q(\Omega X) = 0$ for coconnective $X$. Accordingly, the map $F \to \qdb(X)$ is a $\pi_*$-isomorphism and hence an equivalence of spectra, and in the sequential colimit, we obtain $\qdb(X) \simeq P_2\qdb(X)$ for coconnective complexes $X$. In particular, $\Qoppa\mid_{\cal{E}^{\op}} \simeq HQ$, for $H : \Ab \to \mathrm{D}(\zz)$ the Eilenberg-MacLane embedding.
\begin{corollary}\label{uni}
	$\Qoppa:\kb(\cal{E})^{\op} \to \mathrm{D}(\zz)$ is the essentially unique quadratic extension of $Q:\cal{E}^{\op} \to \Ab$. Moreover, the connective cover $\tau_{\ge 0}\Qoppa(X)$ satisfies
	\[
		\tau_{\ge 0}\Qoppa(X) \simeq \underset{J^{\op}_{X, \mathrm{Frob}}}\colim\ Q
	\]
	in $\mathrm{D}_{\ge 0}(\zz)$.
\end{corollary}
\begin{proof}
	By above, $\Qoppa$ is a quadratic extension of $HQ$; uniqueness follows from \cite[Th.\ 2.19]{BGMN22}.
\end{proof}
\begin{remark}[Cf.\ {\cite[Rem.\ 4.2.19]{CD23a}}]
	It can be shown that upon restriction to coconnective chain complexes, the derived functor of $Q$ is given by the Dold-Puppe non-additive derived functor \cite{DP58}
	\[
		\wchb(\cal{E})^{\op} \hookrightarrow (\bbDelta\cal{E})^{\op} \simeq \bbDelta^{\op}\cal{E}^{\op} \to \sab \to \mathrm{D}(\zz),
	\]
	where the first functor is furnished by Dold-Kan, and is fully faithful in general (and upon replacing the source with all connective chain complexes, an equivalence for $\cal{E}$ weakly idempotent complete; see for instance \cite[Th.\ 1.2.3.7]{HA}). A computation shows that for coconnective $X$ there is an equivalence
	\[
		\qdb(X) \simeq \underset{[n] \in \bbDelta^{\op}}\colim\ Q(X^n),
	\]
	for $\underset{[n] \in \bbDelta}\lim\ X^n \simeq X$ a cosimplicial resolution of $X$, and so for $X \in \wchb(\cal{E})$, we recover
	\[
		\Qoppa(X) \simeq \qdb(X).
	\]
\end{remark}
Passage to the derived $\infty$-category of $\cal{E}$ is then given by the Verdier projection $\pi:\kb(\cal{E}) \to \db(\cal{E})$ at the subcategory of acyclics $\acb(\cal{E})$; left Kan extending $\Qoppa$ along $\pi$, we obtain a hermitian $\infty$-category $(\db(\cal{E}), \pi_!\Qoppa)$.
\begin{theorem}\label{dbpoi}
	For $(\cal{E}, Q, \D, \eta)$ an exact form category with strong duality, the pair $(\db(\cal{E}), \pi_!\Qoppa)$ is the data of a Poincar\'e $\infty$-category. Moreover, the derived functor $\pi_!\Qoppa$ satisfies
	\[
		\tau_{\ge0}\pi_!\Qoppa(X) \simeq \underset{J_X^{\op}}\colim\ Q
	\]
	in $\mathrm{D}_{\ge 0}(\zz)$.
\end{theorem}
\begin{proof}
	This follows from Theorem \ref{comppoi} (and Remark \ref{ddual}), in light of the fact that the duality on $\cal{E}$ induces a strong duality on $\chb(\cal{E})$ and hence on both $\kb(\cal{E})$ and $\db(\cal{E})$.
\end{proof}
\section{Grothendieck-Witt spaces}\label{comp}
In this section we construct a natural equivalence
\[
	\cal{GW}(\cal{E}, Q, w) \xto{\simeq} \cal{GW}(L_w(\cal{E}), \mathbf{R}Q),
\]
for $(\cal{E}, Q, \D, \eta, w)$ a complicial exact form category with weak equivalences and strong duality, and associated Poincar\'e category $(L_w(\cal{E}), \mathbf{R}Q)$. The argument builds on that of \cite[App.\ B.2]{CD23a}: since the Grothendieck-Witt space in both the classical and Poincar\'e setting is constructed as the fibre of a map between realisations of certain simplicial spaces, it is enough to exhibit a levelwise equivalence between the latter. The comparison then reduces to an analysis of the underlying infinite loopspace of the right derived quadratic functor $\mathbf{R}Q$ of the previous section. We begin in \S\ref{grothendieck-witt-theory} by recalling the constructions of the Grothendieck-Witt spaces in each setting. We then proceed in \S\ref{poinob} with a comparison of the spaces of nondegenerate forms associated to an exact form category and to a Poincar\'e $\infty$-category, which in \S\kern-1pt\S\ref{sdot} we extend to the hermitian $S_\bullet$- and $\mathcal{Q}_\bullet$-constructions. The comparison then follows easily in \S\ref{gwcomp}.
\subsection{Grothendieck-Witt theory}\label{grothendieck-witt-theory}
\subsubsection{$\cal{GW}$ for exact categories}\label{cgw}
In this section we recall the construction of the Grothendieck-Witt space associated to an exact form category with weak equivalences and strong duality; see \cite[\S6]{Sch21}, \cite[\S7]{Sch24} for details.\\
To an exact form category with weak equivalences $(\cal{E}, Q, \D, \eta, w)$, we may associate the category of quadratic spaces $\wquad(\cal{E}, Q, w)$. This is the Grothendieck construction on the functor $Q_\mathrm{nd}$ of Remark \ref{poi},
\[
	Q_\mathrm{nd}:w\cal{E}^{\op} \to \set \subset \cat, x \mapsto \{\xi \in Q(x) \mid \rho_x(\xi) \in w\cal{E}\},
\]
classifying the right fibration $\wquad(\cal{E}, Q, w) \to w\cal{E}$, $(x, q) \mapsto x$. Explicitly, objects are pairs $(x, \xi)$ for $x$ an object of $\cal{E}$ and $\xi \in Q(x)$ such that $\rho_x(\xi) \in \Hom_{\cal{E}}(x, \D(x))$ is a weak equivalence, and maps $(x, \xi) \to (y, \zeta)$ are weak equivalences $f:x \to y$ in $\cal{E}$ satisfying $f^\bullet(\zeta) = \xi$.\\
Suppose now given a poset $(\mathcal{P}, \le)$ with strict duality $\mathcal{P}^{\op} \to \mathcal{P}$, $i \mapsto \overline{i}$. For an exact form category with weak equivalences and strong duality $(\cal{E}, Q, \D, \eta, w)$, the category $\fun(\mathcal{P}, \cal{E})$ acquires the structure of an exact form category with weak equivalences and strong duality as follows. The dual of a functor $X : \mathcal{P} \to \cal{E}$, $i \mapsto X_i$ is
\[
	\D_\mathcal{P}(X) : i \mapsto \D(X_{\overline{i}}),
\]
with double dual identification given pointwise by $\eta$. A natural transformation $X \to Y$ is a weak equivalence if it is so pointwise. A quadratic form on a functor $X:\mathcal{P} \to \cal{E}$ is the data of a pair $((\xi_i)_{i \le \overline{i}}, \varphi)$, for $(\xi_i)_{i\le \overline{i}} \in \prod_{i \le \overline{i}}Q(X_i)$ and $\varphi : X \to \D_\mathcal{P}(X)$ a symmetric form on $X$, satisfying $X_{i \le j}^\bullet(\xi_j) = \xi_i$ for each pair $i, j$ satisfying $i \le j \le \overline{j} \le \overline{i}$, and $\rho_{X_i}(\xi_i) = \varphi_{\overline{i}} \circ X_{i \le \overline{i}}$. The set $Q_{\mathcal{P}}(X)$ of such forms is an abelian group with $(\{\xi_i\}_{i \le \overline{i}}, \varphi) + (\{\eta_j\}_{j \le \overline{j}}, \psi) = (\{\xi_i + \zeta_i\}_{i \le \overline{i}}, \varphi + \psi)$. Informally, a form on $X$ is a compatible family of pointwise forms indexed by the subposet of elements $i \in \mathcal{P}$ with $i \le \overline{i}$, and a symmetric form on $X$ compatible with the family of symmetric forms coming from the $\xi_i$; we shall see below that in the case $\mathcal{P} = \mathrm{Ar}([n])$, the induced structure of an exact form category on $S_n(\cal{E})$ is such that the map $\varphi$ is determined uniquely by the $\xi_i$. Functoriality is inherited from that of $Q$ on $\cal{E}$: for $\theta:X \to Y$ a natural transformation, we define
\[
	Q_{\mathcal{P}}(Y) \to Q_{\mathcal{P}}(X), \quad (\{\xi_i\}_{i \le \overline{i}}, \varphi) \mapsto (\{\theta^\bullet(\xi_i)\}_{i \le \overline{i}}, \D(\theta)\varphi\theta).
\]
Associated to each functor $X$ is the $C_2$-Mackey functor
\begin{align*}
	& \mathrm{Nat}(X, \D_\mathcal{P}(X)) \xto{\tau_X} Q(X) \xto{\rho_X} \mathrm{Nat}(X, \D_\mathcal{P}(X)) \\
	& \tau_X : \varphi \mapsto (\{\tau(\varphi_{\overline{i}}\circ X_{i \le \overline{i}})\}_{i \le \overline{i}}, \varphi + \sigma(\varphi)), \\
	& \rho_X : (\{\xi_i\}_{i \le \overline{i}}, \varphi) \mapsto \varphi,
\end{align*}
from which one may check that $Q_\mathcal{P}$ is quadratic by the criteria of \cite[Lem.\ A.10(ii)]{Sch21}.\\
Recall that the category $\mathrm{Ar}([n])$ has a unique strict duality $(i \le j) \mapsto (n-j, n-i)$; for $\cal{E}$ as above, the category $\fun(\mathrm{Ar}([n]), \cal{E})$ thus inherits the structure of an exact form category with weak equivalences and strong duality. To ease notation, we write $X^{i \le j}_{k \le l}$ for the image under $X$ of the map $(i \le j) \le (k \le l)$, and $Q_n := Q_{\mathrm{Ar}([n])}$. Restricting to the subcategory of functors $X : \mathrm{Ar}([n]) \to \cal{E}$ such that $X_{i \le i} \cong 0$, and such that for each $0 \le i \le j \le k \le n$ the sequence
\[
	X_{i \le j} \mono X_{i \le k} \epi X_{j \le k}
\]
is a conflation in $\cal{E}$, we obtain an exact form category with weak equivalences and strong duality $(S_n(\cal{E}), Q, \D, \eta, w)$. Denote by $\mathcal{P}_n \subset \mathrm{Ar}([n])$ the subposet of arrows $(i \le j)$ satisfying
\[
	(i \le j) \le \overline{(i \le j)} = (n-j \le n-i),
\]
i.e.\ with $i + j \le n$. A quadratic form on $X \in S_n(\cal{E})$ is a pair $(\{\xi_{i \le j}\}_{i+j \le n}, \varphi)$, with the family $(\xi_{i \le j})$ an object in the limit $\lim_{(i \le j) \in \mathcal{P}_n^{\op}} Q$, satisfying
\[
	\rho_{X_{i \le j}}(\xi_{i \le j}) = \varphi_{n-j \le n-i} \circ X^{i \le j}_{(n-j \le n-i)}.
\]
In particular we have $\rho_{X_{0 \le n}}(\xi_{0\le n}) = \varphi_{0 \le n}$; as noted in \cite[Rem.\ 8.2.5(i)]{HS23}, $\varphi$ is determined by the $\xi_{i \le j}$: for each $0 \le i \le n$ we have a map of conflations
\[\begin{tikzcd}
		X_{0\le i} \ar[rr, rightarrowtail, "X^{0 \le i}_{0 \le n}"] \ar[d, "\varphi_{0 \le i}"] && X_{0\le n} \ar[rr, twoheadrightarrow, "X^{0 \le n}_{i \le n}"] \ar[d, "\varphi_{0 \le n}"] && X_{i\le n} \ar[d, "\varphi_{i \le n}"] \\
		\D(X_{n-i\le n}) \ar[rr, rightarrowtail, "\D(X^{0 \le n}_{n-i \le n})"] && \D(X_{0\le n}) \ar[rr, twoheadrightarrow, "\D(X^{0 \le n-i}_{0 \le n})"] && \D(X_{0\le n-i}),
\end{tikzcd}\]
from which we see that $\varphi_{0 \le i}$ and $\varphi_{i \le n}$ are determined by $\varphi_{0 \le n}$, and repeating this with the map of conflations associated with $X_{0 \le i} \mono X_{0 \le j} \epi X_{i \le j}$, similarly for $\varphi_{i \le j}$. Accordingly, we have an injection
\[
	Q_n(X) \hookrightarrow \lim_{(i \le j) \in \mathcal{P}_n^{\op}} Q \circ X^{\op}, \quad ((\xi_i)_i, \varphi) \mapsto (\xi_i)_i.
\]
\begin{construction}
	Contravariant naturality of the $S_\bullet$-construction yields a simplicial exact category $S_\bullet(\cal{E})$. The simplicial structure maps of $S_\bullet(\cal{E})$ are however not compatible with the dualities, but are if we consider the edgewise subdivision $S^e_\bullet(\cal{E})$, obtained by precomposing with the functor $\mathrm{sd}:\bbDelta^{\op}\to\bbDelta^{\op}$, $[n] \mapsto [n]^{\op} \star [n] \cong [2n+1]$ (this has the effect of symmetrising faces and degeneracies). We thus obtain a simplicial exact form category with weak equivalences and strong duality, and a functor
	\[
		S_\bullet : \wformcat \to \bbDelta^{\op}\kern-1pt\wformcat, \quad (\cal{E}, Q, \D, \eta, w) \mapsto (S^e_\bullet(\cal{E}), Q^e_\bullet, \D^e_\bullet, \eta^e_\bullet, w).
	\]
	Postcomposing with $\wquad$ and taking nerves, we have a simplicial space
	\[
		[p] \mapsto \left|\wquad(S_{2p+1}(\cal{E}), Q_{2p+1}, w)\right|
	\]
    The forgetful functors
	\[
		\wquad(S_n\cal{E}, Q_n, w) \to wS_n(\cal{E})
	\]
	induce a map of simplicial spaces
	\[
		[p] \mapsto \left(\left|\wquad(S_{2p+1}(\cal{E}), Q_{2p+1}, w)\right| \to \left|wS_{2p+1}(\cal{E})\right| \to \left|wS_p(\cal{E})\right|\right),
	\]
    where the last map is induced levelwise by the natural transformation $1_{\bbDelta} \Rightarrow \mathrm{sd}$, $[n] \subset [n]^{\op}\star[n]$, $i \mapsto i$. The Grothendieck-Witt space \cite[Def.\ 6.3]{Sch21} is defined to be the homotopy fibre of the corresponding map of geometric realisations
	\[
		\cal{GW}(\cal{E}, Q, w) \to \left|\wquad(S^e_\bullet(\cal{E}), Q^e_\bullet, w)\right| \to \left|wS_\bullet(\cal{E})\right|.
	\]
\end{construction}
\begin{remark}\label{subd}
	The constructions in \cite{Sch21} and \cite{Sch24} make use of the functor
	\[
		\mathrm{sd}_0:\bbDelta \to \bbDelta, \quad [n] \mapsto [n]^{\op} \star [n],
	\]
	while the authors of \cite{CD23b} work with the $\mathcal{Q}_\bullet$-construction corresponding to the subdivision
	\[
		\mathrm{sd}_1:\bbDelta \to \bbDelta, \quad [n] \mapsto [n] \star [n]^{\op}.
	\]
    Here we write $(-)^{\op}:\bbDelta^{\op} \to \bbDelta^{\op}$ for the involution $[n] \mapsto [n]^{\op}$, $\theta \mapsto \theta^{\op}$, where $[n]^{\op} = (n' < (n-1)' < \dots < 1' < 0')$ and $\theta^{\op}:[m]^{\op} \to [n]^{\op}$ is $i' \mapsto \theta(i)'$. Given a map $\theta : [m] \to [n]$, $\mathrm{sd}_0(\theta)$ (resp.\ $\mathrm{sd}_1(\theta)$) is the map $\theta^{\op}\star \theta$ (resp.\ $\theta \star \theta^{\op}$).
    There is an identification of simplicial spaces
    \[
        \left([n] \mapsto  \left|\wquad(S_{\mathrm{sd}_0[n]}(\cal{E}), Q_{\mathrm{sd}_0[n]})\right|\right) = \left([n] \mapsto  \left|\wquad(S_{\mathrm{sd}_1[n]}(\cal{E}), Q_{\mathrm{sd}_1[n]})\right|\right)^{\op},
    \]
    where the opposite of a simplicial space $X$ is given by precomposition with $(-)^{\op}$. Accordingly, since $(-)^{\op}$ is cofinal (see for instance \cite[\S2]{Bar13}), we have equivalences
    \[\begin{tikzcd}
         \left|\wquad(\mathrm{sd}_0^*(S_\bullet(\cal{E}), Q_\bullet))\right| \ar[r] \ar[d, "\rotatebox{270}{$\simeq$}"] & \left|wS_\bullet(\cal{E})\right| \ar[d, "\rotatebox{270}{$\simeq$}"] \\
         \left|\wquad(\mathrm{sd}_1^*(S_\bullet(\cal{E}), Q_\bullet))\right| \ar[r] & \left|wS_\bullet(\cal{E})^{\op}\right|
    \end{tikzcd}\]
    and an induced equivalence on fibres.
\end{remark}
\subsubsection{$\cal{GW}$ for stable $\infty$-categories}\label{mgw}
Fix a Poincar\'e $\infty$-category $(\cal{C}, \Qoppa)$; we have a presheaf of spaces $\Omega^\infty\Qoppa$ on $\cal{C}$, associating to each $x \in \cal{C}$ the space of forms on $x$, the cartesian unstraightening of which is denoted $\mathrm{He}(\cal{C}, \Qoppa)$. Call a form $\xi\in \Omega^\infty\Qoppa(x)$ \textit{nondegenerate} if the symmetric form associated via the natural transformation $\Qoppa(x) \to B_\Qoppa(x, x)$ is an equivalence. There is a subfunctor inclusion
\begin{equation}\begin{tikzcd}\label{qnd}
	\Qoppa_\mathrm{nd} \subset \Omega^\infty\Qoppa : \cal{C}^{\op} \to \cal{S},
\end{tikzcd}\end{equation}
with $\Qoppa_\mathrm{nd}(x)$ the subspace spanned by the nondegenerate forms on $x$, i.e.\ the pullback
\begin{equation}\begin{tikzcd}\label{qndg}
	\Qoppa_\mathrm{nd}(x) \ar[r] \ar[d] & \Omega^\infty\Qoppa(x) \ar[d] \\
	\Map_{\cal{C}^\simeq}(x, \D(x)) \ar[r] & \Map_{\cal{C}}(x, \D(x)).
\end{tikzcd}\end{equation}
The inclusion $\Qoppa_\mathrm{nd} \subset \Omega^\infty\Qoppa$ gives rise to a map of right fibrations over $\cal{C}$
\[
	\int_{\cal{C}}\Qoppa_\mathrm{nd} \to \mathrm{He}(\cal{C}, \Qoppa),
\]
and pulling back along the inclusion $\cal{C}^\simeq \subset \cal{C}$, naturality of the straightening-unstraightening equivalence yields a square in $\cat_\infty$ 
\begin{equation}\begin{tikzcd}\label{unstrsq}
	\int_{\cal{C}^\simeq}\Qoppa_\mathrm{nd} \ar[r] \ar[d] & \mathrm{He}(\cal{C}^\simeq, \Qoppa) \ar[d] \\
	\int_{\cal{C}}\Qoppa_\mathrm{nd} \ar[r] & \mathrm{He}(\cal{C}, \Qoppa)
\end{tikzcd}\end{equation}
commuting up to canonical equivalence.
\begin{lemma}\label{unstr2}
	Suppose given a functor $F: \cal{D}^{\op} \to \cal{S}$ classifying the right fibration $\int_{\cal{D}}F$, and denote by $\iota : \cal{D}^\simeq \subset \cal{D}$ the inclusion of the groupoid core. Then there is a natural equivalence of spaces
	\[
		\left(\int_{\cal{D}}F\right)^\simeq \xto{\simeq} \iota^*\int_{\cal{D}}F = \int_{\cal{D}^\simeq}\iota^*F.
	\]
\end{lemma}
\begin{proof}
	That $\iota^*\int_{\cal{D}}F$ is a space follows from the fact that a right fibration over a Kan complex is a Kan fibration. The pullback square
	\[\begin{tikzcd}
		\iota^*\int_{\cal{D}}F \ar[r] \ar[d] & \int_{\cal{D}}F \ar[d] \\
		\cal{D}^\simeq \ar[r, hookrightarrow, "\iota"] & \cal{D}
	\end{tikzcd}\]
	induces an essentially unique map $\left(\int_{\cal{D}}F\right)^\simeq \to \iota^*\int_{\cal{D}}F$, and given a space $X$, we observe that this induces an equivalence
	\begin{align*}
		\Map_{\cal{S}}(X, \left(\int_{\cal{D}}F\right)^\simeq) \simeq & \Map_{\cat_\infty}(X, \int_{\cal{D}}F) \\
		\simeq & \Map_{\cat_\infty}(X, \int_{\cal{D}}F) \times_{{\Map_{\cat_\infty}(X, \cal{D})}} \Map_{\cat_\infty}(X, \cal{D}^\simeq) \\
		\simeq & \Map_{\cal{S}}(X, \iota^*\int_{\cal{D}}F),
	\end{align*}
	since the composite $X \to \int_{\cal{D}}F \to \cal{D}$ necessarily factors through $\iota:\cal{D}^\simeq \hookrightarrow \cal{D}$.
\end{proof}
The upper-right corner in (\ref{unstrsq}) thus identifies with the maximal subgroupoid $\mathrm{Fm}(\cal{C}, \Qoppa) \subset \mathrm{He}(\cal{C}, \Qoppa)$ of \cite[Def.\ 2.1.1]{CD23a}, and the upper left with the subgroupoid $\mathrm{Pn}(\cal{C}, \Qoppa) \subset \mathrm{Fm}(\cal{C}, \Qoppa)$ spanned by pairs $(x, \xi)$, for $\xi \in \Omega^\infty\Qoppa(x)$ nondegenerate. $\mathrm{Pn}$ assembles into a functor $\cat^p_\infty \to \cal{S}$ (note that the restriction to duality-preserving functors in $\cat^p_\infty$ is necessary for this) which is moreover corepresentable (\cite[Prop.\ 4.1.3]{CD23a}) and hence preserves limits, and by \cite[Prop.\ 6.1.8]{CD23a} filtered colimits.\\
For an $\infty$-category $K$, denote by $\twar(K)$ the twisted arrow category of $K$, with $n$-simplices
\[
	\twar_n(K) = \Hom_{\sset}(\Delta^n \star (\Delta^n)^{\op}, K). 
\]
We may equip $\fun(\twar(K), \cal{C})$ with a quadratic functor
\[
	\Qoppa_K = \underset{\twar(K)^{\op}}\lim\ \Qoppa,
\]
and restricting this to the subcategory spanned by functors $F$ for which the squares
\[\begin{tikzcd}
	F(i \to l) \ar[r] \ar[d] & F(i \to k) \ar[d] \\
	F(j \to l) \ar[r] & F(j \to k)
\end{tikzcd}\]
are exact in $\cal{C}$ for each map $\Delta^3 \to K$, $i \to j \to k \to l$, we obtain a hermitian $\infty$-category $(\mathcal{Q}_K(\cal{C}, \Qoppa), \Qoppa_K)$. In general, this is not Poincar\'e, but 
upon restricting along the inclusion $\bbDelta \subset \cat_\infty$, we obtain a simplicial Poincar\'e $\infty$-category, denoted $(\mathcal{Q}_\bullet(\cal{C}), \Qoppa_n)$. As in the classical setting, there is a forgetful functor
\[
	\mathrm{Pn}(\mathcal{Q}_n(\cal{C}), \Qoppa_n) \to \mathcal{Q}_n(\cal{C})^{\simeq},
\]
and the Grothendieck-Witt space is equivalent \cite[Cor.\ 4.1.7]{CD23b} to the fibre
\[
	\cal{GW}(\cal{C}, \Qoppa) \to |\mathrm{Pn}(\mathcal{Q}_\bullet(\cal{C}), \Qoppa_\bullet)| \to |\mathcal{Q}_\bullet(\cal{C})^\simeq|.
\]
\begin{remark}
    As in \cite[App.\ B]{CD23b}, it will be convenient to reformulate the above constructions in terms of the hermitian $\mathscr{S}^e_\bullet$-construction: there is a functor
    \[
    	\iota_n:\twar(\Delta^n) \to \mathrm{Ar}(\Delta^n \star (\Delta^n)^{\op})
    \]
    sending $(i \le j)$ to $(i \le j')$, where we write the vertices of $\Delta^n \star (\Delta^n)^{\op}$ as
    \[
    	(0 < \dots < (n-1) < n < n' < (n-1)' < \dots < 0'),
    \]
    which is well known to induce a levelwise equivalence of simplicial stable $\infty$-categories
    \[
    	\iota^*_n:\mathscr{S}^e_\bullet(\cal{C}) \to \mathcal{Q}_\bullet(\cal{C}).
    \]
    We equip $\mathscr{S}^e_n(\cal{C}) = \mathscr{S}_{2n+1}(\cal{C})$ with a hermitian structure as follows. Write $\mathcal{P}_n \subset \mathrm{Ar}(\Delta^n)$ for the subposet on pairs $(i \le j)$ such that $i+j \le n$. Under the isomorphism
    \[
        \mathrm{Ar}(\Delta^{2n+1}) \cong \mathrm{Ar}(\Delta^n \star (\Delta^n)^{\op}),
    \]
    $\mathcal{P}_{2n+1}$ corresponds to the subposet on pairs $i \le j$ or $k \le l'$, for $0 \le k \le l \le n$. The poset inclusion $\twar(\Delta^n) \to \mathrm{Ar}(\Delta^n \star (\Delta^n)^{\op})$ thus factors as
    \[
        \twar(\Delta^n) \stackrel{j_n}\subset \mathcal{P}_{2n+1} \subset \mathrm{Ar}\Delta^n \star (\Delta^n)^{\op})
    \]
    and we claim that $j_n$ is cofinal. Objects of $((i \le j) \downarrow j_n)$ are maps $(i \le j) \to (k \le l')$, for some $0 \le k \le l \le n$, and likewise objects of $((i \le j') \downarrow j_n)$ are maps $(i \le j') \to (k \le l')$. It is easy to see that the former category has initial object $(i \le j) \to (i \le n')$, and the latter $(i \le j') =\joinrel= (i \le j')$. The claim follows by Quillen's theorem A.\\
    For each $n \geq 0$ we equip $\mathscr{S}^e_n(\cal{C})$ with quadratic functor
    \begin{equation}\label{twarqn}
    	X \mapsto \underset{\twar(\Delta^n)^{\op}}\lim\ \Qoppa \circ (\iota_n^*X)^{\op} \simeq \underset{\mathcal{P}_{2n+1}^{\op}}\lim\ \Qoppa \circ X^{\op}
    \end{equation}
    obtaining a Poincar\'e category $(\mathscr{S}_n^e(\cal{C}), (\iota_n^*)^*\Qoppa_n)$ tautologically Poincar\'e equivalent to $(\mathcal{Q}_n(\cal{C}), \Qoppa_n)$ via $\iota^*_n$.
\end{remark}
\subsection{Categories of Poincar\'e objects}\label{poinob}
Given a complicial exact form category with weak equivalences and strong duality $(\cal{E}, Q, \D, \eta, w)$, write $(L_w(\cal{E}), \mathbf{R}Q)$ for the associated Poincar\'e $\infty$-category of \S\ref{dquad}. There is a functor $\wquad(\cal{E}, Q, w) \to \mathrm{Pn}(L_w(\cal{E}), \mathbf{R}Q)$ induced by naturality of unstraightening: the localisation $\gamma:\cal{E} \to L_w(\cal{E})$ induces a map of right fibrations
\[\begin{tikzcd}
	\int_{w\cal{E}}\gamma^*\mathbf{R}Q_\mathrm{nd} \ar[r] \ar[d] & \mathrm{Pn}(L_w(\cal{E}), \mathbf{R}Q) \ar[d] \\
	w\cal{E} \ar[r, "\gamma"] & L_w(\cal{E})^\simeq,
\end{tikzcd}\]
where the left-hand vertical arrow is classified by the restriction $\gamma^*\mathbf{R}Q_\mathrm{nd}:w\cal{E}^{\op} \to \cal{S}$, and composing with the map in $\mathrm{RFib}(\cal{E})$ induced by the natural transformation $Q_\mathrm{nd} \Rightarrow \gamma^*\mathbf{R}Q_\mathrm{nd}$ then gives a map
\[
	\wquad(\cal{E}, Q, w) \to \mathrm{Pn}(L_w(\cal{E}), \mathbf{R}Q).
\]
The following proof is a generalisation of the techniques of \cite[Prop.\ B.2.3]{CD23b}.
\begin{proposition}\label{lvl}
	For $(\cal{E}, Q, \D, \eta, w)$ a complicial exact form category with weak equivalences and strong duality, the functor $\wquad(\cal{E}, Q, w) \to \mathrm{Pn}(L_w(\cal{E}), \mathbf{R}Q)$ induces an equivalence upon realisation.
\end{proposition}
\begin{proof}
		Consider the square
	\begin{equation}\label{fibs}\begin{tikzcd}
		\wquad(\cal{E}, Q, w) \ar[r] \ar[d] & \mathrm{Pn}(L_w(\cal{E}), \mathbf{R}Q) \ar[d] \\
		w\cal{E} \ar[r, "\gamma"] & L_w(\cal{E})^{\simeq},
	\end{tikzcd}\end{equation}
	where each vertical map is a right fibration. By \cite[Cor.\ 3.3.4.6]{HTT} the homotopy type of the total space of a right fibration is given by the colimit in spaces over the classified functor, and so we have an induced map
	\[
		\underset{w\cal{E}^{\op}}\colim\ Q_\mathrm{nd} \to \underset{L_w(\cal{E})^{\simeq, \op}}\colim\ \mathbf{R}Q_\mathrm{nd}.
	\]
	The functor $w\cal{E}^{\op} \to L_w(\cal{E})^{\simeq, \op}$ is cofinal by \cite[Cor.\ 7.6.9]{Cis19}, and so it suffices to consider the map
	\[
		\underset{w\cal{E}^{\op}}\colim\ Q_\mathrm{nd} \to \underset{w\cal{E}^{\op}}\colim\ \gamma^*\mathbf{R}Q_\mathrm{nd}
	\]
	induced by the natural transformation $Q_\mathrm{nd} \Rightarrow \gamma^*\mathbf{R}Q_\mathrm{nd}$. Write as before $\widetilde{\mathbf{R}}Q$ resp.\ $\mathbf{R}Q$ for the right derived quadratic $\mathrm{D}(\zz)$-valued functor of $Q$ at $w_\mathrm{Frob}$ resp.\ $w$. In the case $w$ is the class of Frobenius weak equivalences, the identification
	\begin{equation}\label{eq}
		\gamma^*\mathbf{R}Q_\mathrm{nd}(x) = \gamma^*\widetilde{\mathbf{R}}Q_\mathrm{nd}(x) \simeq Q^{\Delta^\bullet}_\mathrm{nd}(x) \simeq \underset{J_{x, \mathrm{Frob}}^{\op}}\colim\ Q_\mathrm{nd}
	\end{equation}
    furnished by the cofinality result Proposition \ref{hoco} exhibits the former as the left Kan extension of $Q_\mathrm{nd}$ along $\gamma$, and the result follows from Proposition \ref{global}. In the general case, we note that the pullback square (\ref{qndg}) defining $\mathbf{R}Q_\mathrm{nd}$ is stable under the filtered colimit defining $\pi_!\widetilde{\mathbf{R}}Q$, for $\pi : L_{\mathrm{Frob}}(\cal{E}) \to L_w(\cal{E})$ the Verdier projection; accordingly, we have a natural equivalence
	\[
		\pi_!\left(\widetilde{\mathbf{R}}Q_\mathrm{nd}\right) \simeq (\pi_!\widetilde{\mathbf{R}}Q)_\mathrm{nd} = (\mathbf{R}Q)_\mathrm{nd}
	\]
	and taking filtered colimits of (\ref{eq}) over $J_x^{\op}$, we obtain again that $\mathbf{R}Q_\mathrm{nd}$ is the right derived functor of $Q_\mathrm{nd}$, so again by Proposition \ref{global} we are done.
\end{proof}
\subsection{The hermitian $S_\bullet$-construction}\label{sdot}
We wish to upgrade Proposition \ref{lvl} to a levelwise equivalence of simplicial spaces
\[
	\left([s] \mapsto |\wquad(S_{2s+1}(\cal{E}), Q_{2s+1}, w)|\right) \simeq \mathrm{Pn}(\mathcal{Q}_\bullet(L_w(\cal{E})), \mathbf{R}Q_\bullet).
\]
To do so, it will suffice to exhibit the latter as the derived Poincar\'e structure of the former (along with the functoriality argument off \ref{functoriality}).
\begin{remark}
    Since the action of $\mathscr{C}(\zz)$ on $\cal{E}$ is bi-exact, for each $n \ge 0$ the pointwise complicial structure on $\fun(\mathrm{Ar}(\Delta^n), \cal{E})$ preserves the subcategory $S_n(\cal{E})$, and the tuple $(S_n(\cal{E}), Q_n, \D_n, \eta_n, w)$ is a complicial exact form category with weak equivalences and strong duality. 
\end{remark}
Write $\mathscr{S}_n(L_w(\cal{E})$ for the $\infty$-categorical $S_\bullet$-construction on $L_w(\cal{E})$, i.e.\ for the subcategory of functors $X:\mathrm{Ar}(\Delta^n) \to L_w(\cal{E})$ with $X_{i, i} \simeq 0$ for each $0 \le i \le n$ and $X_{i,j} \to X_{i,k} \to X_{j,k}$ a fibre sequence in $L_w(\cal{E})$ for each triple $0 \le i \le j \le k \le n$. Since conflations in $\cal{E}$ localise to fibre sequences in $L_w(\cal{E})$, the localisation $\cal{E} \to L_w(\cal{E})$ induces a functor
\[
    S_n(\cal{E}) \to \mathscr{S}_n(L_w(\cal{E})),
\]
which moreover exhibits $\mathscr{S}_n(L_w(\cal{E}))$ as the localisation of $S_n(\cal{E})$ at the pointwise weak equivalences, by \cite[Lem.\ B.2.4]{CD23a}. Write
\[
	(\mathbf{R}Q)_n : \mathscr{S}_n(L_w(\cal{E}))^{\op} \to \mathrm{D}(\zz), \quad X \mapsto \underset{i+j\le n}\lim\ \mathbf{R}Q(X_{i,j}),
\]
and $\mathbf{R}(Q_n)$ for the 2-excisive (right) derived functor of the functor $Q_n:S_n(\cal{E})^{\op} \to \Ab$ of \S\ref{cgw} with respect to the pointwise weak equivalences in $S_n(\cal{E})$. There is a natural transformation $\mathbf{R}(Q_n) \to (\mathbf{R}Q)_n$ induced by the universal property of the limit. Note that the inclusion
\begin{equation}\label{pn}
    \mathcal{T}_n \subset \{(i, j) \in \mathrm{Ar}(\Delta^n) \mid i+j \le n\},
\end{equation}
for $\mathcal{T}_n$ the subposet spanned by $(i, j)$ with $i+j = n-1,n$, is cofinal, by an application of Quillen's theorem A.\\
Consider for now the case $w=w_\mathrm{Frob}$. Then the localisation $S_n(\cal{E}) \to \mathscr{S}_n(L_\mathrm{Frob}(\cal{E}))$ factors as
\[
    S_n(\cal{E}) \to L_\mathrm{Frob}(S_n(\cal{E})) \to \mathscr{S}_n(L_\mathrm{Frob}(\cal{E})),
\]
for $L_\mathrm{Frob}(S_n(\cal{E}))$ the localisation of $S_n(\cal{E})$ at the class of Frobenius equivalences internal to the complicial exact category $S_n(\cal{E})$. Note that in general, the second localisation is not an equivalence; we claim however that for a class of diagrams, the corresponding map of right derived functors $\mathbf{R}_\mathrm{Frob}(Q_n) \to \mathbf{R}(Q_n)$ is an equivalence, where we write $\mathbf{R}_\mathrm{Frob}(Q_n)$ for the derived functor of $Q_n$ with respect to the internal Frobenius equivalences.
\begin{lemma}\label{ptwise}
    The natural map $\mathbf{R}_\mathrm{Frob}(Q_n)(X) \to \mathbf{R}(Q_n)(X)$ is an equivalence for each diagram $X \in S_n(\cal{E})$ in the image of an iterated degeneracy $S_1(\cal{E}) \to S_n(\cal{E})$.
\end{lemma}
\begin{proof}
    Write $\gamma:L_\mathrm{Frob}(S_n(\cal{E})) \to \mathscr{S}_n(L_\mathrm{Frob}(\cal{E}))$ for the localisation at the pointwise Frobenius equivalences. Since $\mathscr{C}_{\zz}$ acts pointwise on $S_n(\cal{E})$, Proposition \ref{hoco} implies that $\mathbf{R}_\mathrm{Frob}(Q_n)(X) \simeq (Q_n)^{\Delta^{\bullet}}(X)$, and the natural map $\mathbf{R}_\mathrm{Frob}(Q_n)(X) \to \mathbf{R}(Q_n)(X)$ identifies with
    \begin{equation}\label{frob-to-pt}
        (Q_n)^{\Delta^\bullet}(X) \to \underset{J_{X}^{\op}}\colim\ (Q_n)^{\Delta^\bullet}(X)
    \end{equation}
    for $J_X \subset (S_n(\cal{E})\downarrow X)$ the subcategory of the overcategory spanned by the pointwise Frobenius deflations over $X$. For $0 \le i \le n+1$, $s^n_i:[n] \to [1]$ for the map in $\bbDelta$ sending $0 \le j < i$ to $0$, and $i \le j \le n$ to $1$. Then for $x \in S_1(\cal{E}) = \cal{E}$, we have
    \[
        s^n_i(x)_{j, k} = \begin{cases}
            x, & 0 \le j < i, k \ge i\\
            0, & \text{ else,}
        \end{cases}
    \]
    with each nonzero map given by $1_x$. In the case $X = s^n_i(x)$ for some $x \in \cal{E}$, $(Q_n)^{\Delta^\bullet}(X)$ identifies with $\underset{\bbDelta^{\op}}\colim\ Q_n(s_i^n(\Delta^\bullet x)) = \underset{y \in J_x^{\op}}\colim\ Q_n(s_i^n(y))$, for $J_x\subset (\cal{E}\downarrow x)$ the subcategory spanned by the trivial Frobenius deflations over $x$, with the map (\ref{frob-to-pt}) induced by the functor
    \[
        (s_i^n)_x^{\op}: J_x^{\op} \to J_X^{\op}, \quad (y \stackrel{\sim}\epi x) \mapsto (s_i^n(y) \stackrel{\sim}\epi s_i^n(x)).
    \]
        For $i=0, n+1$, $X$ is zero, and the result follows trivially since each of $\mathbf{R}_\mathrm{Frob}(Q_n)$ and $\mathbf{R}(Q_n)$ is reduced. For $0 < i \le n$, $(s_i^n)_x^{\op}$ admits a left adjoint
    \[
        J_{s_i^n(x)}^{\op} \to J_x^{\op}, \quad (Y \stackrel{\sim}\epi s_i^n(x)) \mapsto (Y_{0,1} \stackrel{\sim}\epi x),
    \]
    so that $(s_i^n)_x^{\op}$ is cofinal.
\end{proof}
\begin{proposition}\label{qncomp}
	For each $n \geq 1$ the canonical map $\mathbf{R}(Q_n) \to (\mathbf{R}Q)_n$ is an equivalence of quadratic functors.
\end{proposition}
\begin{proof}
    By \cite[Lem.\ 1.1.25]{CD23a}, it suffices to exhibit a natural equivalence between connective covers. The case $n=1$ holds by the identifications $Q_1 \cong Q$, $(\mathbf{R}Q)_1 \simeq \mathbf{R}Q$. Assume firstly that $w=w_\mathrm{Frob}$ is the class of Frobenius equivalences. Then for $n \ge 2$, a diagram $X \in \mathscr{S}_n(L_w(\cal{E}))$ is determined up to contractible choice by the first row $(X_{0,1} \to X_{0,2} \to \dots \to X_{0,n})$: indeed, writing $\mathscr{I} \subset \mathrm{Ar}(\Delta^n)$ for the subcategory spanned by pairs $(i, j)$ with either $i=0$ or $i=j$, $X$ is the left Kan extension along the inclusion $\mathscr{I} \subset \mathrm{Ar}(\Delta^n)$ of the restriction $X\mid_{\mathscr{I}}$. There is a filtration
    \[
        X^1 \hookrightarrow X^2 \hookrightarrow \dots \hookrightarrow X^{n-1} \hookrightarrow X^n = X,
    \]
    for $X^k$ the diagram generated by the row $(X_{0,1} \to \dots \to X_{0,k-1} \to X_{0,k} =\joinrel= X_{0,k} =\joinrel= \dots =\joinrel= X_{0,k})$, and for each $1 \le k \le n-1$ the quotient $X^{k, k+1} := \cofib(X^k \hookrightarrow X^{k+1})$ satisfies
    \[
        \left(X^{k, k+1}\right)_{ij} =
            \begin{cases}
                X_{k,k+1}, & i \le k, j \ge k+1, \\
                0, & \text{ else,}
            \end{cases}
    \]
    i.e.\ is equivalent to $s_{k+1}^n(X_{k,k+1})$. Now $\Omega^\infty(\mathbf{R}Q)_n(X^{k,k+1}) = \Omega^\infty\mathbf{R}Q(X_{k,k+1}) = \qdb(X_{k,k+1})$, and similarly $\Omega^\infty \mathbf{R}(Q_n)(X^{k,k+1}) \simeq \qdb(X_{k,k+1})$, by Lemma \ref{ptwise}. By 2-excisivity we have identifications, for $1 le k \le n-1$,
    \[\begin{tiny}\hspace*{-2cm}
            \mathrm{fibt}\left(
            \begin{tikzcd}
                \Omega^\infty \mathbf{R}(Q_n)(X^{k+1}) \ar[r] \ar[d] & \Omega^\infty \mathbf{R}(Q_n)(X^k) \ar[d] \\
                \Map_\Delta(X^{k+1}, \D(X^k)) \ar[r] & \Map_\Delta(X^k, \D(X^k))
            \end{tikzcd}\right) \simeq \qdb(X_{k,k+1}) \simeq \mathrm{fibt}\left(
            \begin{tikzcd}
                \Omega^\infty(\mathbf{R}Q)_n(X^{k+1}) \ar[r] \ar[d] & \Omega^\infty(\mathbf{R}Q)_n(X^k) \ar[d] \\
                \Map_\Delta(X^{k+1}, \D(X^k)) \ar[r] & \Map_\Delta(X^k, \D(X^k))
            \end{tikzcd}\right).
    \end{tiny}\]
    Noting that the fibre of the lower row in each case is $\Map_\Delta(X^{k,k+1}, \D(X^k))$, we have a map of fibre sequences
    \[\begin{tikzcd}
        \qdb(X_{k,k+1}) \ar[r] \ar[d, "\rotatebox{270}{$\simeq$}"] & \fib\left(\Omega^\infty \mathbf{R}(Q_n)(X^{k+1}) \to \Omega^\infty \mathbf{R}(Q_n)(X^k)\right) \ar[r] \ar[d] & \Map_\Delta(X^{k,k+1}, \D(X^k)) \ar[d, "\rotatebox{270}{$\simeq$}"] \\
        \qdb(X_{k,k+1}) \ar[r] & \fib\left(\Omega^\infty(\mathbf{R}Q)_n(X^{k+1}) \to \Omega^\infty(\mathbf{R}Q)_n(X^k)\right) \ar[r] & \Map_\Delta(X^{k,k+1}, \D(X^k)),
    \end{tikzcd}\]
    so by induction on $k$, it suffices to show that the map $\Omega^\infty(\mathbf{R}Q)_n(X^1) \to \Omega^\infty \mathbf{R}(Q_n)(X^1)$ is an equivalence; but $X^1 = s^n_1(X_{0,1})$ is generated by the row $(X_{0,1} =\joinrel= \dots =\joinrel= X_{0,1})$, and the map in question evaluates to the equivalence
    \[
         \qdb(X_{0,1}) \xto{\simeq} \Omega^\infty\mathbf{R}Q(X_{0,1})
    \]
    if $n \le 2$; for $n > 2$, each of $\mathbf{R}(Q_n)(s^n_1(X_{0,1}))$ and $(\mathbf{R}Q)_n(s^n_1(X_{0,1}))$ are zero.\\
    For general $w$, write $\widetilde{\mathbf{R}}Q$ for the derived functor of $Q$ at the Frobenius equivalences, and $\widetilde{\mathbf{R}}(Q_n)$ for the derived functor of $Q_n$ at the pointwise Frobenius equivalences, so that by above $\widetilde{\mathbf{R}}(Q_n) \simeq (\widetilde{\mathbf{R}}Q)_n$. Now the derived functor $\mathbf{R}(Q_n)$ of $\widetilde{\mathbf{R}}(Q_n)$ at the pointwise weak equivalences is given by the left Kan extension along the localisation
    \[
        \gamma:\mathscr{S}_n(L_\mathrm{Frob}(\cal{E})) \to \mathscr{S}_n(L_w(\cal{E})),
    \]
    and we see that since the poset $\mathcal{T}$ of (\ref{pn}) is finite for each $n$, we have equivalences
    \[
        \mathbf{R}(Q_n) \simeq \gamma_!\widetilde{\mathbf{R}}(Q_n) \simeq \gamma_!(\widetilde{\mathbf{R}}Q)_n \simeq (\mathbf{R}Q)_n,
    \]
    since the colimit computing left Kan extension along $\gamma$ is filtered.
\end{proof}
Propositions \ref{lvl} and \ref{qncomp} then imply:
\begin{corollary}\label{sn}
    Let $(\cal{E}, Q, \D, \eta, w)$ be a complicial exact form category with weak equivalences and strong duality. Then for each $n \ge 1$, the map
    \[
        \wquad(S_n(\cal{E}), Q_n, w_\mathrm{pt}) \to \mathrm{Pn}(\mathscr{S}_n(L_w(\cal{E})), (\mathbf{R}Q)_n)
    \]
    induces an equivalence upon realisation.
\end{corollary}
\subsection{Comparing Grothendieck-Witt spaces}\label{gwcomp}
We are now ready to prove the main result of the present paper.
\begin{theorem}\label{main-comp}
	Let $(\cal{E}, Q, \D, \eta, w)$ be a complicial exact form category with weak equivalences and strong duality, with associated Poincar\'e category $(L_w(\cal{E}), \mathbf{R}Q)$. Then the localisation functor $\cal{E} \to L_w(\cal{E})$ induces a natural equivalence of Grothendieck-Witt spaces
	\[
		\cal{GW}(\cal{E}, Q) \to \cal{GW}(L_w(\cal{E}), \mathbf{R}Q).
	\]
\end{theorem}
\begin{proof}
	Write $S^e_n(\cal{E}, Q, w)$ for the subdivided hermitian $S_\bullet$-construction on an exact form category with weak equivalences. Then by Proposition \ref{sn}, for each $n \ge 0$ we have equivalences
	\[
		|\wquad(S^e_n(\cal{E}, Q, w))| \xto{\simeq} \mathrm{Pn}(\mathcal{Q}_n(L_w(\cal{E})), \mathbf{R}Q_n),
	\]
	assembling into a levelwise equivalence of simplicial spaces
	\[
		\left([s] \mapsto |\wquad(S_{2s+1}(\cal{E}), Q_{2s+1})|\right) \simeq \mathrm{Pn}(\mathcal{Q}_\bullet(L_w(\cal{E})), \mathbf{R}Q_\bullet),
	\]
 and accordingly, since the composite $wS^e_\bullet(\cal{E}) \to \mathscr{S}^e_\bullet(L_w(\cal{E}))^\simeq \to \mathcal{Q}_\bullet(L_w(\cal{E})^\simeq)$ realises to an equivalence, the induced map on fibres
    \[\begin{tikzcd}
		\cal{GW}(\cal{E}), Q, w) \ar[d] \ar[r] & \ar[d] |\wquad(S^e_\bullet\cal{E}, Q^e_\bullet, w)| \ar[r] & \ar[d] |wS_\bullet\cal{E}| \ar[d] \\
		\cal{GW}(L_w(\cal{E}), \mathbf{R}Q) \ar[r] & {|\mathrm{Pn}(\mathcal{Q}_\bullet(L_w(\cal{E})), \mathbf{R}Q_\bullet)|} \ar[r] & {|(\mathcal{Q}_\bullet(L_w(\cal{E}))^{\simeq}|}
	\end{tikzcd}\]
 is an equivalence.
\end{proof}
In particular, we have the following.
\begin{corollary}\label{ex-comp}
    Suppose $(\cal{E}, Q, \D, \eta)$ is an exact form category with strong duality,  with derived Poincar\'e category $(\db(\cal{E}), \Qoppa)$; then there is a natural equivalence of Grothendieck-Witt spaces
    \[
        \cal{GW}(\cal{E}, Q) \to \cal{GW}(\db(\cal{E}), \Qoppa).
    \]
\end{corollary}
\begin{proof}
    We have equivalences
    \[
        \cal{GW}(\cal{E}, Q) \to \cal{GW}(\chb(\cal{E}), Q, \mathbf{qis}) \to \cal{GW}(\db(\cal{E}), \Qoppa),
    \]
    where the first is the Gillet-Waldhausen equivalence of \cite{Sch24}, and the second follows from Theorem \ref{main-comp} since $(\chb(\cal{E}), Q, \mathbf{qis}, \D, \eta)$ is a complicial exact form category with weak equivalences and strong duality.
\end{proof}
\section{Grothendieck-Witt spectra}\label{gwsp}
In this section we refine the equivalence of Theorem \ref{main-comp} to a stable equivalence of Grothendieck-Witt spectra. In both the form and Poincar\'e-categorical setting, the Grothendieck-Witt space admits a canonical delooping into a generally non-connective spectrum $\mathrm{GW}$; we first review the relevant formalism.
\subsection{Delooping $\cal{GW}$ of complicial exact form categories}
Suppose given a complicial exact from category with weak equivalences and strong duality $(\cal{E}, Q, w, \D, \eta)$; for any small category $I$, the functor category $\cal{E}^I:=\fun(I, \cal{E})$ inherits a pointwise complicial structure via
\[
    \otimes^I :\mathscr{C}_{\zz} \times \cal{E}^I \to \cal{E}^I, \quad (X, F) \mapsto (i \mapsto X \otimes F(i)),
\]
adjoint to $\otimes \circ (1 \times \ev_{\cal{E}}):\mathscr{C}_{\zz} \times \cal{E}^I \times I \to \mathscr{C}_{\zz} \times \cal{E} \to \cal{E}$, for $\ev_I$ the evaluation functor $(F, i) \mapsto F(i)$. For $I = [1]$, the arrow category $\arr(\cal{E}) := \cal{E}^{[1]}$ inherits a duality with
\[
    \D(x \xto{f} y) = (\D(y) \xto{\D(f)} \D(x)),
\]
and double dual identification $\eta_f = (\eta_x, \eta_y)$. Frobenius contractible objects of $\arr(\cal{E})$ are those maps $f$ which are retracts of cones on maps, i.e.\ admitting factorisations
\[\begin{tikzcd}
    x \ar[r] \ar[d, "f"] \ar[rr, bend left=4ex, "1_x"] & C\otimes u \ar[r] \ar[d, "C\otimes g"] & x \ar[d, "f"] \\
    y \ar[r] \ar[rr, bend right=4ex, "1_y"'] & C\otimes v \ar[r] & y,
\end{tikzcd}\]
and there are (generally strict) inclusions $w_\mathrm{Frob} \subset w_{\mathrm{Frob}_\mathrm{pt}} \subset \arr(\cal{E})$, where the former is the class of Frobenius equivalences internal to $\arr(\cal{E})$, and the latter the class of maps whose image under the source and targets functors are Frobenius equivalences in $\cal{E}$. Write also $w_\mathrm{pt} \subset w_\mathrm{cone}$ for the classes of maps which are pointwise in $w$, resp.\ such that the induced map on cones $C(f) \to C(g)$ is a weak equivalence in $\cal{E}$. We then have a sequence of complicial exact form categories
\[
    (\cal{E}, Q, w, \D\, \eta) \to (\arr(\cal{E}), Q^{\Delta^1}, w_\mathrm{pt}, \D, \eta) \to (\arr(\cal{E}), Q^{\Delta^1}, w_\mathrm{cone}, \D, \eta),
\]
where the first functor is $x \mapsto 1_x$ with duality compatibility the identity $1_{\D(x)} = \D(1_x)$, and the latter is the change of weak equivalences, with underlying functor the identity. We write $Q^{\Delta^1}:\mathrm{Ar}(\cal{E})^{\op} \to \Ab$ for the functor $Q_{[1]}$ of \S\ref{cgw}, where the poset $[1]$ is given the unique strong duality $i \mapsto 1-i$; note the potential clash of notation with the $\mathrm{D}_{\ge 0}(\zz)$-valued Frobenius derived functor $Q^{\Delta^\bullet}$ of $Q$. Explicitly, for a map $f:x \to y$ in $\cal{E}$, we have
\[
    Q^{\Delta^1}(f) = \{(\xi, a) \mid \xi \in Q(x), a:x \to \D(y), \rho(\xi) = \D(f)\circ a\},
\]
with $\D_{[1]}(f) = (\D(y) \xto{\D(f)} \D(x))$, which is strong with respect to $w_\mathrm{pt}$. Write $\arr(\cal{E})^\mathrm{cone} \subset \arr(\cal{E})$ for the subcategory of $w_\mathrm{cone}$-acyclic objects. Then the functor $(\cal{E}, Q, w, \D, \eta) \to (\arr(\cal{E}), Q^{\Delta^1}, w_\mathrm{pt}, \D, \eta)$ factors through the inclusion
\[
    (\arr(\cal{E})^\mathrm{cone}, Q^{\Delta^1}, w_\mathrm{pt}, \D, \eta) \hookrightarrow (\arr(\cal{E}), Q^{\Delta^1}, w_\mathrm{pt}, \D, \eta),
\]
and by \cite[Lem.\ 11.3]{Sch24} there is an induced equivalence
\[
    \left|\wquad(S_\bullet^e(\cal{E}, Q, w))\right| \xto{\simeq} \left|\wquad(S_\bullet^e(\arr(\cal{E})^\mathrm{cone}, Q^{\Delta^1}, w_\mathrm{pt}))\right| 
\]
Write $(\cal{E}^{[n]}, Q^{[n]}, w)$ for the complicial exact form category with weak equivalences and strong duality $(\cal{E}^{[n-1]}, Q^{[n-1]}, w)^{[1]}$ for $n \ge 2$, where $(\cal{E}, Q, w)^{[1]} = (\mathrm{Ar}(\cal{E}), Q^{\Delta^1}, w_\mathrm{cone})$. Then by \cite[Prop.\ 11.4]{Sch24}, there are canonical maps for each $n\ge 0$,
\begin{equation}\label{bond}
	\left|\wquad(\mathcal{R}^{(n)}(\cal{E}^{[n]}, Q^{[n]}, w))\right| \to \Omega\left|\wquad(\mathcal{R}^{(n+1)}(\cal{E}^{[n+1]}, Q^{[n+1]}, w))\right|,
\end{equation}
which for $n \ge 1$ are equivalences, and for $n=0$ factors as
\[
	\left|\wquad(\cal{E}, Q, w)\right| \to \cal{GW}(\cal{E}, Q, w) \xto{\simeq} \Omega\left|\wquad(\mathcal{R}^{(1)}_\bullet(\cal{E}^{[1]}, Q^{\Delta^1}))\right|.
\]
Here we write $\mathcal{R}^{(n)}_\bullet(\cal{E}, Q, w)$ for the diagonal of the $n^\mathrm{th}$ iterate of the subdivided $S_\bullet$-construction on $(\cal{E}, Q, w)$, i.e.\ the multisimplicial exact form category
\[
	(\bbDelta^{\op})^{\times n} \to \formcat, \quad ([k_1], \dots, [k_n]) \mapsto S^e_{k_1}S^e_{k_2}\dots S^e_{k_n}(\cal{E}, Q, w).
\]
\begin{definition}[{\cite[Def.\ 11.5]{Sch24}}]
	The Grothendieck-Witt spectrum associated to a complicial exact form category with weak equivalences is the positive $\Omega$-spectrum
	\[
		\mathrm{GW}(\cal{E}, Q, w) := (\left|\wquad(\cal{E}, Q, w)\right|, \left|\wquad(\mathcal{R}^{(1)}(\cal{E}^{[1]}, Q^{\Delta^1})\right|, \left|\wquad(\mathcal{R}^{(2)}(\cal{E}^{[2]}, Q^{[2]})\right|, \dots ),
	\]
	with bonding maps those given in (\ref{bond}). For $n \ge 0$, define the $n$-shifted Grothendieck-Witt spectrum $\mathrm{GW}^{[n]}(\cal{E}, Q, w) := \mathrm{GW}(\cal{E}^{[n]}, Q^{[n]}, w)$.
\end{definition}
\subsection{Delooping $\cal{GW}$ of Poincar\'e $\infty$-categories}
\begin{recollection}
	We have the Rezk adjunction
	\[\begin{tikzcd}
		s\cal{S} \ar[r, bend left=2ex, shift left=.4ex, "\mathrm{acat}", "\rotatebox{90}{$\vdash$}"'] \ar[r, hookleftarrow, bend right=2ex, shift right=.4ex, "N"'] & \cat_\infty
	\end{tikzcd}\]
	presenting $\cat_\infty$ of (small) $\infty$-categories as a reflective subcategory of the $\infty$-category of (small) simplicial spaces, with the essential image of the inclusion $N = \Hom_{\cat_\infty}(\Delta^\bullet, -)$ the $\infty$-category of complete Segal spaces; see \cite[\S2.1]{CD23b} for a review. We refer to the left adjoint $\mathrm{acat}$ as the \textit{associated category}. Given a Segal space $X_\bullet$ (not necessarily complete) and $x, y \in X_0$, we may recover the mapping space $\Hom_{\mathrm{acat}(X)}(x, y)$ as the pullback
	\begin{equation}\label{segal}\begin{tikzcd}
		\Hom_{\mathrm{acat}(X)}(x, y) \ar[r] \ar[rd, phantom, "\scalebox{1.5}{$\lrcorner$}", very near start] \ar[d] & X_1 \ar[d, "{(d_1, d_0)}"] \\
		\Delta^0 \ar[r, "{(x, y)}"] & X_0 \times X_0,
	\end{tikzcd}\end{equation}
by \cite[\S2.1]{CD23b}.
\end{recollection}
Recall \cite{CD23b} that a sequence
\begin{equation}\label{pvseq}
    (\cal{C}, \Qoppa) \xto{(i, \eta)} (\cal{D}, \Phi) \xto{(p, \vartheta)} (\cal{E}, \Psi)
\end{equation}
in $\cat^p_\infty$ is a \textit{Poincar\'e-Verdier sequence} if it is both a fibre and cofibre sequence; in this case the image of (\ref{pvseq}) under the canonical functor $\cat^p_\infty \to \cat^\mathrm{st}_\infty$ is a Verdier sequence, and we say (\ref{pvseq}) is \textit{split} if the underlying Verdier sequence splits, i.e.\ $p$ (equivalently $i$) admits both a left and right adjoint. In this situation, we refer to $(i, \eta)$ as a (split) Poincar\'e-Verdier inclusion, and $(p, \vartheta)$ as a (split) Poincar\'e-Verdier projection. A cartesian square
\[\begin{tikzcd}
    (\cal{C}, \Phi) \ar[r] \ar[d] & (\cal{D}, \Psi) \ar[d] \\
    (\cal{C}', \Phi') \ar[r] & (\cal{D}', \Psi')
\end{tikzcd}\]
is then (split) Poincar\'e-Verdier if each of the vertical functors are (split) Poincar\'e-Verdier projections.\\
A functor $\cal{F}:\cat^p_\infty \to \cal{S}$ is Verdier localising (resp.\ additive) if it sends Poincar\'e-Verdier (resp.\ split Poincar\'e-Verdier) squares to cartesian squares of spaces, and grouplike if it is additive, and the canonical lift to $\mathrm{Mon}_{\mathbb{E}_\infty}(\cal{S})$ arising from semi-additivity of $\cat^p_\infty$ takes values in the subcategory of grouplike $\mathbb{E}_\infty$-spaces; see \cite[\S1.5]{CD23b}. Given an additive functor $\cal{F}:\cat^p_\infty \to \cal{S}$, to any Poincar\'e category $(\cal{C},\Qoppa)$ we obtain an associated simplicial space $\cal{F}\mathcal{Q}_\bullet(\cal{C}, \Qoppa^{[1]})$ which is Segal for $\cal{F}$ additive, and additionally complete if $\cal{F}$ preserves limits. The split Poincar\'e-Verdier sequence
\[
    (\cal{C}, \Qoppa) \to \mathcal{Q}_1(\cal{C}, \Qoppa^{[1]}) \xto{(d_1, d_0)} \mathcal{Q}_0(\cal{C}, \Qoppa^{[1]})^{\times 2} = (\cal{C}, \Qoppa^{[1]})^{\times 2},
\]
where the left-hand map is informally given by $x \mapsto (0 \leftarrow x \to 0)$ with quadratic compatibility the canonical unit equivalence $\Qoppa(x) \to \Qoppa^{[1]}_1(0 \leftarrow x \to 0) = \Omega\Sigma\Qoppa(x)$, yields upon applying $\cal{F}$ a cartesian square of spaces
\begin{equation}\label{end0}\begin{tikzcd}
    \cal{F}(\cal{C}, \Qoppa) \ar[r] \ar[d] & \cal{F}\mathcal{Q}_1(\cal{C}, \Qoppa^{[1]}) \ar[d] \\
    \Delta^0 \ar[r, "{(0,0)}"] & \cal{F}(\cal{C}, \Qoppa^{[1]})^{\times 2}.
\end{tikzcd}\end{equation}
Writing $\cob^{\cal{F}}(\cal{C}, \Qoppa)$ for the $\infty$-category associated to the Segal space $\cal{F}\mathcal{Q}_\bullet(\cal{C}, \Qoppa^{[1]})$, the square (\ref{segal}) gives an identification, for $x, y \in \cal{F}(\cal{C}, \Qoppa^{[1]})$,
\[
    \Map_{\mathrm{Cob}(\cal{C}, \Qoppa)}(x, y) \simeq \fib\left(\cal{F}\mathcal{Q}_1(\cal{C}, \Qoppa^{[1]}) \to \cal{F}(\cal{C}, \Qoppa^{[1]})^{\times 2}\right),
\]
Accordingly, (\ref{end0}) induces a map $\cal{F}(\cal{C}, \Qoppa) \xto{\simeq} \Hom_{\cob^{\cal{F}}(\cal{C}, \Qoppa)}(0,0) \to \Omega\left|\cob^{\cal{F}}(\cal{C}, \Qoppa)\right|$. For additive $\cal{F}$, the functor $(\cal{C}, \Qoppa) \mapsto \left|\cob^{\cal{F}}(\cal{C}, \Qoppa)\right|$ is again additive, and the assignment $\cal{F} \mapsto \left|\cob^{\cal{F}}(-)\right|$ assembles into an endofunctor $\left|\cob^{(-)}(-)\right|$ of $\fun^\mathrm{add}(\cat_\infty^p, \cal{S})$. By \cite[Th.\ 3.3.4]{CD23b}, $\left|\cob^{\cal{F}}(-)\right|$ is a model for the suspension of $\cal{F}$ in $\fun^\mathrm{add}(\cat_\infty^p, \cal{S})$, and for $\cal{F}$ group-like, the map $\cal{F} \to \Omega\left|\cob^{\cal{F}}(-)\right|$ is an equivalence; moreover, the latter map exhibits the target as the group completion of $\cal{F}$. For $n\ge 0$, write
\[
	\mathcal{Q}^{(n)}_{\bullet, \dots, \bullet}(\cal{C}, \Qoppa):(\bbDelta^{\op})^{\times n} \to \cat^p_\infty, \quad ([k_1], \dots, [k_n]) \mapsto \mathcal{Q}_{k_1}\mathcal{Q}_{k_2}\dots\mathcal{Q}_{k_n}(\cal{C}, \Qoppa)
\]
for the $n$-simplicial hermitian $\mathcal{Q}$-construction on $(\cal{C}, \Qoppa)$, $\mathcal{Q}^{(n)}_\bullet$ for the diagonal, and $\mathrm{Cob}_n^{\cal{F}}(\cal{C}, \Qoppa) := \cal{F}\mathcal{Q}_\bullet^{(n)}(\cal{C}, \Qoppa^{[n]})$. We have $\mathrm{Cob}^{\cal{F}}_1(\cal{C}, \Qoppa) = \cal{F}\mathcal{Q}_\bullet(\cal{C}, \Qoppa^{[1]})$, and canonical identifications
\[
	\left|\mathrm{Cob}_1^{|\mathrm{Cob}_n^{\cal{F}}(-)|}(\cal{C}, \Qoppa)\right| \simeq \left|\cal{F}\mathcal{Q}_\bullet^{(n+1)}(\cal{C}, \Qoppa^{[n+1]})\right| = \left|\mathrm{Cob}_{n+1}(\cal{C}, \Qoppa)\right|.
\]
Then for additive $\cal{F}$, by \cite[Prop.\ 3.4.5]{CD23b}, the maps
\[
	\left|\mathrm{Cob}_n^{\cal{F}}(\cal{C}, \Qoppa)\right| \to \Omega\left|\mathrm{Cob}_{n+1}^{\cal{F}}(\cal{C}, \Qoppa)\right|
\]
adjoint to the equivalences $\Sigma|\mathrm{Cob}_n^{\cal{F}}(-)| \simeq |\mathrm{Cob}_1^{|\mathrm{Cob}_n^{\cal{F}}(-)|}|$ are equivalences for $n\ge 1$, and we write $\mathbb{C}\mathrm{ob}^{\cal{F}}(-)$ for the associated positive $\Omega$-spectrum, with $\mathbb{C}\mathrm{ob}^{\cal{F}}(\cal{C}, \Qoppa)_n = \left|\mathrm{Cob}^{\cal{F}}_n(\cal{C}, \Qoppa)\right|$. For $\cal{F}$ grouplike, $\mathbb{C}\mathrm{ob}^{\cal{F}}(\cal{C}, \Qoppa)$ is an $\Omega$-spectrum, and the group completion map $\cal{F} \to \cal{F}^\mathrm{grp}$ induces the spectrification map
\[
	\mathbb{C}\mathrm{ob}^{\cal{F}}(\cal{C}, \Qoppa) \to \mathbb{C}\mathrm{ob}^{\cal{F}^\mathrm{grp}}(\cal{C}, \Qoppa)
\]
for each $(\cal{C}, \Qoppa) \in \cat^p_\infty$. Recall that $\mathrm{Pn}:\cat^p_\infty \to \cal{S}$ is additive, with group completion $\cal{GW}$, and write $\mathrm{Cob}_n :=\mathrm{Cob}_n^\mathrm{Pn}$.
\begin{definition}
	The Grothendieck-Witt pre-spectrum associated to a Poincar\'e category $(\cal{C}, \Qoppa)$ is the positive $\Omega$-spectrum
\[
	\mathrm{GW}(\cal{C}, \Qoppa) := (\left|\mathrm{Cob}_0(\cal{C}, \Qoppa)\right|, \left|\mathrm{Cob}_1(\cal{C}, \Qoppa)\right|, \left|\mathrm{Cob}_2(\cal{C}, \Qoppa)\right|, \dots),
\]
and for $n \in \zz$ we set $\mathrm{GW}^{[n]}(\cal{C}, \Qoppa) := \mathrm{GW}(\cal{C}, \Qoppa^{[n]})$, the $n$-shifted Grothendieck-Witt pre-spectrum. As above, the map $\mathrm{Pn}(\cal{C}, \Qoppa) = \left|\mathrm{Cob}_0^\mathrm{Pn}(\cal{C}, \Qoppa)\right| \to \Omega\left|\mathrm{Cob}_1^\mathrm{Pn}(\cal{C}, \Qoppa)\right| \simeq \cal{GW}(\cal{C}, \Qoppa)$ is the universal map exhibiting the functor $\cal{GW}:\cat_\infty^p \to \cal{S}$ as the group completion of $\mathrm{Pn}$.
\end{definition}
\subsection{Comparing Grothendieck-Witt spectra}
Fix a complicial exact form category with weak equivalences and strong duality $(\cal{E}, Q, w, \D, \eta)$, with derived Poincar\'e category $(L_w(\cal{E}), \mathbf{R}Q)$. By \cite[Th.\ 7.6.17]{Cis19}, the functor $\mathrm{Ar}(\cal{E}) \to \mathrm{Ar}(L_w(\cal{E}))$ induces an equivalence $L_{w_\mathrm{pt}}(\mathrm{Ar}(\cal{E})) \xto{\simeq} \mathrm{Ar}(L_w(\cal{E}))$, and we write $(\mathrm{Ar}(L_w(\cal{E})), \mathbf{R}(Q^{\Delta^1}))$ for the derived Poincar\'e category associated to $(\mathrm{Ar}(\cal{E}), Q^{\Delta^1}, w_\mathrm{pt}, \D, \eta)$. Given a Poincar\'e category $(\cal{C}, \Qoppa)$, the arrow category $\mathrm{Ar}(\cal{C})$ can be equipped with a quadratic functor $\Qoppa_\mathrm{ar}$, defined by the cartesian square
\[\begin{tikzcd}
    \Qoppa_\mathrm{ar}(x \xto{f} y) \ar[r] \ar[d] & \Qoppa(x) \ar[d] \\
    B_\Qoppa(y, x) \ar[r] & B_\Qoppa(x, x).
\end{tikzcd}\]
The pair $(\mathrm{Ar}(\cal{C}), \Qoppa_\mathrm{ar})$ is Poincar\'e if $(\cal{C}, \Qoppa)$ is \cite[\S2.4]{CD23a}.
\begin{lemma}
    There is a natural Poincar\'e equivalence
    \[
        (\mathrm{Ar}(L_w(\cal{E})), \mathbf{R}(Q^{\Delta^1})) \simeq (\mathrm{Ar}(L_w(\cal{E})), (\mathbf{R}Q)_\mathrm{ar}).
    \]
\end{lemma}
\begin{proof}
    Assume for now that $w=w_\mathrm{Frob}$ is the class of Frobenius equivalences in $\cal{E}$. Then by the results of \S\ref{from-additivity-to-homotopy-coherence}, there are equivalences
    \[
        \Omega^\infty\mathbf{R}Q \simeq \qdb, \quad \Omega^\infty\mathbf{R}_\mathrm{Frob}(Q^{\Delta^1}) \simeq (Q^{\Delta^1})^{\Delta^\bullet},
    \]
    where we write $\mathbf{R}_\mathrm{Frob}(Q^{\Delta^1})$ for the right derived quadratic functor of $Q^{\Delta^1}$ at the Frobenius equivalences internal to $\mathrm{Ar}(\cal{E})$. Write $\gamma:L_\mathrm{Frob}(\mathrm{Ar}(\cal{E})) \to \mathrm{Ar}(L_\mathrm{Frob}(\cal{E}))$ for the Dwyer-Kan localisation at the pointwise Frobenius equivalences, so that $\mathbf{R}(Q^{\Delta^1})(f)= \gamma_!\mathbf{R}_\mathrm{Frob}(Q^{\Delta^1})(f)$ for a map $x \xto{f} y$. Writing $J_f \subset (\mathrm{Ar}(\cal{E}) \downarrow f)$ for the full subcategory of the slice spanned by (pointwise) trivial Frobenius deflations over $f$, by the pointwise formula for left Kan extensions we have
    \[
        \gamma_!(Q^{\Delta^1})^{\Delta^\bullet}(f) \simeq \underset{g \in J_f^{\op}}\colim\ (Q^{\Delta^1})^{\Delta^\bullet}(g).
    \]
    We claim (cf.\ Lemma \ref{frob-to-pt}) that the canonical maps
    \[
        (Q^{\Delta^1})^{\Delta^\bullet}(0 \to x) \to \gamma_!(Q^{\Delta^1})^{\Delta^\bullet}(0 \to x), \quad (Q^{\Delta^1})^{\Delta^\bullet}(1_x) \to \gamma_!(Q^{\Delta^1})^{\Delta^\bullet}(1_x)
    \]
    are equivalences in $\mathrm{D}_{\ge 0}(\zz)$. By Lemma \ref{hoco} we have
    \begin{align*}
        (Q^{\Delta^1})^{\Delta^\bullet}(0 \to x) & = Q^1(0 \to \Delta^\bullet x) \simeq \underset{y \in J_x^{\op}}\colim\ Q^1(0 \to y), \\
        (Q^{\Delta^1})^{\Delta^\bullet}(1_x) & = Q^1(1_{\Delta^\bullet x}) \simeq \underset{y \in J_x^{\op}}\colim\ Q^1(1_y).
    \end{align*}
    Now we have adjoint pairs
    \[\begin{tikzcd}[row sep=tiny]
        (y \stackrel{\sim}\epi x) \ar[r, mapsto] & (0 \to y \stackrel{\sim}\epi x) \\
        J_x \ar[r, "", "\rotatebox{90}{$\vdash$}"', bend left=1.5ex, shift left=.4ex] & J_{0 \to x} \ar[l, bend left=1.5ex, shift left=.4ex] \\
        (z \stackrel{\sim}\epi x) & \ar[l, mapsto] (y \to z \stackrel{\sim}\epi x),
    \end{tikzcd}\]
    and
    \[\begin{tikzcd}[row sep=tiny]
        (p:y \stackrel{\sim}\epi x) \ar[r, mapsto] & ((p,p) : (1_y \stackrel{\sim}\epi 1_x) \\
        J_x \ar[r, "", "\rotatebox{90}{$\vdash$}"', bend left=1.25ex, shift left=.4ex] & J_{1_x} \ar[l, bend left=1.25ex, shift left=.4ex] \\
        (p:y \stackrel{\sim}\epi x) & \ar[l, mapsto] ((p, q):(y\to z) \stackrel{\sim}\epi 1_x),
    \end{tikzcd}\]
    and since any right adjoint is cofinal, the claim follows upon taking opposite categories. Note accordingly that
    \[
        \Omega^\infty\gamma_!(Q^{\Delta^1})^{\Delta^\bullet}(0 \to x) \simeq Q^{\Delta^1}(0 \to \Delta^\bullet x) = 0, \quad \Omega^\infty\gamma_!(Q^{\Delta^1})^{\Delta^\bullet}(1_x) = Q^{\Delta^1}(1_{\Delta^\bullet x}) = \qdb(x).
    \]
    Now for $f \in \mathrm{Ar}(\cal{E})$, consider the map of fibre sequences in $\mathrm{Ar}(L_\mathrm{Frob}(\cal{E}))$ 
    \[\begin{tikzcd}
        & 0 \ar[rr] \ar[dd] && x \ar[dd, "f", near start] \ar[rr, equal] && x \ar[dd] \\
        0 \ar[ru, equal] \ar[rr] \ar[dd] && x \ar[ru, equal] && x \ar[from=ll, crossing over, equal] \ar[dd] \ar[ru, equal] \\
        & y \ar[rr, equal] && y \ar[rr] && 0, \\
        x \ar[rr, equal] \ar[ru, "f"] && x \ar[from=uu, crossing over, equal] \ar[ru, "f"] \ar[rr] && 0 \ar[ru, equal]
    \end{tikzcd}\]
    furnished by $f$. Since $\Omega^\infty\mathbf{R}(Q^{\Delta^1})$ is 2-excisive, there is an induced equivalence of total fibres
    \[\hspace*{-1cm}\begin{small}
        \mathrm{fibt}\left[
            \begin{tikzcd}
                \mathbf{R}(Q^{\Delta^1})(f) \ar[r] \ar[d] & \mathbf{R}(Q^{\Delta^1})(0 \to y) \ar[d] \\
                B_{\mathbf{R}(Q^{\Delta^1})}(f, 0 \to y) \ar[r] & B_{\mathbf{R}(Q^{\Delta^1})}(0 \to y, 0 \to y)
            \end{tikzcd}\right] \xto{\simeq} \mathrm{fibt}\left[
            \begin{tikzcd}
                \mathbf{R}(Q^{\Delta^1})(1_x) \ar[r] \ar[d] & \mathbf{R}(Q^{\Delta^1})(0 \to x) \ar[d] \\
                B_{\mathbf{R}(Q^{\Delta^1})}(1_x, 0 \to x) \ar[r] & B_{\mathbf{R}(Q^{\Delta^1})}(0 \to x, 0 \to x),
            \end{tikzcd}\right]
    \end{small}\]
    and upon passing to connective covers and commuting fibres, we obtain a cartesian square
    \[\begin{tikzcd}
        \Omega^\infty\mathbf{R}(Q^{\Delta^1})(f) \ar[r] \ar[d] & \Omega^\infty\mathbf{R}Q(x) \ar[d] \\
        \Map_\Delta(y, \D(x)) \ar[r] & \Map_\Delta(x, \D(x)),
    \end{tikzcd}\]
    and a natural equivalence $\Omega^\infty\mathbf{R}(Q^{\Delta^1})(f) \simeq \Omega^\infty(\mathbf{R}Q)_\mathrm{ar}(f)$, which by Proposition \ref{ext} promotes to an equivalence of quadratic $\mathrm{D}(\zz)$-valued functors $\mathbf{R}(Q^{\Delta^1}) \simeq (\mathbf{R}Q)_\mathrm{ar}$ on $\mathrm{Ar}(L_\mathrm{Frob}(\cal{E}))$. To conclude for general $w \subset \cal{E}$, we simply note that the definition of $(\mathbf{R}Q)_\mathrm{ar}$ is stable under filtered colimits, so the derived functor $\mathbf{R}(Q^{\Delta^1})$ of $Q^{\Delta^1}$ at the pointwise ($w$-)weak equivalences is naturally equivalent to $(\mathbf{R}Q)_\mathrm{ar}$ on $\mathrm{Ar}(L_w(\cal{E})) \simeq L_{w_\mathrm{pt}}(\mathrm{Ar}(\cal{E}))$.
\end{proof}
Recall that associated to a Poincar\'e category $(\cal{C}, \Qoppa)$ is the \textit{metabolic} Poincar\'e category $\mathrm{Met}(\cal{C}, \Qoppa)$, with underlying stable category $\mathrm{Ar}(\cal{C}) := \fun(\Delta^1, \cal{C})$ and quadratic functor
\[
    \Qoppa_\mathrm{met}(f) := \fib(\Qoppa(y) \xto{\Qoppa(f)} \Qoppa(x)),
\]
inducing a duality
\[
    (x \xto{f} y) \mapsto (\fib(\D(y) \to \D(x)) \to \D(y)).
\]
$\mathrm{Met}(\cal{C}, \Qoppa)$ fits into the metabolic (split) Poincar\'e-Verdier sequence
\[
    (\cal{C}, \Qoppa) \to \mathrm{Met}(\cal{C}, \Qoppa^{[1]}) \to (\cal{C}, \Qoppa^{[1]}),
\]
with the first functor informally given by $x \mapsto (x \to 0)$, and the second by $(x \to y) \mapsto y$, with quadratic compatibilities
\[
    \Qoppa(x) \xto{\simeq} \fib(0 \to \Sigma\Qoppa(x)), \quad \Sigma\fib\left(\Qoppa(y) \xto{\Qoppa(f)} \Qoppa(x)\right) \to \Sigma\Qoppa(y).
\]
Writing $S_2\cal{C} \subset \fun(\Delta^2, \cal{C})$ for the full subcategory on fibre-cofibre sequences in $\cal{C}$, by \cite[Lem.\ 2.4 .5]{CD23a} the zig-zag of equivalences
\[\begin{tikzcd}[row sep=tiny]
	\mathrm{Ar}(\cal{C}) \ar[r, leftarrow, "d_2"] & S_2(\cal{C}) \ar[r, "d_0"] & \mathrm{Ar}(\cal{C}) \\
	(x \to y) & \ar[l, mapsto] (x \to y \to z) \ar[r, mapsto] & (y \to z)
\end{tikzcd}\]
induces an equivalence of restricted quadratic functors
\[
    d_2^*\Qoppa_\mathrm{ar} \simeq d_0^*\Qoppa_\mathrm{met}^{[1]}.
\]
We thus have an equivalence of Poincar\'e-Verdier inclusions
\begin{equation}\label{poinc-verd-inc}\begin{tikzcd}
    (L_w(\cal{E}), \mathbf{R}Q) \ar[r, hookrightarrow] \ar[d, equal] & (\mathrm{Ar}(L_w(\cal{E})), (\mathbf{R}Q)_\mathrm{ar}) \ar[d, "\rotatebox{270}{$\simeq$}"] \\
    (L_w(\cal{E}), \mathbf{R}Q) \ar[r, hookrightarrow] & \mathrm{Met}(L_w(\cal{E}), \mathbf{R}Q)^{[1]},
\end{tikzcd}\end{equation}
with the right vertical arrow informally given by $(x \xto{f} y) \mapsto (y \to \cofib(f))$. The functor $L_w(\cal{E}) \to \mathrm{Ar}(L_w(\cal{E}))$, $x \mapsto 1_x$ clearly lands in the stable subcategory of cone-acyclic maps, and participates in the adjoint pair
\[\begin{tikzcd}
    L_w(\cal{E}) \ar[r, shift left=.4ex, "x \mapsto 1_x", bend left=1.25ex] & \ar[l, shift left=.4ex, "s", bend left=1.25ex, "\rotatebox{90}{$\vdash$}"'] \mathrm{Ar}(L_w(\cal{E}))
\end{tikzcd}\]
where $s$ is the source functor. Since $\mathrm{Ar}(L_w(\cal{E}))^\mathrm{cone}$ is simply the full subcategory spanned by the equivalences in $L_w(\cal{E})$, this restricts to an equivalence
\[
    L_w(\cal{E}) \simeq \mathrm{Ar}(L_w(\cal{E}))^\mathrm{cone},
\]
and accordingly the upper Poincar\'e-Verdier inclusion of (\ref{poinc-verd-inc}) is equivalent to
\[
    (\mathrm{Ar}(L_w(\cal{E})^\mathrm{cone}), (\mathbf{R}Q)_\mathrm{ar}\mid_{\mathrm{Ar}(L_w(\cal{E})^\mathrm{cone})}) \hookrightarrow (\mathrm{Ar}(L_w(\cal{E}), (\mathbf{R}Q)_\mathrm{ar}).
\]
Writing $\pi:\mathrm{Ar}(L_w(\cal{E})) \to L_\mathrm{cone}(\mathrm{Ar}(\cal{E}))$ for the localisation at $w_\mathrm{cone}$, we obtain:
\begin{corollary}\label{e[1]}
    There is a natural Poincar\'e equivalence
    \[
        \left(L_\mathrm{cone}(\mathrm{Ar}(\cal{E})), \pi_!(\mathbf{R}(Q^{\Delta^1}))\right) \simeq \left(L_w(\cal{E}), (\mathbf{R}Q)^{[1]}\right).
    \]
\end{corollary}
\begin{corollary}\label{lvlsp}
    Let $(\cal{E}, Q, w, \D, \eta)$ be a complicial exact form category with weak equivalences, with derived Poincar\'e category $(L_w(\cal{E}), \Qoppa)$. Then for each $n \ge 0$, there is a natural equivalence of spaces
    \[
        \left|\wquad\mathcal{R}_\bullet^{(n)}(\cal{E}, Q, w)^{[n]}\right| \xto{\simeq} \left|\mathrm{Cob}_n(L_w(\cal{E}, \mathbf{R}Q))\right|
    \]
\end{corollary}
\begin{proof}
    We show inductively that the complicial exact form category $(\cal{E}, Q, w)^{[n]}$ has derived Poincar\'e category $(L_w(\cal{E}), \mathbf{R}Q^{[n]})$: the case $n=0$ holds by definition. By Corollary \ref{e[1]}, for $(\cal{E}, Q, w, \D, \eta)$ any complicial exact form category with weak equivalences with derived Poincar\'e category $(L_w(\cal{E}), \mathbf{R}Q)$, the derived Poincar\'e category of $(\cal{E}, Q, w)^{[1]}$ is (Poincar\'e-equivalent to) $\left(L_w(\cal{E}), (\mathbf{R}Q)^{[1]}\right)$. Since for any $n\ge 1$ and any indices $k_, \dots, k_n \in \mathbb{N}$, we have that $\mathcal{R}^{(n)}_{k_1, \dots, k_n}(\cal{E}, Q, w)$ is levelwise a complicial exact form category with weak equivalences, for $n\ge 1$, there are natural equivalences
    \begin{align*}
        \left|\wquad\mathcal{R}_\bullet^{(n)}(\cal{E}^{[n]}, Q^{[n]}, w^{[n]})\right| & \simeq \left|\wquad\mathcal{R}_\bullet^{(n)}(\cal{E}^{[n-1]}, Q^{[n-1]}, w^{[n-1]})^{[1]}\right| \\
        & \simeq \left|\mathrm{Pn}\mathcal{Q}^{(n)}_\bullet\left(L_{w_{n-1}}(\cal{E}^{[n-1]}), (\mathbf{R}Q_{n-1})^{[1]}\right)\right| \\
        & \simeq \left|\mathrm{Pn}\mathcal{Q}^{(n)}_\bullet(L_w(\cal{E}), (\mathbf{R}Q)^{[n]})\right| = \left|\mathrm{Cob}_n(L_w(\cal{E}), \mathbf{R}Q)\right|,
    \end{align*}
    where the second equivalence uses the natural identification
    \[
        |\wquad\mathcal{R}_\bullet(\cal{E}, Q, w)| \simeq |\mathrm{Pn}\mathscr{S}^e_\bullet(L_w(\cal{E}), \mathbf{R}Q)| \simeq |\mathrm{Pn}\mathcal{Q}_\bullet(L_w(\cal{E}), \mathbf{R}Q)|
    \]
    for $(\cal{E}, Q, w)$ any complicial exact form category with weak equivalences.
\end{proof}
For $n \ge 0$, write $\gamma_n : \left|\wquad\mathcal{R}_\bullet^{(n)}(\cal{E}, Q, w)^{[n]}\right| \to \Omega\left|\wquad\mathcal{R}_\bullet^{(n+1)}(\cal{E}, Q, w)^{[n+1]}\right|$ for the bonding maps constructed in \cite{Sch24}, and $\delta_n:\left|\cob_n(L_w(\cal{E}), \mathbf{R}Q)\right| \to \Omega\left|\cob_{n+1}(L_w(\cal{E}), \mathbf{R}Q)\right|$ for those of \cite{CD23b}. We make the following recollection: given a Poincar\'e category $(\cal{C},\Qoppa)$, the shifted metabolic sequence
\[
    (\cal{C}, \Qoppa^{[-1]}) \to \mathrm{Met}(\cal{C}, \Qoppa) \to (\cal{C}, \Qoppa)
\]
is split Poincar\'e-Verdier, and is hence taken by the additive functor $\left|\cob(-)\right|$ to the fibre sequence
\[
    \left|\cob(\cal{C}, \Qoppa^{[-1]})\right| \to \left|\cob(\mathrm{Met}(\cal{C}, \Qoppa))\right| \to \left|\cob(\cal{C}, \Qoppa)\right|.
\]
By \cite[Prop.\ 3.1.10]{CD23b}, the boundary map $\Omega\left|\cob(\cal{C}, \Qoppa)\right| \xto{\partial} \left|\cob(\cal{C}, \Qoppa^{[-1]})\right|$ fits into the diagram
\begin{equation}\label{boundary}\begin{tikzcd}
    & \mathrm{Pn}(\cal{C}, \Qoppa) \ar[ld, "\delta_0"'] \ar[rd] \\
    \Omega\left|\cob(\cal{C}, \Qoppa)\right| \ar[rr, "\partial"] && \left|\cob(\cal{C}, \Qoppa^{[-1]})\right|,
\end{tikzcd}\end{equation}
with the left-hand map the bonding map induced by the identification $\mathrm{Pn}(\cal{C}, \Qoppa) \simeq \Hom_{\cob(\cal{C}, \Qoppa)}(0,0)$, and the right-hand by the inclusion of $0$-simplices $(\cal{C}, \Qoppa^{[-1]}) = \mathcal{Q}_0(\cal{C}, \Qoppa) \to \mathcal{Q}_\bullet(\cal{C}, \Qoppa)$. Now, recall from \cite[\S11]{Sch24} that associated to a complicial exact form category with weak equivalences $(\cal{E}, Q, w, \D, \eta)$ is a commutative diagram
\[\begin{tikzcd}
    \wquad(\cal{E}, Q, w) \ar[rr] \ar[d, "\iota"] && 0 \ar[d] \\
    \wquad\mathcal{R}_\bullet(\cal{E}, Q, w) \ar[r, hookrightarrow, "I"] \ar[d] & \wquad\mathcal{R}_\bullet(\mathrm{Ar}(\cal{E}), Q^{\Delta^1}, w_\mathrm{pt}) \ar[d] \ar[r] & wS_\bullet\cal{E} \\
    \wquad\mathcal{R}_\bullet(\mathrm{Ar}(\cal{E})^\mathrm{cone}, Q^{\Delta^1}, w_\mathrm{cone}) \ar[r, hookrightarrow] & \wquad\mathcal{R}_\bullet(\mathrm{Ar}(\cal{E}), Q^{\Delta^1}, w_\mathrm{cone}),
\end{tikzcd}\]
where the functor $\iota$ is the inclusion of $0$-simplices, and $I$ is levelwise given by the form functor
\[
    (\cal{E}, Q, w) \to (\mathrm{Ar}(\cal{E}), Q^{\Delta^1}, w), \quad x \mapsto 1_x.
\]
That the map
\begin{equation}\label{additivity}
\left|\wquad(\mathcal{R}_\bullet(\arr(\cal{E}), Q^{\Delta^1}, w_\mathrm{pt}))\right| \to \left|wS_\bullet\cal{E}\right|
\end{equation}
is a weak equivalence follows ultimately for additivity for hermitian $K$-theory of exact form categories (with weak equivalences and strong duality) \cite[Th.\ 3.1 (8.1)]{Sch24}, and moreover, the maps in these theorems are natural in the input form category, as seen from inspection of the proof of \cite[Th.\ 3.1]{Sch24}. The map (\ref{additivity}) thus admits a functorial homotopy inverse, and accordingly the rectangle
\[\begin{tikzcd}
    \wquad(\cal{E}, Q, w) \ar[r] \ar[d, "\iota"] & 0 \ar[d] \\
    \wquad\mathcal{R}_\bullet(\cal{E}, Q, w) \ar[r, hookrightarrow] & \wquad\mathcal{R}_\bullet(\mathrm{Ar}(\cal{E}), Q^{\Delta^1}, w_\mathrm{pt})
\end{tikzcd}\]
commutes up to natural homotopy. Taking derived Poincar\'e categories, we obtain a diagram
\[\begin{tikzcd}
    \mathrm{Pn}(L_w(\cal{E}), \mathbf{R}Q) \ar[r] \ar[d] & 0 \ar[d] \\
    \mathrm{Pn}\mathcal{Q}_\bullet(L_w(\cal{E}),\mathbf{R}Q) \ar[r] \ar[d] & \mathrm{Pn}\mathcal{Q}_\bullet (\mathrm{Ar}(L_w(\cal{E})), (\mathbf{R}Q)_\mathrm{ar})^{[1]} \ar[d] \\
    0 \ar[r] & \mathrm{Pn}\mathcal{Q}_\bullet(L_w(\cal{E}), (\mathbf{R}Q)^{[1]})
\end{tikzcd}\]
of simplicial Poincar\'e categories commuting up to canonical homotopy, which we further identify via (\ref{poinc-verd-inc}) with
\[\begin{tikzcd}
    \mathrm{Pn}(L_w(\cal{E}), \mathbf{R}Q) \ar[r] \ar[d] & 0 \ar[d] \\
    \mathrm{Pn}\mathcal{Q}_\bullet(L_w(\cal{E}),\mathbf{R}Q) \ar[r] \ar[d] & \mathrm{Pn}\mathcal{Q}_\bullet\mathrm{Met}(L_w(\cal{E}), \mathbf{R}Q)^{[1]} \ar[d] \\
    0 \ar[r] & \mathrm{Pn}\mathcal{Q}_\bullet(L_w(\cal{E}), (\mathbf{R}Q)^{[1]}),
\end{tikzcd}\]
where the lower square encodes the fibre sequence associated to the metabolic Poincar\'e-Verdier sequence. Taking realisations, we obtain the essentially unique map of spaces $\gamma_0:\left|\mathrm{Pn}(L_w(\cal{E}), \mathbf{R}Q)\right| \to \Omega\left|\mathrm{Cob}(L_w(\cal{E}), \mathbf{R}Q)\right|$ rendering the diagram
\[\begin{tikzcd}
    \mathrm{Pn}(L_w(\cal{E}), \mathbf{R}Q) \ar[rd, "u"] \ar[rrd, bend left=2ex] \ar[rdd, bend right=2ex] \\
    & \Omega\left|\cob(L_w(\cal{E}), \mathbf{R}Q)\right| \ar[r] \ar[d, "\partial"] & 0 \ar[d] \\
    & \left|\cob(L_w(\cal{E}), \mathbf{R}Q^{[-1]})\right| \ar[r] \ar[d] & \left|\cob(\mathrm{Met}(L_w(\cal{E}), \mathbf{R}Q)\right| \ar[d] \\
    & 0 \ar[r] & \left|\cob(L_w(\cal{E}), \mathbf{R}Q)\right|
\end{tikzcd}\]
commutative, in which the top square encodes the rotated fibre sequence $\Omega\left|\cob(L_w(\cal{E}), \mathbf{R}Q^{[-1]})\right| \to \left|\cob(\mathrm{Met}(L_w(\cal{E}), \mathbf{R}Q)\right| \to \left|\cob(L_w(\cal{E}), \mathbf{R}Q)\right|$. The map $\gamma_0$ is determined up to contractible choice by the universal property of the limit, and accordingly is homotopic (up to contractible choice) to $\delta_0$. The same argument applies replacing $(\cal{E}, Q, w, \D, \eta)$ with the multisimplicial complicial exact form category with weak equivalences $\mathcal{R}_\bullet^{(n)}(\cal{E}, Q, w)^{[n]}$, and accordingly we obtain (a contractible space of) homotopies between the bonding maps $\gamma_n$ and $\delta_n$.
\begin{remark}
    Suppose given a sequence of pointed spaces $(X_n, *)_{n \ge 0}$, contractible subspaces $i_n:C_n \subset \Map_*((X_n, *), \Omega(X_{n+1}, *))$, for $\Map_*$ the pointed mapping space in the closed symmetric monoidal $\infty$-category $\cal{S}_*^\wedge$ of pointed spaces with the smash product. The maps $i_n$ adjoin to pointed maps $j_n:(X_n, *) \to \Map_*((C_n)_+, \Omega(X_{n+1}, *)$, and we set $\widetilde{X}_n := \Map_*(\bigwedge_{i < n}(C_n)_+, (X_n, *))$. We then have bonding maps $\alpha_n:\widetilde{X}_n \to \Omega\widetilde{X}_{n+1}$ given by the composite
    \begin{small}\[
        \Map_*(\bigwedge_{i < n}(C_n)_+, (X_n, *)) \xto{[1, j_n]} \Map_*(\bigwedge_{i < n} (C_i)_+, \Map_*((C_n)_+, \Omega(X_{n+1}, *))) \simeq \Omega\Map_*(\bigwedge_{i \le n}(C_i)_+, (X_{n+1}, *)).
    \]\end{small}
    This map is contravariantly functorial in the maps $i_n$: given for each $n \ge 0$ contractible subspaces $C_n, C'_n$ and a diagram
    \[\begin{tikzcd}
        C'_n \ar[r, hookrightarrow, "i'_n"] \ar[d, "f_n"] & \Map_*((X_n,*), \Omega(X_{n+1}, *)) \\
        C_n \ar[ru, "i_n", hookrightarrow],
    \end{tikzcd}\]
    we have a commuting diagram
    \[\begin{tikzcd}
        \Map_*(\bigwedge_{i < n} (C_i)_+, (X_n, *)) \ar[r, "{[1, j_n]}"] \ar[d, "{\left[\left(\bigwedge_{i<n}(f_i)_+\right), [(f_n)_+, 1]\right]}"] & \Map_*(\bigwedge_{i \le n}(C_i)_+, \Map_*((C_n)_+, \Omega(X_{n+1}, *))) \ar[d, "{\left[\left(\bigwedge_{i<n}(f_i)_+\right), 1\right]}"] \\
        \Map_*(\bigwedge_{i < n} (C'_i)_+, (X_n, *)) \ar[r, "{[1, j'_n]}"] & \Map_*(\bigwedge_{i \le n}(C'_i)_+, \Map_*((C_n)_+, \Omega(X_{n+1}, *))),
    \end{tikzcd}\]
    and accordingly we have maps of (pre-)spectra $(\widetilde{X}_n, \alpha_n)_n \to (\widetilde{X}'_n, \alpha'_n)_n$. In the case $C''_n, C'_n \to C_n$ are simply the inclusions of points corresponding to choices of bonding maps $\gamma_n, \delta_n$, we obtain a canonical zig-zag of weak equivalences of (pre-)spectra
    \[
        (X_n, \gamma_n)_n \stackrel{\simeq}\leftarrow (\widetilde{X}_n, \alpha_n)_n \xto{\simeq} (X_n, \delta_n)_n.
    \]
\end{remark}
The above discussion and Proposition \ref{lvlsp} imply:
\begin{theorem}\label{main-comp-sp}
    Suppose given a complicial exact form category with weak equivalences $(\cal{E}, Q, w, \D, \eta)$, with derived Poincar\'e category $(L_w(\cal{E}), \mathbf{R}Q)$. Then the localisation $\cal{E} \to L_w(\cal{E})$ induces a weak equivalence of Grothendieck-Witt spectra
    \[
        \mathrm{GW}(\cal{E}, Q, w) \xto{\simeq} \mathrm{GW}(L_w(\cal{E}), \mathbf{R}Q).
    \]
\end{theorem}
\section{Genuine symmetric Poincar\'e structures}\label{genuine-structures}
In this section we give a treatment of the formalism of genuine Poincar\'e structures of \cite{CD23a} and \cite{CHN24}, showing in cases of interest that these are precisely the derived Poincar\'e structures of exact form categories.
\subsection{Bounded derived Poincar\'e categories}
Given an exact category with strong duality $(\cal{E}, \D, \eta)$ embedding exactly into the complicial exact category with weak equivalences and strong duality $(\chb(\cal{E}), \mathbf{qis}, \D, \eta)$, we have an associated duality-preserving Verdier-projection
\[
	(\kb(\cal{E}), \D, \eta) \to (\db(\cal{E}), \D', \eta'),	
\]
where $\D'$ is the left Kan extension of the composite $\pi\circ\D$ along the localisation $\pi^{\op}:\kb(\cal{E})^{\op} \to \db(\cal{E})^{\op}$, which exists since $\D:\chb(\cal{E})^{\op} \to \chb(\cal{E})$ preserves acyclic complexes. Writing $(\kb(\cal{E}), \Qoppa_\oplus^s)$ and $(\db(\cal{E}), \Qoppa^s)$ for the corresponding symmetric Poincar\'e-structures, one may ask whether in general this promotes to a Poincar\'e-Verdier projection
\[
	(\kb(\cal{E}), \Qoppa_\oplus^s) \to (\db(\cal{E}), \Qoppa^s),
\]
or equivalently \cite[Cor.\ 1.1.6]{CD23b} if the induced map $\Qoppa_\oplus^s \to \pi^*\Qoppa^s$ exhibits $\Qoppa^s$ as the left Kan extension of $\Qoppa_\oplus^s$ along $\pi^{\op}$. This will not be the case for a general duality-preserving Verdier quotient of stable $\infty$-categories, since a priori there is no reason for the commutation of the filtered colimit computing the mapping spectra $\map_{\cal{D}}(-,-)$ and the formation of $C_2$-homotopy fixed points. The situation at hand is however sufficiently finitary enough to permit this.
\begin{remark}
\hspace*{1pt}\\
	\begin{enumerate}[label=(\roman*)]
	\vspace*{-12pt}\item The below essentially amounts to showing that $\kb(\cal{E}) \to \db(\cal{E})$ is a bounded Karoubi projection of $\dperf(\zz)$-linear stable $\infty$-categories in the sense of \cite[\S4]{CHN24}. Since we only need that mapping spectra in $\kb(\cal{E})$ can be computed (up to equivalence) as mapping complexes in the dg-category $\chb(\cal{E})$, we dispense with the full $\dperf(\zz)$-linearity; it is the case however that any complicial exact category localises to a $\zz$-linear stable $\infty$-category, giving us an honest tensoring over $\dperf(\zz)$.
		\item For the remainder of the section we assume $\cal{E}$ is \textbf{weakly idempotent complete }(i.e., every retract admits a kernel). Since by \cite[Lem.\ 10.6]{Sch24} for any exact form category with strong duality $(\cal{E}, Q, \D, \eta)$ the weak idempotent completion $\cal{E} \to \cal{E}^\flat$ enhances to an exact form functor
	\[
		(\cal{E}, Q, \D, \eta) \to (\cal{E}^\flat, Q^\flat, \D^\flat, \eta^\flat)
	\]
	of exact form categories with strong duality, such that the induced map on Grothendieck-Witt spaces is an equivalence, this constitutes no real loss of generality. The category $\chb(\cal{E}^\flat)$ is also idempotent complete, and the induced functor $\chb(\cal{E}) \to \chb(\cal{E}^\flat)$ promotes to an exact functor of complicial exact categories with weak equivalences and duality.
	\end{enumerate}
\end{remark}
\begin{lemma}\label{bddverdier}
	For $\cal{E}$ weakly idempotent complete, the functor $(\kb(\cal{E}), \Qoppa_\oplus^s) \to (\db(\cal{E}), \Qoppa^s)$ is a Poincar\'e-Verdier projection.
\end{lemma}
\begin{proof}
	Write $\gamma:\chb(\cal{E}) \to \kb(\cal{E})$ for the Dwyer-Kan localisation at the Frobenius (chain homotopy) equivalences, and $\pi:\kb(\cal{E}) \to \db(\cal{E})$ for the Verdier projection. The pointwise formula for left Kan extensions gives for $x \in \kb(\cal{E})$ an identification
	\[
		(\pi_!\Qoppa_\oplus^s)(\pi(x)) \simeq \colim\left(\cal{I}_x^{\op} \to \kb(\cal{E}) \xto{\Qoppa_\oplus^s} \mathrm{D}(\zz)\right) = \underset{y \in \cal{I}_x^{\op}}\colim\ \hom_{\kb(\cal{E})}(y, \D(y))^{\hct},
	\]
	for $\cal{I}_x \subset (\kb(\cal{E})\downarrow x)$ the full subcategory of quasi-isomorphisms over $x$. Now by \cite[Prop.\ 1.3.5.21]{HA}, the standard $t$-structure on $\mathrm{D}(\zz)$ is right-separated with $\mathrm{D}_{\le 0}(\zz)$ stable under filtered colimits, so by \cite[Cor.\ 4.2.10]{CHN24} we see that filtered colimits of uniformly bounded above diagrams in $\mathrm{D}(\zz)$ commute with finite type limits, i.e.\ limits indexed by simplicial sets $K$ with finitely many non-degenerate simplices in each degree. We observe that the duality $\D:\chb(\cal{E})^{\op} \to \chb(\cal{E})$ restricts to an equivalence $\mathrm{Ch}_{[a, b]}(\cal{E})^{\op} \simeq \mathrm{Ch}_{[-b,-a]}(\cal{E})$ for each $b\ge a$, and accordingly for $y \in \mathrm{Ch}_{[a, b]}(\cal{E})$, the mapping complex $\underline\Hom(y, \D(y))$ is concentrated in degrees $[-2b,-2a]$. The mapping spectrum $\map_{\kb(\cal{E})}(\gamma(y), \D(\gamma(y))$ hence lies in $\mathrm{D}_{[-2b, -2a]}(\zz) = \mathrm{D}_{\ge -2b}(\zz) \cap \mathrm{D}_{\le -2a}(\zz)$, with respect to the standard $t$-structure on $\mathrm{D}(\zz)$. Given a family of complexes $\{y_j\}_{j \in \cal{J}} \subset \kb(\cal{E})$, uniformly bounding the coconnectivity of the family $\{\hom_{\kb(\cal{E})}(y_j, \D(y_j))\}_{j \in \cal{J}}$ then amounts to uniformly bounding the connectivity of the $y_j$.\\
	Suppose $x \in \mathrm{Ch}_{\ge 0, \mathrm{b}}(\cal{E})$; the general case follows by shifting. Then a quasi-isomorphism over $\gamma(x)$ in $\kb(\cal{E})$ is represented by some map $f:y\to x$ in $\chb(\cal{E})$ with acyclic cone; since $\cal{E}$ is weakly idempotent complete, any acyclic complex is strictly acyclic \cite[Prop.\ 10.14]{Buh10}, and accordingly $\mathrm{cone}(f)$ admits factorisations
	\[\begin{tikzcd}
		\dots \ar[r] & x_1 \oplus y_0 \ar[rd, twoheadrightarrow, "{\begin{psmallmatrix} p_0 \amsamp q_0\end{psmallmatrix}}"] \ar[rr, "{\begin{psmallmatrix}d_1 \amsamp -f_0 \\ 0 \amsamp -d_0\end{psmallmatrix}}"] && x_0 \oplus y_{-1} \ar[rd, "{\begin{psmallmatrix} p_{-1} \amsamp q_{-1}\end{psmallmatrix}}", twoheadrightarrow] \ar[rr, "{\begin{psmallmatrix}0 \amsamp -d_{-1}\end{psmallmatrix}}"] && y_{-2} \ar[r] & \dots \\
		&& w_0 \ar[ru, rightarrowtail, "{\begin{psmallmatrix}i_0 \\ j_0 \end{psmallmatrix}}"'] && w_{-1} \ar[ru, rightarrowtail, "j_{-1}"'] 
	\end{tikzcd}\]
	with the sequence $w_0 \mono x_0\oplus y_{-1} \epi w_{-1}$ a conflation in $\cal{E}$. Since $j_{-1}$ is monic, the map $q_{-1}$ is necessarily zero, and accordingly for $y_{\ge -2}$ the complex $\dots \to y_0 \to y_{-1} \epi w_{-1} \to 0 \to \dots$, the composite map $y_{\ge -2} \mono y \to x$ has acyclic cone given by
	\[
		\dots \to x_1 \oplus y_0 \to x_0 \oplus y_{-1} \epi w_{-2} \to 0 \dots
	\]
	(we note that this truncation is by no means sharp). Write $j_x:\cal{I}_x^{\ge -2} \hookrightarrow \cal{I}_x$ for the full subcategory on quasi-isomorphisms $y \to x$ in $\kb(\cal{E})$ with $y$ concentrated in degrees $\ge -2$. Then we claim the inclusion $j_x^{\op}$ is cofinal: since $\kb(\cal{E})$ admits finite limits and colimits, that the same is true for the iterated comma categories $((y, f) \downarrow j_x^{\op}) = (j_x \downarrow (y, f))^{\op}$ follows from the stability of the class of quasi-isomorphisms under pushout and pullbacks. These are hence contractible as soon as they are inhabited; but this follows from above, and Quillen's theorem A completes the proof.
\end{proof}
\subsection{Orientations on the derived category}
Suppose given a small exact category with duality $(\cal{E}, \D, \eta)$, and write $(\cal{E}_\oplus, \D, \eta)$ for the underlying additive category with duality. The exact duality-preserving functor $(\cal{E}_\oplus, \D, \eta) \to (\cal{E}, \D, \eta)$ induces a duality-preserving functor of stable $\infty$-categories with perfect duality $(\kb(\cal{E}), \D) \to (\db(\cal{E}), \D)$. For $\cal{C}$ a stable $\infty$-category with category of $\mathrm{Ind}$-objects $\mathrm{Ind}(\cal{C})$, we refer to the data of a $t$-structure on $\mathrm{Ind}(\cal{C})$ as an \textit{orientation}, following \cite{CHN24}.\\
For $x \in \cal{E}$, write $P_x$ for the set of deflations over $x$. Then the families $(P_x)_{x \in \cal{E}}$ generate a Grothendieck topology on $x$, for which a presheaf of abelian groups $\cal{E}^{\op} \to \Ab$ is a sheaf precisely if it sends conflations in $\cal{E}$ to left-exact sequences of abelian groups \cite[App.\ A]{Buh10}. Write $\elex := \mathrm{Sh}^{\Ab}(\cal{E})$ for the Gabriel-Quillen abelian hull of such sheaves, a Grothendieck abelian category with system of compact generators given by the image under the Yoneda embedding of $\cal{E}$; in the case $\cal{E}$ is split exact, these generators are additionally projective. With the above Grothendieck topology, $\cal{E}$ is an additive $\infty$-site in the sense of \cite[\S2]{Pst23}; since for each $x, y \in \cal{E}$ the sequence $x \mono x \oplus y \epi y$ is a conflation, each sheaf is necessarily additive, i.e.\ the map $X(x \oplus y) \to X(x) \oplus X(y)$ induced by the canonical inclusions is an equivalence. Following \cite{Pst23}, we call an additive  presheaf \textit{spherical}. By the universal properties of $\kb(\cal{E})$ and $\db(\cal{E})$, there are identifications
\begin{align*}
    \mathrm{Ind}(\kb(\cal{E})) = \fun^\mathrm{ex}(\kb(\cal{E})^{\op}, \Sp) \simeq \fun^\oplus(\cal{E}^{\op}, \Sp) = \cal{P}^{\Sp}_\Sigma(\cal{E}), \\
    \mathrm{Ind}(\db(\cal{E})) = \fun^\mathrm{ex}(\db(\cal{E})^{\op}, \Sp) \simeq \fun^\mathrm{ex}(\cal{E}^{\op}, \Sp) = \sh^{\Sp}_\Sigma(\cal{E}),
\end{align*}
where we use \cite[Th.\ 2.8]{Pst23} to identify spherical sheaves of spectra on $\cal{E}$ with those spherical presheaves sending conflations to fibre sequences of spectra. Note that the Yoneda embedding $\cal{E}_\oplus \hookrightarrow \cal{E}_\oplus^\mathrm{lex} = \fun^\oplus(\cal{E}^{\op}, \Ab)$ induces by \cite[Lem.\ 2.61]{Pst23} an equivalence
\[
	\cal{P}^{\Sp}_\Sigma(\cal{E}) \xto{\simeq} \mathrm{D}(\cal{E}_\oplus^\mathrm{lex})
\]
between the the $\infty$-category of spherical presheaves of spectra on $\cal{E}$ and unbounded derived $\infty$-category of $\cal{E}_\oplus^\mathrm{lex}$. Given a spherical sheaf of spectra $X$ on $\cal{E}$, write $\pi_n^\dagger(X) : \cal{E}^{\op} \to \Ab$ for the sheaf associated to the assignment $c \mapsto \pi_nX(c)$; call $X$ connective if $\pi_n^\dagger(X) = 0$ for each $n < 0$, and coconnective if the sheaf of spaces $\Omega^\infty X$ given by postcomposition with $\Omega^\infty:\Sp \to \cal{S}$ is discrete. By \cite[Prop.\ 2.16]{Pst23}, the full subcategories $\sh^{\Sp}_\Sigma(\cal{E})_{\ge 0}$ and $\sh^{\Sp}_\Sigma(\cal{E})_{\le 0}$ of connective and coconnective spherical sheaves respectively define a $t$-structure on $\sh^{\Sp}_\Sigma(\cal{E})$, and we denote by $\tau_{\ge 0}$, $\tau_{\le 0}$ the respective truncation functors.
\begin{lemma}\label{filt}
    The truncation functor $\tau_{\ge 0}:\sh^{\Sp}_\Sigma(\cal{E}) \to \sh^{\Sp}_\Sigma(\cal{E})_{\ge 0}$ commutes with filtered colimits of spherical sheaves.
\end{lemma}
\begin{proof}
    By \cite[Th.\ 2.8]{Pst23}, filtered colimits of spherical sheaves are computed levelwise. Writing $\sh_\Sigma(\cal{E})$ for the $\infty$-category of spherical sheaves of spaces on $\cal{E}$, there is by \cite[Prop.\ 2.19]{Pst23} an adjunction
\[\begin{tikzcd}
    \sh_\Sigma(\cal{E}) \ar[r, bend left=2ex, "\Sigma^\infty"] & \sh^{\Sp}_\Sigma(\cal{E}) \ar[l, bend left=2ex, "\rotatebox{90}{$\vdash$}"', "\Omega^\infty"]
\end{tikzcd}\]
with $\Omega^\infty$ computed levelwise, and with $\Sigma^\infty$ fully faithful with essential image the connective aisle $\sh^{\Sp}_\Sigma(\cal{E})_{\ge 0}$. Write
\[\begin{tikzcd}
    \sh_\Sigma(\cal{E}) \ar[r, bend left=2ex, "\sigma"] & \sh^{\Sp}_\Sigma(\cal{E})_{\ge 0} \ar[l, bend left=2ex, "\rotatebox{90}{$\vdash$}"', "\omega"]
\end{tikzcd}\]
for the induced adjoint equivalence; then by uniqueness of adjoints, there is a natural equivalence $\omega\tau_{\ge 0} \simeq \Omega^\infty$, and since $\Omega^\infty$ commutes with filtered colimits, we are done.
\end{proof}
\subsection{Genuine symmetric structures}
For a hermitian category $(\cal{C}, \Qoppa)$, recall that the first excisive approximation $P_1\Qoppa = \colim_n\Omega^n\Qoppa\Sigma^n$ is computed as the cofibre
\[
    \left[\Delta^*B_\Qoppa\right]_{\hct} \to \Qoppa \to P_1\Qoppa,
\]
since by \cite[Prop.\ 1.1.13]{CD23a} this cofibre is 1-excisive, $P_1$ is left exact, and by \cite[Lem.\ 1.3.1]{CD23a} the functor $\left[\Delta^*B_\Qoppa\right]_{\hct}$ has vanishing 1-excisive approximation. We may view $P_1\Qoppa$ as an object of $\mathrm{Ind}(\cal{C}) = \fun^\mathrm{lex}(\cal{C}^{\op}, \cal{S}) \simeq \fun^\mathrm{ex}(\cal{C}^{\op}, \Sp)$, and if $\cal{C}$ is equipped with an orientation $(\mathrm{Ind}_{\ge 0}(\cal{C}), \mathrm{Ind}_{\le 0}(\cal{C}))$, we say that $\Qoppa$ is \textit{$m$-connective} if $P_1\Qoppa \in \mathrm{Ind}(\cal{C})$ is $m$-connective. By \cite[Lem.\ 3.3.1]{CHN24}, the inclusion $\fun_{\ge m}^\mathrm{q}(\cal{C}) \subset \fun^\mathrm{q}(\cal{C})$ of $m$-connective quadratic functors with respect to this orientation admits a right adjoint $\Qoppa \mapsto \Qoppa \underset{P_1\Qoppa}\times \tau_{\ge m}P_1\Qoppa$, where the connective cover is taken in $\mathrm{Ind}(\cal{C})$. In the case of an exact category with duality $(\cal{E}, \D, \eta)$, write $Q^s:\cal{E}^{\op} \to \Ab$ for the symmetric forms functor, and $\Qoppa_{\kb}^s:\db(\cal{E})^{\op} \to \Sp$ for the (homotopy) symmetric Poincar\'e structure on $\db(\cal{E})$, with linear part
\[
    \Lambda^s := \hom_{\db(\cal{E})}(-, \D(-))^{\tct}.
\]
Write $(\sh_\Sigma^{\Sp}(\cal{E})_{\ge 0}, \sh^{\Sp}_\Sigma(\cal{E})_{\le 0})$ for the $t$-structure of the previous section, and for $m \in \zz$ write $\Qoppa^{\ge m} := \Qoppa^s \underset{\tau_{\ge m}\Lambda^s}\times \Lambda^s$ for the associated truncated hermitian structure.
\begin{lemma}
    For $(\cal{E}, \D, \eta)$ a weakly idempotent complete- exact category with strong duality, the genuine symmetric Poincar\'e structure $\Qoppa^{\ge 0}$ coincides with the derived quadratic functor of $Q^s$ on $\db(\cal{E})$.
\end{lemma}
\begin{proof}
    We first treat the split-exact case. For $\cal{E} = \cal{E}_\oplus$ an additive category, the $t$-structure on $\cal{P}_\Sigma^{\Sp}(\cal{E})$ has connective part those spherical presheaves of spectra factoring through $\Sp_{\ge 0} \subset \Sp$, and since postcomposition with the truncation functor associated with the canonical $t$-structure on spectra sends spherical presheaves of spectra to spherical presheaves of connective spectra, we have for each $x \in \cal{E}$ and $\cal{F} \in \cal{P}_\Sigma^{\Sp}(\cal{E})$ that $(\tau_{\ge 0}\cal{F})(x) = \tau_{\ge 0}(\cal{F}(x))$. Accordingly, the cartesian square
    \[\begin{tikzcd}
        \Qoppa^{\ge 0}(x) \ar[r] \ar[d] & \tau_{\ge 0}\hom_{\kb(\cal{E})}(x, \D(x))^{\tct} \ar[d] \\
        \hom_{\kb(\cal{E})}(x, \D(x))^{\hct} \ar[r] & \hom_{\kb(\cal{E})}(x, \D(x))^{\tct}
    \end{tikzcd}\]
    induces a fibre sequence
    \[
        \Qoppa^{\ge 0}(x) \to \hom_{\kb(\cal{E})}(x, \D(x))^{\hct} \to \tau_{\le -1}\hom_{\kb(\cal{E})}(x, \D(x))^{\tct}
    \]
    for each $x \in \cal{E} \subset \kb(\cal{E})$. Since the duality $\D:\kb(\cal{E})^{\op} \to \kb(\cal{E})$ preserves discrete complexes and $\cal{E} \subset \kb(\cal{E})$ is fully faithful, $\hom_{\kb(\cal{E})}(x, \D(x))^{\hct}$ is coconnective, and the associated long exact sequence on homotopy groups implies that $\Qoppa^{\ge 0}(x)$ is concentrated in degree 0, with
    \begin{align*}
        \pi_0\Qoppa^{\ge 0}(x) & \cong \pi_0\Qoppa^s(x) \\
        & = \pi_0\left[\hom_{\kb(\cal{E})}(x, \D(x))^{\hct}\right] \\
        & \cong \left[\pi_0\hom_{\kb(\cal{E})}(x, \D(x))\right]^{\fct} \\
        & \cong \Hom_{\cal{E}}(x, \D(x))^{\fct} = Q^s(x),
    \end{align*}
    and we conclude by \cite[Th.\ 2.19]{BGMN22}.\\
    For $\cal{E}$ a general exact category, write $\Qoppa^s_\oplus:\kb(\cal{E})^{\op} \to \Sp$ and $\Qoppa^s:\db(\cal{E})^{\op} \to \Sp$ for the respective symmetric Poincar\'e structures, with truncations $\Qoppa^{\ge 0}_\oplus$ and $\Qoppa^{\ge 0}$. The statement for general $\cal{E}$ follows from the claim that the left Kan extension of $\tau_{\ge 0}\Qoppa_\oplus$ along the Verdier projection is equivalent to $\Qoppa^{\ge 0}$. By \cite[Rem.\ 6.2.2.12]{HTT}, the sheafification functor $(-)^\dagger:\cal{P}_\Sigma(\cal{E}^{\Sp}) \to \sh_\Sigma^{\Sp}(\cal{E})$ is computed as a filtered colimit
    \[
        \cal{F}^\dagger(x) = \underset{y \epi x}\colim\ \cal{F}(y),
    \]
    and coincides with left Kan extension along $\pi^{\op}$ under the identifications $\cal{P}_\Sigma^{\Sp}(\cal{E}) \simeq \mathrm{Ind}(\kb(\cal{E}))$, $\sh^{\Sp}_\Sigma(\cal{E}) \simeq \mathrm{Ind}(\db(\cal{E}))$; we abuse notation by writing $\pi_!$ for each of these functors below. We claim that the square
    \begin{equation}\label{shv}\begin{tikzcd}
        \cal{P}_\Sigma^{\Sp}(\cal{E}) \ar[r] \ar[d, "\tau_{\ge 0}"] & \sh_\Sigma^{\Sp}(\cal{E}) \ar[d, "\tau_{\ge 0}"] \\
        \cal{P}^{\Sp}_\Sigma(\cal{E})_{\ge 0} \ar[r] & \sh^{\Sp}_\Sigma(\cal{E})_{\ge 0},
    \end{tikzcd}\end{equation}
    commutes. Since $\Omega^\infty:\Sp \to \cal{S}$ commutes with filtered colimits, the square
    \[\begin{tikzcd}
        \cal{P}_\Sigma^{\Sp}(\cal{E}) \ar[r] \ar[d, "\Omega^\infty"] & \sh_\Sigma^{\Sp}(\cal{E}) \ar[d, "\Omega^\infty"] \\
        \cal{P}_\Sigma(\cal{E}) \ar[r] & \sh_\Sigma(\cal{E}),
    \end{tikzcd}\]
    commutes, where $\Omega^\infty$ is computed levelwise; but by \cite[Prop.\ 2.19]{Pst23}, the adjunction
    \[\begin{tikzcd}
        \sh_\Sigma(\cal{E}) \ar[r, bend left=2ex, "\Sigma^\infty"] & \sh^{\Sp}_\Sigma(\cal{E}) \ar[l, "\Omega^\infty", "\rotatebox{90}{$\vdash$}"', bend left=2ex]
    \end{tikzcd}\]
    where again $\Omega^\infty$ is computed levelwise, has $\Sigma^\infty$ fully faithful, with essential image the connective aisle $\sh^{\Sp}_\Sigma(\cal{E})_{\ge 0}$. We thus have a factorisation
    \[\begin{tikzcd}
    		\cal{P}_\Sigma(\cal{E}) \ar[r, "\simeq"] \ar[rd, hookrightarrow, "\Sigma^\infty"'] & \cal{P}_\Sigma^{\Sp}(\cal{E})_{\ge 0} \ar[d, hookrightarrow] \\
    		& \cal{P}^{\Sp}_\Sigma(\cal{E}),
    \end{tikzcd}\]
    and so by uniqueness of adjoints, $\Omega^\infty$ is homotopic up to postcomposition with an equivalence to the truncation $\tau_{\ge 0}$ and in particular preserves filtered colimits, so that (\ref{shv}) commutes. We thus have for each excisive functor $\Lambda:\kb(\cal{E})^{\op} \to \Sp$ that $\pi_!\tau_{\ge 0}\Lambda^s \simeq \tau_{\ge 0}\pi_!\Lambda$, and in particular
    \[
        \pi_!\tau_{\ge 0}P_1\Qoppa_\oplus^s \simeq \tau_{\ge 0}\pi_!P_1\Qoppa_\oplus^s \simeq \tau_{\ge 0}P_1\Qoppa^s,
    \]
    where the last equivalence follows from Lemma \ref{bddverdier} and the fact that $P_1$ is a left adjoint. In particular, the fibre sequence
    \[
        \pi_!(\left(\hom_{\kb(\cal{E})}(-, \D(-))_{\hct}\right)(x) \to \pi_!\left(\hom_{\kb(\cal{E})}(-, \D(-))^{\hct}\right)(x) \to \pi_!\left(\hom_{\kb(\cal{E})}(-, \D(-))^{\tct}\right)(x)
    \]
    coincides for each $x\in \db(\cal{E})$ with
    \[
        \hom_{\db(\cal{E})}(x, \D(x))_{\hct} \to \hom_{\db(\cal{E})}(x, \D(x))^{\hct} \to \hom_{\db(\cal{E})}(x, \D(x))^{\tct} = P_1\Qoppa^s(x).
    \]
    Now the cartesian square defining $\Qoppa_\oplus^{\ge 0}$ is stable under the filtered colimit computing left Kan extension along the Verdier projection $\kb(\cal{E}) \to \db(\cal{E})$, and so
    \[
        \pi_!\Qoppa_\oplus^{\ge 0} = \pi_!\left[\Qoppa_\oplus^s \underset{\Lambda_\oplus^s}\times \tau_{\ge 0}\Lambda^s_\oplus \right] \simeq \left[\pi_!\Qoppa_\oplus^s\right] \underset{\left[\pi_!P_1\Qoppa_\oplus^s\right]}\times (\left[\pi_!\tau_{\ge 0}P_1\Qoppa^s_\oplus\right] \simeq \Qoppa^{\ge 0},
    \]
    and we are done.
\end{proof}
\begin{appendices}
\section{Deriving exact categories}\label{appa}
\subsection{Complicial exact categories}\label{excat}
In this section we review some results on complicial exact categories \cite[Def.\ 3.2.2]{Sch08}. These are similar in spirit to the complicial biWaldhausen categories of Thomason-Trobaugh \cite{TT90}; the notable example is that of (sub)categories of chain complexes over an exact category. We start with following observation.\\
Any additive category admits a canonical action of the closed symmetric monoidal category of finitely generated free $\zz$-modules $\projz$, encoding the direct sum. If $\cal{E}$ is exact, this promotes by virtue of the fact that conflations are stable under direct sum to a biexact functor $\otimes:\projz \times \cal{E} \to \cal{E}$ (with respect to the split-exact structure on $\projz$). It is natural to ask whether this action extends via the canonical inclusion $\projz \hookrightarrow \chb(\projz)$ to an action of the symmetric monoidal category of bounded complexes over $\projz$.\\
Write $\mathscr{C}_{\zz} := \mathrm{Ch}_\mathrm{b}(\projz)$. Then the canonical examples of such an extension is the action $\mathscr{C}_{\zz} \times \chb(\cal{E}) \to \chb(\cal{E})$, defined as follows: for bounded complexes $K_* \in \mathscr{C}_{\zz}$ and $X_* \in \chb(\cal{E})$, $K \otimes X$ is the complex given in degree $n$ by
\[
	(K \otimes X)_n = \bigoplus_i K_i \otimes X_{n-i},
\]
with differential given on the summand $K_i \otimes X_{n-i}$ by $d^K_i \otimes 1_{X_{n-i}} + (-1)^i \otimes d^X_{n-i}$. Given maps of complexes $(A_*, f_*):(K_*, X_*) \to (L_*, Y_*)$, where $A_*$ is degreewise the matrix $(a^n_{\alpha\beta})_{\alpha\beta} \in \mathrm{Mat}_{l_n\times k_n}(\zz)$, the induced map $K \otimes X \to L \otimes Y$ is given in degree $n$ on the summand $K_i \otimes X_{n-i} = X_{n-i}^{\oplus k_i}$ by the matrix $(a^i_{\alpha\beta}\cdot f_{n-i}):X^{\oplus k_i} \to Y^{\oplus l_i}$. Since the tensor product of complexes preserves connectivity, this also restricts to an action $\mathscr{C}_{\zz, \ge 0} \times \mathrm{Ch}_{\mathrm{b}, \ge 0}(\cal{E}) \to \mathrm{Ch}_{\mathrm{b}, \ge 0}(\cal{E})$.\\
Now the category $\chb(\cal{E})$ inherits a canonical exact structure from $\cal{E}$ in which conflations are defined degreewise, and moreover comes equipped with a notion of weak equivalence: a complex $X_*$ is said to be strictly acyclic if the differentials admit factorisations $d_n:X_n \stackrel{p_n}\epi Z_{n-1} \stackrel{i_{n-1}}\mono X_{n-1}$, such that the sequences $Z_n \stackrel{i_n}\mono X_n \stackrel{p_n}\epi Z_{n-1}$ are conflations in $\cal{E}$; $X_*$ is then acyclic if chain homotopy equivalent to a strictly acyclic complex. If $\cal{E}$ is weakly idempotent complete (i.e.\ any retract admits a kernel), the classes of acyclic and strictly acyclic complexes coincide \cite[Prop.\ 10.14]{Buh10}. A map of complexes $X_* \to Y_*$ is then said to be a \textit{quasi-isomorphism} if its mapping cone is acyclic. The degreewise conflations and quasi-isomorphisms equip $\chb(\cal{E})$ with the structure of an exact category with weak equivalences, and the tensor product $\otimes:(\mathscr{C}_{\zz}, \mathbf{ch.htp.}) \times (\chb(\cal{E}), \mathbf{qis}) \to (\chb(\cal{E}), \mathbf{qis})$ is an exact functor of such categories, where $\mathscr{C}_{\zz}$ is equipped with the degreewise split-exact structure and the class of corresponding quasi-isomorphisms, which coincides with the chain homotopy equivalences. The symmetric monoidal structure on $\mathscr{C}_{\zz}$ is compatible with this structure in the sense that the tensor product preserves chain homotopy equivalences in each variable and is biexact; we call such a category a symmetric monoidal exact category with weak equivalences, and note that the inclusion $\projz \hookrightarrow \mathscr{C}_{\zz}$ is monoidal.
\begin{notation}
	Write $\mathbb{1} := \zz[0]$ for the unit of the monoidal structure on $\mathscr{C}_{\zz}$, and $C$ for the complex $0 \to \zz =\joinrel= \zz \to 0$ in degrees $[0,1]$, sitting in the conflation $\mathbb{1} \mono C \epi T$. Dually, for $P := C[-1]$, we have a conflation $\Omega \mono T \epi \mathbb{1}$. Write $[-, -]$ for the mapping complex in $\mathscr{C}_{\zz}$, given degreewise by
	\[
		[X, Y]_n = \prod_i \Hom_{\projz}(X_i, Y_{i+n}),
	\]
	with differential $\partial(f) = d \circ f - (-1)^{|f|}f \circ d$. The closed symmetric monoidal structure $(\mathscr{C}_{\zz}, \otimes, [-,-], \mathbb{1})$ induces for each $n \in \zz$ a strong duality $K \mapsto [K, \mathbb{1}[n]]$, with double dual identification $\mathrm{can}:K \to [[K, \mathbb{1}[n]], \mathbb{1}[n]]$ the adjunct of the evaluation map $\mathrm{ev}_{\mathbb{1}[n]}:[K, \mathbb{1}[n]] \otimes K \to \mathbb{1}[n]$.
\end{notation}
\begin{definition}
	A \textit{complicial} structure on an exact category with weak equivalences $(\cal{E}, w)$ is the data of a biexact action of the symmetric monoidal exact category with weak equivalences $(\mathscr{C}_{\zz}, \mathbf{ch.htp}, \otimes, \mathbb{1})$, i.e.\ a functor
	\[
		\otimes:\mathscr{C}_{\zz} \times \cal{E} \to \cal{E}
	\]
	which is exact in each variable as a functor of exact categories with weak equivalences, and is coherently unital and associative in the sense of \cite{Gra76}. In particular, we note that given maps $f_*:K_* \to L_*$ in $\mathscr{C}_{\zz}$ and $p:x \to y$ in $\cal{E}$, the induced map $f \otimes p:K \otimes x \to L \otimes y$ lies in $w$ as soon as $f$ is a chain homotopy equivalence and $p \in w$. An exact functor of complicial exact categories with weak equivalences $f:(\cal{E}, w) \to (\cal{E}', w')$ is simply an exact functor of the underlying exact categories with weak equivalences; $f$ is said to be complicial exact if it is compatible with the action of $\mathscr{C}_{\zz}$ in the sense that there is a natural isomorphism rendering the diagram
	\[\begin{tikzcd}
		\mathscr{C}_{\zz} \times \cal{E} \ar[r, "\otimes"] \ar[d, "1 \otimes f"] & \cal{E} \ar[d, "f"] \\
		\mathscr{C}_{\zz} \times \cal{E}' \ar[r, "\otimes'"] & \cal{E}'
	\end{tikzcd}\]
	commutative, satisfying the constraints of \cite{Gra76}. We will generally abuse terminology by referring to an exact category with a complicial structure as a complicial exact category. For $K \in \mathscr{C}_{\zz}$ and $x \in \cal{E}$, we may write $Kx := K \otimes x$.
\end{definition}
We may refine the above definition to the case where $(\cal{E}, w, \D, \eta)$ is an exact category with weak equivalences and duality; recall that this implies in particular that $\D:\cal{E}^{\op} \to \cal{E}$ is an exact functor of exact categories with weak equivalences.
\begin{definition}\label{comdual}
	 An exact category with weak equivalences and duality $(\cal{E}, w, \D, \eta)$ is complicial if the tensor product $\otimes:\mathscr{C}_{\zz} \times \cal{E} \to \cal{E}$ refines to a duality-preserving form functor of exact categories with duality, i.e.\ if the diagram
    \[\begin{tikzcd}
        \mathscr{C}_{\zz}^{\op} \times \cal{E}^{\op} \ar[r, "{[-,\mathbb{1}] \otimes \D}"] \ar[d, "\otimes^{\op}"] & |[alias=cze]|\mathscr{C}_{\zz} \times \cal{E} \ar[d, "\otimes"] \\
        |[alias=e]|\cal{E}^{\op} \ar[r, "\D"] & \cal{E}
        \ar[Rightarrow,from=cze, to=e,shorten >=5mm,shorten <=5mm]
    \end{tikzcd}\]
    commutes up to indicated natural isomorphism. That this is the case for $\mathscr{C}_{\zz}$ with complicial structure induced by monoidality, with duality $X \mapsto [X, \mathbb{1}]$, follows from the natural isomorphism
    \[
        [X, A] \otimes [Y, B] \to [X \otimes Y, A \otimes B], \quad f \otimes g \mapsto (f \otimes g: x \otimes y \mapsto (-1)^{|x|||g|}f(x) \otimes g(y)).
    \]
\end{definition}
\begin{remark}
	We note that the opposite of a complicial exact category is complicial exact, with structure map furnished by the duality in $\mathscr{C}_{\zz}$:
	\begin{equation}\label{opcomp}
		\mathscr{C}_{\zz} \times \cal{E}^{\op} \xto{[-,\mathbb{1}] \times 1_{\cal{E}}} \mathscr{C}_{\zz}^{\op} \times \cal{E}^{\op} \xto{\otimes^{\op}} \cal{E}^{\op}.
	\end{equation}
\end{remark}
\begin{recollection}
	For $\cal{E}$ any exact category, an object $z \in \cal{E}$ is said to be injective if for each inflation $i:x \mono y$, the map $i^*:\Hom_{\cal{E}}(y, z) \to \Hom_{\cal{E}}(x, z)$ is a surjection; dually, $z$ is projective if for any deflation $p:y \epi x$, the map $p_*\Hom_{\cal{E}}(z, y) \to \Hom_{\cal{E}}(z, x)$ is a surjection. $\cal{E}$ is said to have enough injectives (projectives) if each $x \in \cal{E}$ admits an inflation $x \mono I_x$ into an injective object (a deflation $P_x \epi x$ from a projective object). A \textit{Frobenius exact category} is then an eaxct category with enough injectives and projectives, and in which the classes of injectives and projectives coincide.\\
	Recall \cite[Lem.\ A.2.16]{Sch08} that any complicial exact category has an underlying Frobenius exact category, as follows: declare an inflation $i:x \mono y$ to be Frobenius if for each $u \in \cal{E}$, the map $i^*:\Hom_{\cal{E}}(y, C\otimes u) \to \Hom_{\cal{E}}(x, C \otimes u)$ is surjective. Dually, a deflation $p:y \epi x$ is Frobenius if for each $u \in \cal{E}$, $p_*:\Hom_{\cal{E}}(C\otimes u, y) \to \Hom_{\cal{E}}(C \otimes u, x)$ is a surjection. The Frobenius inflations and deflations equip $\cal{E}$ with a Frobenius exact structure, in which an object is injective-projective if it is a retract of $C \otimes u$ for some $u$.
\end{recollection}
\begin{definition}
	Given a Frobenius exact category $\cal{E}$, call morphisms $f, g:x \to y$ Frobenius homotopic if their difference factors through an injective-projective object, and call a map $f:x \to y$ a Frobenius equivalence if it admits a two-sided inverse up to Frobenius homotopy. Frobenius homotopy then defines an equivalence relation on $\Hom_{\cal{E}}(x, y)$ for each $x, y\in \cal{E}$, and the corresponding quotient category is the \textit{stable category} $\underline{\cal{E}}$. This can be equipped with a canonical triangulated structure \cite[\S2.14]{Sch08}, as follows: choose for each $x$ an inflation into an injective-projective $x \mono I_x$; then a triangle in $\underline{\cal{E}}$ is distinguished if it is isomorphic to one of the form
	\[
		 x \xto{f} y \to I_x \cup_x y \to I_x/x.
	\]
\end{definition}
\begin{example}
	The Frobenius exact structure on $\chb(\cal{E})$ arising from its complicial exact structure comprises the degreewise-split monomorphisms and epimorphisms. A map of complexes is then a Frobenius equivalence if and only if it is a chain homotopy equivalence.
\end{example}
\begin{remark}\label{frobhtp}
	Any complicial exact category $\cal{E}$ supports a notion of simplicial homotopy: call maps $f, g:x \to y$ simplicially homotopic if there exists a map $H:\Delta^1 \otimes x \to y$ rendering the diagram
	\[\begin{tikzcd}
		x \ar[d, "d^1"'] \ar[rrd, bend left=2ex, "f"] \\
		\Delta^1x \ar[rr, "H"] && y \\
		x \ar[u, "d^0"] \ar[rru, bend right=2ex, "g"']
	\end{tikzcd}\]
	commutative; here we write $\Delta^1$ for the image of the standard 1-simplex under the functor $N_*\zz[-]:\sset \to \mathscr{C}_{\zz}$, and $d^0, d^1$ respectively for the maps induced by the inclusions $\begin{psmallmatrix} 0\\ 1\end{psmallmatrix}, \begin{psmallmatrix}1 \\ 0\end{psmallmatrix}:\zz[0] = \Delta^0 \to \Delta^1$. Maps $f, g:x \to y$ are simplicially homotopic if and only if they are Frobenius homotopic: a Frobenius homotopy $f-g:x \to Cu \to y$, factors by injectivity of $Cu$ over the Frobenius inflation $x \mono Cx$ induced by the map $\mathbb{1} \mono C$. With reference to the retract diagram
	\[\begin{tikzcd}
		C \ar[d, "i"] & 0 \ar[r] & \zz \ar[r, equal] \ar[d, equal] & \zz \ar[r] \ar[d, "{\begin{psmallmatrix}-1 \\ 1\end{psmallmatrix}}"] & 0 \\
		\Delta^1 \ar[d, "p"] & 0 \ar[r] & \zz \ar[r, "{\begin{psmallmatrix}-1 \\ 1\end{psmallmatrix}}"] \ar[d, equal] & \zz \oplus \zz \ar[r] \ar[d, "{\begin{psmallmatrix}-1 \amsamp 0\end{psmallmatrix}}"] & 0 \\
		C & 0 \ar[r] & \zz \ar[r, equal] & \zz \ar[r] & 0,
	\end{tikzcd}\]
	we then have a simplicial homotopy
	\[\begin{tikzcd}
		x \ar[d, "d^1"] \ar[rrd, bend left=2ex, "g-f"] \\
		\Delta^1 \ar[r, "p"] & Cx \ar[r] & y, \\
		x \ar[u, "d^0"] \ar[rru, bend right=2ex, "0"']
	\end{tikzcd}\]
	and hence from $f$ to $g$ by the evident additivity of simplicial homotopy. Conversely, given a simplicial homotopy $f \sim g$, the difference $d^0-d^1:x \to \Delta^1x$ factors through the inclusion $Cx \to \Delta^1x$, and we have a factorisation $g-f:x \xto{d^0-d^1} Cx \xto{Hi} y$.
\end{remark}
Let $(\cal{E}, w)$ be an exact category with weak equivalences. An object $x \in \cal{E}$ is said to be ($w$-)\textit{acyclic} if the unique map $0 \to x$ is a weak equivalence. By \cite[Lem.\ 7.1]{Sch24}, an inflation $x \mono y$ is a weak equivalence if and only if its cokernel is acyclic, and dually a deflation $y \epi x$ is a weak equivalence if and only if its kernel is acyclic (we call such inflations or deflations ($w$-)trivial). Accordingly, the full subcategory $\cal{E}^w \subset \cal{E}$ of acyclic objects is exact and closed under extensions in $\cal{E}$.\\
Write $(w)\mathrm{Inf}$ and $(w)\mathrm{Def}$ for the classes of (trivial) inflations and deflations in $\cal{E}$. Then the pairs $(w\mathrm{Inf}, \mathrm{Def})$ and $(\mathrm{Inf}, w\mathrm{Def})$ each form a functorial factorisation system on $\cal{E}$: given a map $f:x \to y$, we have a commutative diagram
\begin{equation}\begin{tikzcd}\label{factsys}
	x \ar[r, rightarrowtail, "{\begin{psmallmatrix}\iota_x \\ f\end{psmallmatrix}}"] \ar[d, rightarrowtail, "{\begin{psmallmatrix}1_x \\ 0\end{psmallmatrix}}"', "\rotatebox{90}{$\sim$}"] \ar[rd, "f"'] & Cx \oplus y \ar[d, "{\begin{psmallmatrix}0 \amsamp 1_y\end{psmallmatrix}}", "\rotatebox{90}{$\sim$}"', twoheadrightarrow] \\
	x \oplus Py \ar[r, "{\begin{psmallmatrix}0 \amsamp \pi_y \end{psmallmatrix}}"', twoheadrightarrow] & y,
\end{tikzcd}\end{equation}
where we write $\iota_x:x \mono Cx$ and $\pi_y:Py \epi y$ for the maps induced respectively by $\mathbb{1} \mono C$ and $P \epi \mathbb{1}$ in $\mathscr{C}_{\zz}$. Note that each vertical map is a Frobenius equivalence by contractibility of the respective kernel or cokernel. 
\begin{remark}\label{Frob-in-w}
	For a map $f:x \to y$ in $\cal{E}$, define the mapping cone $C(f)$ as the pushout
\[\begin{tikzcd}
	x \ar[r, "f"] \ar[d, rightarrowtail, "\iota_x"'] & y \ar[d, rightarrowtail] \\
	Cx \ar[r] & C(f),
\end{tikzcd}\]
and dually define the mapping cocone $F(f)$ as the pullback
\[\begin{tikzcd}
	F(f) \ar[r] \ar[d, twoheadrightarrow] & Py \ar[d, "\pi_y", twoheadrightarrow] \\
	x \ar[r, "f"] & y.
\end{tikzcd}\]
	From the diagram
	\[\begin{tikzcd}
	& Cx \ar[d, rightarrowtail] \\
		x \ar[r, rightarrowtail] \ar[rd, "f"'] & Cx \oplus y \ar[d, "\rotatebox{270}{$\sim$}", twoheadrightarrow] \ar[r, twoheadrightarrow] & C(f) \\
		& y
	\end{tikzcd}\]
	in which the horizontal and vertical sequences are conflations, we see that $f \in w\cal{E}$ if and only if $C(f)$ is acyclic, and dualising, if and only if $F(f)$ is acyclic. Given a Frobenius equivalence $f:x\to y$, it follows from the map of distinguished triangles
	\[\begin{tikzcd}
		x \ar[d, equal] \ar[r, equal] & x \ar[r] \ar[d, "f"] & Cx \cong 0 \ar[r, "f"] \ar[d] & Tx \ar[d, equal] \\
		x \ar[r, "f"] & y \ar[r] & C(f) \ar[r] & Tx,
	\end{tikzcd}\]
	in $\underline{\cal{E}}$ that $C(f) \cong 0$ in $\underline{\cal{E}}$, so that $C(f)$ is a retract of some $Cu$ in $\cal{E}$; but $Cu$ is acyclic since the map $0 \to Cu$ is the tensor product $(0 \to C) \otimes 1_u$ of weak equivalences, and $w\cal{E}$ is closed under retracts, so that any Frobenius equivalence lies in $w\cal{E}$.
\end{remark}
Writing $w_\mathrm{Frob} \subset w$ for the Frobenius equivalences, the pair $(\cal{E}, w_\mathrm{Frob})$ is again an exact category with weak equivalences: given a Frobenius equivalence $f:x \to z$, and a deflation $q:z' \epi z$, we have by virtue of the factorisation system $(w\mathrm{Inf}, \mathrm{Def})$ a diagram of cartesian squares
\[\begin{tikzcd}
	x' \ar[d, twoheadrightarrow] \ar[r, rightarrowtail, "i'"] & y' \ar[d, twoheadrightarrow] \ar[r, twoheadrightarrow, "p'"] & z' \ar[d, twoheadrightarrow] \\
	x \ar[r, rightarrowtail, "i"] & y \ar[r, twoheadrightarrow, "p"] & z,
\end{tikzcd}\]
with the left-hand square bicartesian. $\ker(p') \cong \ker(p)$ and $\coker(i') \cong \coker(i)$ are thus Frobenius contractible, so that the pullback $f' = p'i'$ of $f$ along $q$ is a Frobenius equivalence. The argument for stability under pushouts is dual, and the closure of $w_\mathrm{Frob}$ follows from the closure of $\cal{E}^w$ under retracts and the fact that $f:x \to y$ is a Frobenius equivalence precisely when its cone is acyclic. We accordingly have exact inclusions
\[
	(\cal{E}_\mathrm{Frob}, w_\mathrm{Frob}) \subset (\cal{E}, w_\mathrm{Frob}) \subset (\cal{E}, w)
\]
of exact categories with weak equivalences.\\
For $n \ge 0$, write $\Delta^n \in \mathscr{C}_{\zz}$ for image of the standard $n$-simplex under the composite $\sset \xto{\zz[-]} \sab \xto{N} \mathrm{Ch}_{\ge 0}(\zz)$. Then we have a mapping space bifunctor
\[
	\cal{E}^{\op} \times \cal{E} \to \sab, \quad (x, y) \mapsto \Map_\Delta(x, y) := \Hom_{\cal{E}}(\Delta^\bullet x, y),
\]
equipping $\cal{E}$ with a simplicial enrichment such that the resulting simplicial category $\cal{E}_\Delta$ is fibrant in the Bergner model structure. By Remark \ref{frobhtp}, we have $\pi_0\Map_\Delta(x, y) \cong \Hom_{\underline{\cal{E}}}(x, y)$, and writing $L_\mathrm{Frob}(\cal{E})$ for the Dwyer-Kan localisation at the Frobenius equivalences, we have a factorisation
\[\begin{tikzcd}
	\cal{E} \ar[rd] \ar[r, "\gamma"] & L_\mathrm{Frob}(\cal{E}) \ar[d, dashed, "f"] \\
	& N_\Delta(\cal{E}_\Delta),
\end{tikzcd}\]
with $f$ clearly essentially surjective, and the functor $\cal{E} \to N_\Delta(\cal{E}_\Delta)$ induced by the inclusion of $0$-simplices. Now it follows from Proposition \ref{ftil} that $\Map_\Delta(-,-)$ sends Frobenius equivalences in each variable to homotopy equivalences. Applying Proposition \ref{hoco} to the functors $\Map_\Delta(-,y)$ and $\Map_\Delta(x, -)$ yields natural equivalences
\[
	\Map_\Delta(x, y) \simeq \underset{x' \in J^{\op}_{x, \mathrm{Frob}}}\colim\ \Hom(x', y),
\]
where we write $J_{x, \mathrm{Frob}} \subset (\cal{E} \downarrow x)$ for the full subcategory spanned by trivial Frobenius deflations over $x$ which lie in $w_\mathrm{Frob}$, and the colimit is take in $\cal{S}$. Since the maps $\Hom_{\cal{E}}(x, y) \to \Map_{L_\mathrm{Frob}(\cal{E})}(\gamma(x), \gamma(y))$ and $\Hom_{\cal{E}}(x, y) \to \Map_\Delta(x, y)$ each identify with the canonical inclusions
\[
	\Hom_{\cal{E}}(x, y) \to \underset{x' \in J_{x, \mathrm{Frob}}^{\op}}\colim\ \Hom_{\cal{E}}(x', y),
\]
we see that the functor $f$ is fully faithful, i.e.\ $N_\Delta(\cal{E}_\Delta)$ is a model for the Dwyer-Kan localisation $L_\mathrm{Frob}(\cal{E})$.
\begin{proposition}\label{comstab}
	The Dwyer-Kan localisation of any complicial exact catgeory with weak equivalences is a stable $\infty$-category. Moreover, an exact functor between such induces an exact functor of stable $\infty$-categories.
\end{proposition}
\begin{proof}
	This is in spirit the same as \cite[Prop.\ 2.7]{BC20}, but we provide the details. We first note that $(\cal{E}, w)$ is an $\infty$-category of fibrant objects (Lemma \ref{fib}), so that by \cite[Prop.\ 7.5.6]{Cis19}, the localisation $L_w(\cal{E})$ admits finite limits, and the localisation $\gamma:\cal{E} \to L_w(\cal{E})$ is left exact. The same argument applied to the complicial exact category with weak equivalences $(\cal{E}^{\op}, w^{\op})$ furnishes us with finite colimits, and since $0$ is both initial and final, $L_w(\cal{E})$ is in addition pointed.\\
	Now each map in $L_w(\cal{E})$ is up to equivalence the image of an inflation under $\gamma$, and hence every cofibre sequence is up to equivalence the image under $\gamma$ of a pushout
	\[\begin{tikzcd}
		x \ar[r, rightarrowtail] \ar[d, twoheadrightarrow] & y \ar[d, twoheadrightarrow] \\
		0 \ar[r, rightarrowtail] & z.
	\end{tikzcd}\]
	Since $x \mono y \epi z$ is a conflation, this is also a pullback and hence localises to a fibre sequence in $L_w(\cal{E})$. The same argument applied to $(\cal{E}^{\op}, w^{\op})$, yields that fibre and cofibre sequences coincide in $L_w(\cal{E})$, which is then stable. For the last claim, an exact functor $(\cal{E}, w) \to (\cal{E}', w')$ between complicial exact categories with weak equivalences preserves zero objects and conflations by definition, and hence by \cite[Prop.\ 7.5.6]{Cis19} and its dual descends to an exact functor $L_w(\cal{E}) \to L_{w'}(\cal{E}')$.
\end{proof}
\begin{remark}
    By the universal property of the derived $\infty$-category \cite[Cor.\ 7.4.12]{BCKW19}, the localisation $\gamma:\cal{E} \to L_w(\cal{E})$ factors over the canonical inclusion $\cal{E} \hookrightarrow \db(\cal{E})$. Write $\cal{E}^w \subset \cal{E}$ for the inclusion of the exact subcategory of $w$-acyclics, inducing an inclusion $\db(\cal{E})^w) \subset \db(\cal{E})$. Then we claim that the functor $\db(\cal{E}) \to L_w(\cal{E})$ exhibits $L_w(\cal{E})$ as the Verdier quotient $\db(\cal{E})/\db(\cal{E}^w)$: write $\fun^\mathrm{ex}_w(\cal{E}, \cal{C})$ for the full subcategory of exact functors inverting weak equivalences, and $\fun^\mathrm{ex}_{\db(\cal{E}^w)}(\db(\cal{E}), \cal{C})$ for the full subcategory of functors annihilating $\db(\cal{E}^w)$. We claim that $\fun_w^\mathrm{ex}(\cal{E}, \cal{C})$ coincides with the image of $\fun^\mathrm{ex}_{\db(\cal{E}^w)}(\db(\cal{E}), \cal{E}) \subset \fun^\mathrm{ex}(\db(\cal{E}), \cal{C}) \simeq \fun^\mathrm{ex}(\cal{E}, \cal{C})$: given a functor $F:\cal{E} \to \cal{C}$ inverting $w$, then since any $w$-acyclic $x$ arises as the kernel of a trivial deflation, say $\begin{psmallmatrix}0 & 1\end{psmallmatrix}:x \oplus x \to x$, we have that the composite $\cal{E}^w \subset \cal{E} \to \cal{C}$ vanishes, and accordingly the induced functor $\db(\cal{E}^w) \to \cal{C}$ also. Conversely, given a functor $G:\db(\cal{E}) \to \cal{C}$ vanishing on $\db(\cal{E}^w)$, the composite $\cal{E}^w \subset \db(\cal{E}) \to \cal{C}$ vanishes, and since any weak equivalence in $\cal{E}$ can be written as the composite of a trivial inflation followed by a trivial deflation, of which the cokernel and kernel vanish in $\cal{C}$, we see that $G\mid_{\cal{E}}$ inverts maps in $w$. Accordingly, the horizontal map
    \[\begin{tikzcd}
        \fun^\mathrm{ex}(L_w(\cal{E}), \cal{C}) \ar[r] \ar[rd, "\simeq"] & \fun^\mathrm{ex}_{\db(\cal{E}^w)}(\db(\cal{E}), \cal{C}) \ar[d, "\rotatebox{270}{$\simeq$}"] \\
        & \fun_w^\mathrm{ex}(\cal{E}, \cal{C})
    \end{tikzcd}\]
    is an equivalence by 2-of-3.
\end{remark}
\begin{lemma}\label{map}
	The bifunctor $\Map_\Delta(-, -):\cal{E}^{\op} \times \cal{E} \to \sab$ sends Frobenius conflations in each variable to homotopy fibre sequences.
\end{lemma}
\begin{proof}
	Let $x \mono y \epi z$ in $\cal{E}$ be a Frobenius conflation. There is a diagram of cosimplicial conflations
	\begin{equation}\begin{tikzcd}\label{seqs}
		\tilde{x}^\bullet \ar[r, rightarrowtail] \ar[d, rightarrowtail] & \tilde{y}^\bullet \ar[r, twoheadrightarrow] \ar[d, rightarrowtail] & \tilde{z}^\bullet \ar[d, rightarrowtail] \\
		\Delta^\bullet x \ar[r, rightarrowtail] \ar[d, twoheadrightarrow] & \Delta^\bullet y \ar[r, twoheadrightarrow] \ar[d, twoheadrightarrow] & \Delta^\bullet z \ar[d, twoheadrightarrow] \\
		cx \ar[r, rightarrowtail] & cy \ar[r, twoheadrightarrow] & cz,
	\end{tikzcd}\end{equation}
	where the vertical conflations are induced levelwise by
	\[
		\tilde{\Delta}^n \mono \Delta^n \epi \Delta^0
	\]
	in $\mathscr{C}_{\zz}$. Since $\tilde{x}^n$, $\tilde{y}^n$, $\tilde{z}^n$ are Frobenius contractible, the vertical conflations and the top row split. For any $v \in \cal{E}$, the induced diagram
	\[\begin{tikzcd}
		c\Hom_{\cal{E}}(z, v) \ar[r, rightarrowtail] \ar[d, rightarrowtail] & c\Hom_{\cal{E}}(y, v) \ar[r, twoheadrightarrow] \ar[d, rightarrowtail] & c\Hom_{\cal{E}}(x, v) \ar[d, rightarrowtail] \\
		\Hom_{\cal{E}}(\Delta^\bullet z, v) \ar[r, rightarrowtail] \ar[d, twoheadrightarrow] & \Hom_{\cal{E}}(\Delta^\bullet y, v) \ar[r, twoheadrightarrow] \ar[d, twoheadrightarrow] & \Hom_{\cal{E}}(\Delta^\bullet x, v) \ar[d, twoheadrightarrow] \\
		\Hom_{\cal{E}}(\tilde{z}^\bullet, v) \ar[r, rightarrowtail] & \Hom_{\cal{E}}(\tilde{y}^\bullet, v) \ar[r, twoheadrightarrow] & \Hom_{\cal{E}}(\tilde{x}^\bullet, v),
	\end{tikzcd}\]
	has the lower row and each column a homotopy fibre sequence in simplicial abelian groups (in fact, in $\hzmod$). The top row is clearly a homotopy fibre sequence, since maps of discrete simplicial sets are Kan fibrations, and hence so is the middle row. The same holds for $\Map_\Delta(v, -)$.
\end{proof}
\begin{lemma}\label{ssec}
	The space of sections of any trivial Frobenius deflation is a contractible Kan complex.
\end{lemma}
\begin{proof}
    The space of such sections is nonempty since $\ker(p)$ is Frobenius contractible and hence injective-proj; since $p_*$ is split by $s_*$ for any such section, we see from above that $p_*:\Map_\Delta(x, y) \to \Map_\Delta(x, x)$ is a degreewise surjective homotopy equivalence of simplicial abelian groups, and hence a trivial fibration in the classical model structure on $\sab$. Accordingly, the pullback
	\[\begin{tikzcd}
		\mathrm{Sec}(p) \ar[r] \ar[d] & \Map_\Delta(x, y) \ar[d, "p_*", "\rotatebox{270}{$\sim$}"', twoheadrightarrow] \\
		\Delta^0 \ar[r, "1_x"] & \Map_\Delta(x, x)
	\end{tikzcd}\]
	is homotopy cartesian, and the left vertical arrow is thus a trivial fibration.
\end{proof}
For $(\cal{E}, w)$ a complicial exact category with duality and $x \in \cal{E}$, write $J_x \subset I_x \subset (\cal{E}\downarrow x)$ for the subcategories of the slice category spanned by trivial deflations resp.\ weak equivalences over $x$.
\begin{remark}\label{ixjx}
    Recall Lemma \ref{fib} that $\cal{E}$ is naturally an $\infty$-category of fibrant objects in the sense of \cite{Cis19}. For each $x \in \cal{E}$, the category $(\cal{E} \downarrow x)$ acquires the structure of an $\infty$-category with weak equivalences and fibrations by declaring a map $y \to y'$ over $x$ to be a fibration resp.\ weak equivalence in $(\cal{E} \downarrow x)$ if its image under the canonical right fibration
    \[
        (\cal{E} \downarrow x) \to \cal{E}
    \]
    is. The subcategory of acyclic objects $(\cal{E} \downarrow x)^w$ then coincides with $I_x$, which inherits the structure of an $\infty$-category with weak equivalences and fibrations, in which a fibrant object is precisely an object of $J_x$. The inclusion $J_x \subset I_x$ is then final by \cite[Prop.\ 7.6.8]{Cis19}.
\end{remark}
Write $\gamma:\cal{E} \to L_\mathrm{Frob}(\cal{E})$ for the Dwyer-Kan localisation at the Frobenius equivalences. Writing $L_\mathrm{Frob}(\cal{E})^w$ for the essential image under $\gamma$ of $\cal{E}^w$, we see that $L_\mathrm{Frob}(\cal{E})^w \subset L_\mathrm{Frob}(\cal{E})$ is a stable subcategory (since the inclusion $\cal{E}^w \subset \cal{E}$ is exact), closed under retracts by virtue of the fact that $w$ is closed under retracts. By for instance \cite[Prop.\ A.1.9]{CD23b}, the functor $L_\mathrm{Frob}(\cal{E})^w \to L_\mathrm{Frob}(\cal{E})$ is a Verdier inclusion, and we write $\pi:L_\mathrm{Frob}(\cal{E}) \to L_w(\cal{E})$ for the cofibre in $\cat^\mathrm{st}_\infty$.
\begin{proposition}\label{rf}
	Given a functor $F:\cal{E}^{\op} \to \Ab$, for $\cal{E} = (\cal{E}, w)$ complicial exact with weak equivalences, the right derived functor $\mathbf{R}F:L_w(\cal{E})^{\op} \to \mathrm{D}_{\ge 0}(\zz)$ satisfies
	\[
		\mathbf{R}F(x) \simeq \underset{J_x^{\op}}\colim\ F
	\]
    in $\mathrm{D}_{\ge 0}(\zz)$.
\end{proposition}
\begin{proof}
	By definition, $\mathbf{R}F$ is the left Kan extension of the composite $\cal{E}^{\op} \to \sab \to \mathrm{D}_{\ge 0}(\zz)$  along the localisation $\cal{E} \to L_w(\cal{E})$. Since $F^{\Delta^\bullet}$ is a model for the derived functor of $F$ along $\gamma$, we may consider the left Kan extension of the latter along along the Verdier quotient $\pi:L_{\mathrm{Frob}}(\cal{E}) \to L_w(\cal{E})$. The structure of an $\infty$-category with weak equivalences and fibrations is inherited by the subcategory of weak equivalences $w\cal{E} \subset \cal{E}$, and so Corollary \ref{cof} implies that for $x \in \cal{E}$, the functor
	\[
		I_x^{\op} = (w\cal{E} \downarrow x)^{\op} \to (L_{\mathrm{Frob}}(w\cal{E})\downarrow \gamma(x))^{\op}
	\] 
	identifies with the localisation $I_x^{\op} \to L(w\cal{E} \downarrow x)^{\op}$ and in particular is cofinal. The formula for left Kan extension along a Verdier projection from \cite[Th.\ I.3.3]{NS18} then gives
	\begin{align*}
		\mathbf{R}F \simeq & \underset{L_\mathrm{Frob}(\cal{E}^w)^{\op}}\colim\ F^{\Delta^\bullet} \simeq \underset{L(I_x)^{\op}}\colim\ F^{\Delta^\bullet} \\
		\simeq & \underset{I_x^{\op}}\colim\ F^{\Delta^\bullet} \stackrel{\ref{ixjx}}\simeq \underset{J_x^{\op}}\colim\ F^{\Delta^\bullet} \\
		\stackrel{\ref{hoco}}\simeq & \underset{(y \stackrel{\sim}\epi x) \in J_x^{\op}}\colim\ \underset{J_{y, \mathrm{Frob}}^{\op}}\colim\ F \simeq \underset{J_x^{\op}}\colim\ F,
	\end{align*}
	where as in \ref{hoco} we denote by $J_{y, \mathrm{Frob}}$ the subcategory of the slice spanned by trivial Frobenius deflations over $y$.
\end{proof}
\begin{recollection}
	Suppose $\cal{C}$ is a small category, and $F:\cal{C} \to \cat$ is a pseudofunctor. The \textit{Grothendieck construction} on $F$ is the category $\int_{\cal{C}}F$, with objects pairs $(x \in \cal{C}, a \in F(x))$, and maps
	\[
		(x, a) \xto{(f, \gamma)} (y, b)
	\]
	the data of arrows $f:x \to y$, and $\gamma:F(f)(a) \to b$ in $F(y)$. The composition
	\[
		(x, a) \xto{(f, \gamma)} (y, b) \xto{(g, \delta)} (z, c)
	\]
	is $(gf, \delta \circ F(g)(\gamma)):(x, a) \to (z, c)$. There is a tautological functor $\int_{\cal{C}}F \to \cal{C}$, $(x, a) \mapsto x$, which on nerves is a right fibration.
\end{recollection}
\begin{proposition}\label{global}
	Suppose given a complicial exact category with weak equivalences $(\cal{E}, w)$, and some simplicial presheaf $F:\cal{E}^{\op} \to \sset$. Then there is a weak equivalence
	\[
		\underset{w\cal{E}^{\op}}\colim\ F \to \underset{w\cal{E}^{\op}}\colim\ \mathbf{R}F.
	\]
\end{proposition}
\begin{proof}
	This follows from the Fubini theorem for unstraightening: given a functor $F:\cal{C} \to \cat$ classifying a left fibration $\int_{\cal{C}}F \to \cal{C}$, and a functor $G: \int_{\cal{C}}F \to \sset$, there is a natural weak equivalence
	\begin{equation}\label{thom}
		\underset{\int_{\cal{C}}F}\hocolim\ G \to \underset{x \in \cal{C}}\hocolim\ \underset{F(x)}\hocolim\ G_x,
	\end{equation}
	where $G_x$ is the restriction of $G$ to the fibre $F(x)$, and the homotopy colimit is with respect to the Kan-Quillen model structure on $\sset$; see for instance \cite[Th.\ 26.8]{CS02} (or \cite[Th.\ 4.4]{Du23} for a $\infty$-categorical generalisation). Consider now the pseudofunctor
	\[
		J_\bullet^{\op}:w\cal{E}^{\op} \to \cat, \quad x \mapsto J_x^{\op},
	\]
	where a map $f:x \to y$ induces a functor $J_y^{\op} \to J_x^{\op}$ given by pullback along $f$. The Grothendieck construction on this functor has objects pairs $(x, y \stackrel{\sim}\twoheadrightarrow x)$, and maps
	\[
		(x, y \stackrel{\sim}\twoheadrightarrow x) \to (x', y' \stackrel{\sim}\twoheadrightarrow x')
	\]
	given by a pair $(f, \gamma)$, where $f:x' \xto{\simeq} x$ is a weak equivalence in $\cal{E}$, and $\gamma: y' \to x' \times_x y$ is a map over $y$ in $\cal{E}$; by the universal property of the pullback, this is the data of a commutative diagram
	\[\begin{tikzcd}
	y' \ar[r, "\sim"] \ar[d, "\rotatebox{270}{$\sim$}", twoheadrightarrow] & x'. \ar[d, "\rotatebox{270}{$\sim$}", twoheadrightarrow] \\
	y \ar[r, "\sim"] & x
	\end{tikzcd}\]
	The category $\int_{w\cal{E}^{\op}}J_\bullet^{\op}$ thus identifies with the full subcategory of $\mathrm{Ar}(w\cal{E})^{\op}$ spanned by the trivial deflations. The degeneracy $s_0 : w\cal{E} \subset \mathrm{Ar}(w\cal{E})$, $x \mapsto (x =\joinrel= x)$ admits a right adjoint given by the source functor $d_1$, and this restricts upon taking opposites to an adjoint pair
	\[\begin{tikzcd}
		w\cal{E}^{\op} \ar[r, shift left=.6ex, bend left=.75ex, "s_0^{\op}", "\rotatebox{270}{$\vdash$}"'] & \int_{w\cal{E}^{\op}} J_\bullet^{\op} \ar[l, shift left=.6ex, bend left=.75ex, "d_1^{\op}"],
	\end{tikzcd}\]
	so in particular the functor $w\cal{E}^{\op} \subset \int_{w\cal{E}^{\op}} J_\bullet^{\op}$ is cofinal. The weak equivalence (\ref{thom}) and Proposition \ref{rf}, and \cite[Th.\ 4.2.4.1]{HTT} then give a natural equivalence
	\[
		\underset{w\cal{E}^{\op}}\colim\ F \xto{\simeq} \underset{\int_{w\cal{E}^{\op}}J_\bullet^{\op}}\colim\ F \xto{\simeq} \underset{x \in w\cal{E}^{\op}}\colim\ \underset{J_x^{\op}}\colim\ F = \underset{w\cal{E}^{\op}}\colim\ \mathbf{R}F,
	\]
	and chasing through the various identifications we see that this is induced by the canonical map $F \Rightarrow \mathbf{R}F$.
\end{proof}
\subsection{The bounded derived $\infty$-category of an exact category}\label{db}
Given a small exact category $\cal{E}$, the classical construction of the (bounded) derived category goes as follows: we first embed $\cal{E}$ into its category of bounded chain complexes, equipped with the class of quasi-isomorphisms as weak equivalences. Taking the quotient by chain homotopy equivalence, we obtain the homotopy category $\cal{K}_\mathrm{b}(\cal{E})$, which has a canonical triangulated structure. We may then take the Verdier quotient by the triangulated subcategory $\cal{A}\mathrm{c_b}(\cal{E}) \subset \cal{K}_\mathrm{b}(\cal{E})$ of acyclic complexes \cite[\S10]{Buh10} to obtain the bounded derived category $\db(\cal{E})$.\\
The homotopy-coherent refinement of the above goes as follows: taking the embedding $\cal{E} \subset \chb(\cal{E})$ as before, we note that $\chb(\cal{E})$ is a complicial exact category, with Frobenius localisation
\[
	\chb(\cal{E}) \to L_{\mathrm{Frob}}(\chb(\cal{E})) =: \kb(\cal{E}),
\]
a stable $\infty$-category (see \cite[Prop.\ 2.7]{BC20}) whose homotopy category recovers $\cal{K}_\mathrm{b}(\cal{E})$. We may consider the full subcategory $\acb(\cal{E}) \subset \kb(\cal{E})$ of objects in the essential image of $\acb(\cal{E})$ under the localisation functor $\gamma:\chb(\cal{E}) \to \kb(\cal{E})$; this is a stable subcategory closed under retracts, and taking the Verdier quotient, we obtain a stable $\infty$-category
\[
	\db(\cal{E}) := \kb(\cal{E})/\acb(\cal{E}) = L_{\mathbf{qis}}(\chb(\cal{E})).
\]
\begin{proposition}\label{stab}
    \begin{enumerate}[label=(\roman*)]
        \item \label{stabadd} \cite[Thm.\ 7.4.9]{BCKW19} For $\cal{A}$ an additive 1-category considered as an exact category with the split exact structure, $\db(\cal{A}) \simeq \kb(\cal{A})$, and the quasi-isomorphisms coincide with the chain homotopy equivalences. Moreover, $\kb(\cal{A})$ is the initial stable $\infty$-category receiving an additive embedding from $\cal{A}$: that is, for a stable $\infty$-category $\cal{C}$, composition with $\cal{A} \subset \kb(\cal{A})$ induces an equivalence of $\infty$-categories
        \[
            \fun^{\mathrm{ex}}(\kb(\cal{A}), \cal{C}) \to \fun^\oplus(\cal{A}, \cal{C}),
        \]
        where $\fun^\oplus$ is the full subcategory of additive functors, i.e.\ for which the natural map $F(x \oplus y) \to F(x) \oplus F(y)$ is an equivalence.
        \item \label{stabex} \cite[Cor.\ 7.4.12]{BCKW19} The embedding $\cal{E} \subset \db(\cal{E})$ enjoys the following universal property: for any stable $\infty$-category $\cal{C}$, precomposition with $\cal{E} \to \db(\cal{E})$ induces an equivalence of $\infty$-categories
        \[
            \fun^{\mathrm{ex}}(\db(\cal{E}), \cal{C}) \to \fun^{\mathrm{ex}}(\cal{E}, \cal{C}),
        \]
        where the latter $\fun^{\mathrm{ex}}$ denotes the full subcategory of additive functors sending conflations in $\cal{E}$ to fibre sequences in $\cal{C}$. In particular, a conflation $x \mono y \epi z$ in $\cal{E}$ is sent to a fibre sequence in $\db(\cal{E})$.
        \item \label{stabk} \cite[Thm.\ I.3.3]{NS18} The Verdier quotient $\kb(\cal{E}) \xto{\pi} \db(\cal{E})$ is an exact functor of stable $\infty$-categories. Given a functor $f:\kb(\cal{E}) \to \cal{C}$, for $\cal{C}$ a stable presentable $\infty$-category, the left Kan extension of $f$ along $\pi$ is given by the formula
        \[
            \pi_!f(\pi(x)) = \varinjlim_{y \in (\acb(\cal{E}))_{/x}}f(\cofib(y \to x) \simeq \varinjlim_{(y \xto{\simeq} x) \in I_x}f(y),
        \]
        where $I_x \subset \kb(\cal{E})_{/x}$ is the full subcategory of the slice category spanned by quasi-isomorphisms over $x$.
    \end{enumerate}
\end{proposition}
\begin{remark}
	For $\cal{A}$ an additive $\infty$-category, there is a universal stable $\infty$-category $\stab(\cal{A})$ receiving a fully faithful additive functor from $\cal{A}$, which we call the \textit{stabilisation}. If $\cal{A}$ is an ordinary additive category, \ref{stabadd} gives an identification $\stab(\cal{A}) \simeq \kb(\cal{A})$.
\end{remark}
\subsection{The underlying $\infty$-category of a model category}\label{mod}
Write $\cat_\Delta$ for the category of small simplicially enriched categories and enriched functors. $\cat_\Delta$ admits a right-proper cofibrantly generated model structure \cite{Ber07}, the \textit{Bergner model structure}, in which the weak equivalences are the Dwyer-Kan equivalences, i.e.\ those simplicial functors $f:\cal{C} \to \cal{D}$ inducing a weak equivalence on mapping spaces and an equivalence $\pi_0(f) :\pi_0(\cal{C}) \to \pi_0(\cal{D})$ on homotopy categories. Write $\sset_{\mathrm{Joyal}}$ for the Joyal model structure on simplicial sets, in which the fibrant objects are the quasi-categories, i.e.\ those simplicial sets having the right lifting property with respect to all inner horn inclusions $\Lambda^n_i \subset \Delta^n$ for each $n \geq 2$. There is a Quillen equivalence \cite[Th.\ 2.2.5.1]{HTT}
\[\begin{tikzcd}
	\mathfrak{C}[-] : \sset_{\mathrm{Joyal}} \ar[r, shift left=.4ex, "\rotatebox{90}{$\vdash$}"', bend left=1ex] & \ar[l, shift left =.4ex, bend left=1ex] \cat_\Delta : N_\Delta,
\end{tikzcd}\]
where the right Quillen adjoint $N_\Delta$ is the homotopy-coherent nerve. Bergner-fibrant simplicially enriched categories are precisely those enriched over Kan complexes, and accordingly for a $\Kan$-enriched category $\cal{C}$, $N_\Delta(\cal{C})$ is a quasi-category.
\begin{example}
	For $\cal{M}$ a simplicial model category \cite[\S9]{Hir03}, the subcategory $\cal{M}^\circ \subset \cal{M}$ of fibrant-cofibrant objects is $\Kan$-enriched; the homotopy-coherent nerve $N_\Delta(\cal{M}^\circ)$ is the \textit{underlying $\infty$-category} of $\cal{M}$. The underlying $\infty$-category of a general model category $\cal{C}$ is given by the (derived) homotopy-coherent nerve of the simplicial localisation at the weak equivalences $N_\Delta(L^H(\cal{C}, w)^{\mathrm{fib}})$. For details, we refer the reader to \cite[\S1.3]{Hin16}.
\end{example}
There is a notion of homotopy (co)limits in a general simplicially enriched category $\cal{C}$ which generalises the corresponding notion for a diagram with values in a simplicial model category: given an enriched functor $F:\cal{I} \to \cal{A}$ with $\cal{A}$ Bergner-fibrant, a pair $(x, \eta)$ consisting of $x \in \cal{A}$ and $\eta:F \Rightarrow \underline{x}$ a cone under $F$ is said to be a \textit{homotopy colimit} for $F$ if for each $y \in \cal{A}$, the maps
\[
	\Map_{\cal{A}}(x, y) \to \underset{i \in I}\lim\ \Map_{\cal{A}}(F(i), y)
\]
exhibit $\Map_{\cal{A}}(x, y)$ as a homotopy limit of the diagram
\[
	\Map_{\cal{A}}(F(-), y):\cal{I} \to \sset
\]
with respect to the Kan-Quillen model structure on $\sset$, and dually for homotopy limits in $\cal{A}$ (see \cite[\S A.3.3]{HTT}). The following theorem translates between this classical notion of homotopy (co)limits, and $\infty$-(co)limits of the diagrams obtained passing to homotopy-coherent nerves.
\begin{theorem}[{\cite[Th.\ 4.2.4.1]{HTT}}]
	Suppose given a simplicial functor $F:\cal{I} \to \cal{C}$ of Bergner-fibrant simplicially enriched categories. Then for an object $c \in \cal{C}$ and a compatible family of maps $\{\eta_i : F(i) \to c\}_{i \in \cal{I}}$, the following are equivalent:
	\begin{enumerate}[label=(\roman*)]
		\item The maps $\eta_i$ exhibit $c$ as a homotopy colimit of $F$;
		\item Write $f:N_\Delta(\cal{I}) \to N_\Delta(\cal{C})$ be the image of $F$ under the homotopy-coherent nerve, and write $\overline{f} : N_\Delta(\cal{I})^\vartriangleright$ be the extension determined by the $\eta_i$. Then $f$ is a colimit diagram in $N_\Delta(\cal{C})$.
	\end{enumerate}
\end{theorem}
\begin{example}
	Consider the subcategory $\Kan \subset \sset$ spanned by the fibrant-cofibrant objects with respect to the Kan-Quillen simplicial model structure on simplicial sets. The homotopy-coherent nerve $N_\Delta(\Kan) =: \cal{S}$ is the $\infty$-category of spaces, (co)limits in which correspond to homotopy (co)limits with respect to this model structure. In particular, pullbacks of spaces can be computed as 1-categorical pullbacks along fibrations.
\end{example}
\begin{remark}
	Given a map of simplicial sets $F:K \to \cal{C}$ with $\cal{C}$ an $\infty$-category, there exists a category $\cal{I}$ and a cofinal map $i:N(\cal{I}) \to K$ by \cite[Prop.\ 4.2.3.14]{HTT}. By \cite[Prop.\ 4.2.4.7]{HTT}, we may then find a Bergner-fibrant simplicially enriched category $\cal{A}$, an enriched functor $f:\cal{I} \to \cal{A}$, a categorical equivalence of simplicial sets $\cal{C} \to N_\Delta(\cal{A})$, and an equivalence in $\fun(N_\Delta(\cal{I}), N_\Delta(\cal{A}))$ between $N_\Delta(f)$ and $jFi$:
	\[\begin{tikzcd}
		N_\Delta(\cal{I}) \ar[r, "i"] \ar[d, "N_\Delta(f)"'{name=Nf}] & |[alias=K]| K \ar[d, "F"{name=F}] \\
		N_\Delta(\cal{A}) & \ar[l, "j"] \cal{C}.
		\ar[Rightarrow,from=F, to=Nf,shorten >=7.5mm,shorten <=7.5mm]
	\end{tikzcd}\]
	We may thus compute the colimit of $F:K \to \cal{C}$ as the homotopy colimit in the strictified diagram $f:\cal{I} \to \cal{A}$, since $i$ is cofinal, and colimits are invariant under categorical equivalence. In practise, the diagrams we consider will be the image under $N_\Delta$ of enriched diagrams.
\end{remark}
\section{Polynomial functors}\label{poly}
In this section we briefly review the Goodwillie-Lurie calculus of functors, and the 1-categorical analogue of polynomial functors on additive categories. Of particular importance is the bridge between these notions: an $n$-polynomial functor with valued in abelian groups on an additive category admits an essentially unique $n$-excisive extension valued in spectra on its stabilisation. We refer the interested reader to \cite[\S6]{HA} as a general reference, and \cite[\S2]{BGMN22}, \cite[\S3]{BM23}, \cite[\S4.2]{CD23a} for further details on $n$-excisive extensions.
\subsection{Polynomial and $n$-excisive functors}\label{n-exc}
\begin{definition}
	Let $\cal{A}, \cal{B}$ be additive $\infty$-categories with $\cal{B}$ weakly idempotent complete. A functor $F:\cal{A} \to \cal{B}$ is \textit{polynomial of degree $\le 0$} if it is a constant functor. For $d \ge 1$, $F$ is \textit{polynomial of degree $\le d$} if for each $x \in \cal{A}$, the functor
	\[
		D_xF:\cal{A} \to \cal{B}, \quad y \mapsto \ker(F(x\oplus y) \to F(y))
	\]
	is polynomial of degree $\le d-1$. Note that the kernel exists in this case, since $F(x \oplus y) \to F(y)$ admits a section induced by the inclusion $y \to x \oplus y$. For general $\cal{B}$ with weak idempotent completion $\cal{B} \to \cal{B}^\flat$, a functor $F:\cal{A} \to \cal{B}$ is polynomial of degree $\le d$ if the composite $\cal{A} \xto{F} \cal{B} \to \cal{B}^\flat$ is. Denote by $\fun^{n\mathrm{-poly}}(\cal{A}, \cal{B})$ for the full subcategory spanned by functors which are polynomial of degree $\le n$, and $\fun_*^{n\mathrm{-poly}}(\cal{A}, \cal{B})$ of such functor which are reduced, i.e.\ preserve zero objects.
\end{definition}
Call a functor $B:\cal{A}^{\op} \times \cal{A}^{\op} \to \Ab$ \textit{symmetric bilinear} if $B$ preserves direct sums in either variable, and there is a natural isomorphism $\sigma_{x, y}:B(x, y) \cong B(y, x)$ for each $x, y \in \cal{A}$.
\begin{lemma}[{\cite[Lem.\ A.10]{Sch21}}]\label{quad}
	Given an additive category $\cal{A}$, a functor $F:\cal{A}^{\op} \to \Ab$ is quadratic if and only if there is a symmetric bilinear functor $B:\cal{A}^{\op} \times \cal{A}^{\op} \to \Ab$, and natural $C_2$-equivariant diagrams of abelian groups
	\[
		B(x, x) \xto{\tau_x} F(x) \xto{\rho_x} B(x, x),
	\]
	such that $\rho_x\tau_x = 1+\sigma_{x, x}$, and for parallel maps $f, g:x \to y$ and $\xi \in F(y)$, we have
	\[
		(f+g)^\bullet(\xi) = f^\bullet(\xi) + g^\bullet(\xi) + \tau_x(B(f, g)(\rho_x(\xi))) \in F(x),
	\]
	where we write $f^\bullet := F(f)$ etc.
\end{lemma}
\begin{remark}
	For quadratic $F$, we call the functor $B=B_F$ above the \textit{polarisation} of $F$.
\end{remark}
\begin{definition}
	Write $\mathcal{P}(n)$ for the power set of $[n] = \{0 < 1 < \dots < n\}$, considered as a poset under inclusion. An \textit{$(n+1)$-cube} in an $\infty$-category $\cal{C}$ is a functor $f:N\mathcal{P}(n) \to \cal{C}$, and such a cube is said to be \textit{(co)cartesian} if it is a limit (colimit) diagram, and \textit{strongly cocartesian} if it is the left Kan extension of its restriction to $N\mathcal{P}_{\le 1}(n)$, for $\mathcal{P}_{\le 1}(n)$ the subposet on subsets of cardinality at most 1 (equivalently, any 2-face of $f$ is a pushout).
\end{definition}
\begin{definition}
Let $\cal{C}, \cal{D}$ be $\infty$-categories admitting finite colimits and limits, respectively. A functor $F:\cal{C} \to \cal{D}$ is said to be \textit{$n$-excisive} if for any strongly cocartesian $(n+1)$-cube $\mathcal{P}(n) \to \cal{C}$, the cube
	\[
		\mathcal{P}(n) \to \cal{C} \xto{F} \cal{D}
	\]
	is cartesian in $\cal{D}$. In the case $\cal{C}$, $\cal{D}$ are pointed, a 2-excisive functor $f:\cal{C} \to \cal{D}$ is \textit{quadratic} if it is reduced, i.e.\ preserves zero objects.
\end{definition}
\begin{example}
	For $\cal{C}, \cal{D}$ stable, a reduced functor $\cal{C} \to \cal{D}$ is 1-excisive precisely when it is exact, i.e.\ preserves finite limits and colimits.
\end{example}
\begin{remark}\label{poly-characterisation}
	For $m \ge 0$, write $\iota_m:\bbDelta^{\op}\mid_{\le m} \subset \bbDelta^{\op}$ for the inclusion of the full subcategory of $\bbDelta^{\op}$ spanned by ordinals $n \le m$. A simplicial object $X:\bbDelta^{\op} \to \cal{C}$ is \textit{$m$-truncated} if $X$ is the left Kan-extension of its restriction $\iota_m^*X:\bbDelta^{\op}\mid_{\le m} \to \cal{C}$, i.e.\ if the counit map $(\iota_m)_!\iota_m^*X \to X$ is an equivalence of functors $\bbDelta^{\op} \to \cal{C}$; $X$ is said to be \textit{finite} if it is $n$-truncated for some $n$, and the colimit of such an $X$ is said to be a \textit{finite geometric realisations}. Then by \cite[Prop.\ 2.15]{BGMN22} a functor $f:\cal{C} \to \cal{D}$ between stable $\infty$-categories is $n$-excisive if and only if the induced functor $\mathrm{Ho}(f):\mathrm{Ho}(\cal{C}) \to \mathrm{Ho}(\cal{D})$ is polynomial of degree $\le n$ as a functor of additive categories, and if $f$ preserves finite geometric realisations.
\end{remark}
\subsection{Extending $n$-polynomial functors on additive categories}\label{extq}
Write $\cat^\mathrm{add}_\infty$ resp.\ $\cat_\infty^\mathrm{st}$ for the large $\infty$-categories of small additive resp.\ stable $\infty$-categories and additive resp.\ exact functors. The forgetful functor $\cat^\mathrm{st}_\infty \to \cat_\infty^\mathrm{add}$ sending a stable category to its underlying additive category admits a left adjoint
\[
	\stab:\cat_\infty^\mathrm{add}\to\cat_\infty^\mathrm{st},
\]
the \textit{stable envelope}, and moreover restriction along the unit $\cal{A} \to \stab(\cal{A})$ induces an equivalence of functor categories
\[
	\fun^{1\mathrm{-exc}}(\stab(\cal{A}), \cal{E}) \simeq \fun^\oplus(\cal{A}, \cal{E}),
\]
for any stable $\infty$-category $\cal{E}$, where we write $\fun^\oplus$ for the full subcategory of additive functors; see \cite[Th.\ 7.4.9]{BCKW19}. Restricted to the full subcategory $\mathrm{Add}_\infty \subset \cat_\infty^\mathrm{add}$ of additive 1-categories and additive functors, there is a natural equivalence $\stab(-) \simeq \kb(-)$, for $\kb(\cal{A})$ the localisation of the category of bounded chain complexes at the chain homotopy equivalences, a stable $\infty$-category by for instance \cite[Prop.\ 2.7]{BC20}. The following polynomial extension of this holds (see \cite[\S2]{BGMN22} or \cite[Prop.\ 4.2.18]{CD23a}), essentially by the characterisation of Remark \ref{poly-characterisation}.
\begin{theorem}[{\cite[Th.\ 2.19]{BGMN22}}]
	For $\cal{A}$ an additive $\infty$-category with stabilisation $\iota_{\cal{A}}:\cal{A} \to \stab(\cal{A})$ and $\cal{C}$ a stable $\infty$-category, there is an equivalence of $\infty$-categories
	\[
		\iota_{\cal{A}}^* : \fun^{n\mathrm{-exc}}(\stab(\cal{A}), \cal{C}) \xto{\simeq} \fun^{n\mathrm{-poly}}(\cal{A}, \cal{C})
	\]
\end{theorem}
\subsection{$n$-excisive approximation}\label{approx}
Suppose $\cal{C}, \cal{D}$ are $\infty$-categories, that $\cal{C}$ admits finite colimits, and that $\cal{D}$ is \textit{differentiable} in the sense of Lurie \cite[Def.\ 6.1.1.6]{HA}\footnote{$\cal{D}$ is differentiable if it admits finite limits and sequential colimits, and if the formation of the two commute, i.e.\ if the functor $\colim:\fun(N(\zz_{\ge 0}), \cal{D}) \to \cal{D})$ is left exact.}. Then the inclusion $\fun^{n\mathrm{-exc}}(\cal{C}, \cal{D}) \subset \fun(\cal{C}, \cal{D})$ admits a left-exact left adjoint $P_n$, the $n$-excisive approximation, by \cite[Th.\ 6.1.1.10]{HA}.
\begin{example}
	For $\cal{C}$ and $\cal{D}$ as above and $F :\cal{C} \to \cal{D}$ a reduced functor, the 1-excisive approximation of $F$ is given by
	\[
		P_1F := \varinjlim_n \left(\Omega^n\circ F \circ \Sigma^n\right) : \cal{C} \to \cal{D}.
	\]
\end{example}
If $\cal{C}, \cal{D}$ are stable $\infty$-categories and $\cal{D}$ is moreover differentiable (by \cite[Ex.\ 6.1.1.7]{HA} this is equivalent to $\cal{D}$ admitting countable coproducts), \cite[Cons.\ 1.1.26]{CD23a} gives a formula for the 2-excisive approximation of a reduced functor $\cal{R}:\cal{C}^{\op} \to \cal{D}$.
\begin{example}
	For each $x \in \cal{C}$, the fibre sequence $\Omega x \to * \to x$ induces a nullcomposite sequence
	\[
		\cal{R}(x) \to \Omega\cal{R}(\Omega x) \to \Omega B_{\cal{R}}(\Omega x, \Omega x) 
	\]
	which in the case $\cal{R}$ is 2-excisive is a fibre sequence by \cite[Lem.\ 1.1.19]{CD23a}. In general, there is a natural map
	\[
		\cal{R}(x) \to \Omega\fib(\cal{R}(\Omega x) \to B_{\cal{R}}(\Omega x, \Omega x)),
	\]
	and this assignment assembles into a functor $T_2\cal{R} : \cal{C}^{\op} \to \cal{D}$, admitting a natural transformation $\cal{R} \to T_2\cal{R}$. Iterating this, we obtain a model for the 2-excisive approximation as the sequential colimit
	\[
		P_2\cal{R}(x) = \colim(\cal{R}(x) \to T_2\cal{R}(x) \to T_2^2\cal{R}(x) \to \dots).
	\]
\end{example}
\end{appendices}
\printbibliography
\end{document}